\documentclass[11pt]{amsart}

\usepackage{graphicx}
\usepackage[english]{babel}
\usepackage[T1]{fontenc}
\usepackage[latin1]{inputenc}
\usepackage{amsfonts}
\usepackage{amssymb}
\usepackage{amsthm}
\usepackage{amsmath}
\usepackage{tikz}
\usepackage{float}
\usepackage{enumerate}
\usepackage{accents}
\usepackage{mathtools}
\usepackage{mathrsfs}
\usepackage{comment}
\usepackage{afterpage}
\usepackage{bbm}
\usepackage{dsfont}
\usepackage{stmaryrd}
\usepackage{hyperref}
\usepackage{mleftright}
\usepackage{subfig}
\usepackage{soul}
\usepackage{tikz-cd} 

\theoremstyle{plain}
\newtheorem{thm}{Theorem}[section]
\newtheorem{lem}[thm]{Lemma}
\newtheorem{prop}[thm]{Proposition}
\newtheorem{cor}[thm]{Corollary}
\newtheorem*{thm*}{Theorem}
\newtheorem*{prop*}{Proposition}
\newtheorem*{cor*}{Corollary}
\newtheorem{thmintro}{Theorem}

\newtheorem{corintro}[thmintro]{Corollary}

\theoremstyle{definition}
\newtheorem{defn}[thm]{Definition}
\newtheorem{ex}[thm]{Example}
\newtheorem{rmk}[thm]{Remark}
\newtheorem*{rmk*}{Remark}

\newtheorem*{quest*}{Question}
\newtheorem*{claim}{Claim}
\newtheorem*{ass}{Standing Assumptions}

\renewcommand{\o}{\circ}
\newcommand{\wt}{\widetilde}
\newcommand{\R}{\mathbb{R}}
\newcommand{\Z}{\mathbb{Z}}
\newcommand{\N}{\mathbb{N}}
\renewcommand{\H}{\mathbb{H}}

\newcommand{\s}{\sigma}
\newcommand{\ra}{\rightarrow}

\newcommand{\cu}{\subseteq}
\newcommand{\g}{\gamma}
\newcommand{\G}{\Gamma}

\newcommand{\mbb}{\mathbb}
\newcommand{\mc}{\mathcal}
\newcommand{\mf}{\mathfrak}
\newcommand{\x}{\times}

\newcommand{\Om}{\Omega}

\newcommand{\Aut}{\mathrm{Aut}}

\newcommand{\acts}{\curvearrowright}

\newcommand{\mscr}{\mathscr}

\newcommand{\Cr}{\mathrm{cr}}

\newcommand{\hull}{\mathrm{Hull}}

\newcommand{\CAT}{{\rm CAT(0)}}
\newcommand{\B}{\mbb{B}}

\title{Cross ratios and cubulations of hyperbolic groups}
\author{Jonas Beyrer and Elia Fioravanti}

\begin{document}

\begin{abstract}
Many geometric structures associated to surface groups can be encoded in terms of invariant cross ratios on their circle at infinity; examples include points of Teichm\"uller space, Hitchin representations and geodesic currents. We add to this picture by studying cubulations of \emph{arbitrary} Gromov hyperbolic groups $G$. Under weak assumptions, we show that the space of cubulations of $G$ naturally injects into the space of $G$--invariant cross ratios on the Gromov boundary $\partial_{\infty}G$.

A consequence of our results is that essential, hyperplane-essential cubulations of hyperbolic groups are length-spectrum rigid, i.e.\ they are fully determined by their length function. This is the optimal length-spectrum rigidity result for cubulations of hyperbolic groups, as we demonstrate with some examples. In the hyperbolic setting, this constitutes a strong improvement on our previous work \cite{BF1}.

Along the way, we describe the relationship between the Roller boundary of a $\CAT$ cube complex, its Gromov boundary and --- in the non-hyperbolic case --- the contracting boundary of Charney and Sultan.

All our results hold for cube complexes with variable edge lengths.
\end{abstract}

\maketitle

\tableofcontents

\section{Introduction.}

Let $G$ be a Gromov hyperbolic group. We denote by $\partial_{\infty}G^{(4)}$ the set of $4$--tuples of pairwise distinct points in the Gromov boundary $\partial_{\infty}G$. A map $\B\colon\partial_{\infty}G^{(4)}\ra\R$ is said to be a \emph{cross ratio} if the following are satisfied:
\begin{itemize}
\item[(i)] $\B(x,y,z,w)=-\B(y,x,z,w)$;
\item[(ii)] $\B(x,y,z,w)=\B(z,w,x,y)$;
\item[(iii)] $\B(x,y,z,w)=\B(x,y,z,t)+\B(x,y,t,w)$;
\item[(iv)] $\B(x,y,z,w)+\B(y,z,x,w)+\B(z,x,y,w)=0$.
\end{itemize}
We say that $\B$ is \emph{invariant} if it is preserved by the diagonal action of $G$ on $\partial_{\infty}G^{(4)}$. Similar notions appear e.g.\ in \cite{Otal-IberoAm,Hamenstaedt-ErgodTh,Labourie-IM,Labourie-IHES}.

Cross ratios provide a unified interpretation of many \emph{geometric structures}, thus proving a valuable tool to study various spaces of representations. 

For instance, when $S$ is a closed hyperbolic surface and $G=\pi_1S$, every point of Teichm\"uller space yields identifications ${\partial_{\infty}G\simeq\partial_{\infty}\H^2\simeq\R\mbb{P}^1}$ and the projective cross ratio on $\R\mbb{P}^1$ can be pulled back to an invariant cross ratio\footnote{To be precise, one has to take the \emph{logarithm} of the \emph{absolute value} of the projective cross ratio if this is to satisfy conditions~(i)--(iv).} on $\partial_{\infty}G$. The latter uniquely determines the original point of Teichm\"uller space \cite{Bonahon}. More generally, the space of all negatively curved Riemannian metrics on $S$ embeds into the space of invariant cross ratios on $\partial_{\infty}G$ \cite{Otal-Ann}.

Another setting where cross ratios play a central role is the study of representations of surface groups into higher-rank Lie groups. A striking result of Labourie identifies the space of Hitchin representations ${\rho\colon G\ra PSL_n\R}$ with a space of H\"older-continuous invariant cross ratios on $\partial_{\infty}G$ \cite{Labourie-IHES}. 

\medskip
In this paper, we consider yet another significant geometric structure that groups can be endowed with. More precisely, we study the \emph{space of cubulations} of a non-elementary hyperbolic group $G$. Our main result is that the space of cubulations of $G$ naturally injects\footnote{Some mild and inevitable assumptions are required; cf.\ Theorem~\ref{hyp CR} below.} into the space of invariant $\Z$--valued cross ratios on $\partial_{\infty}G$ (Theorem~\ref{hyp CR}). An important consequence is that most cubulations of $G$ are \emph{length-spectrum rigid} (Corollary~\ref{hyp MLSR}). 

Recall that a \emph{cubulation} is a proper cocompact action of $G$ on a \emph{$\CAT$ cube complex} $X$. A group is said to be \emph{cubulated} if it admits a cubulation. Cubulated hyperbolic groups are ubiquitous in geometric group theory: they include surface groups, hyperbolic $3$--manifold groups \cite{Bergeron-Wise}, hyperbolic free-by-cyclic groups \cite{Hagen-Wise1,Hagen-Wise2}, hyperbolic Coxeter groups \cite{Niblo-Reeves}, finitely presented small cancellation groups \cite{Wise-GAFA}, random groups at low density \cite{Ollivier-Wise} and many arithmetic lattices in $SO(n,1)$ \cite{Bergeron-Haglund-Wise,Haglund-Wise-Ann}. Cubulated hyperbolic groups are also particularly significant due to recent advances in low-dimensional topology \cite{Haglund-Wise-GAFA,Wise-ICM,Agol-ICM}.

Among cubulations of a group $G$, a subclass is especially relevant for us: that of \emph{essential}, \emph{hyperplane-essential} cubulations\footnote{We refer the reader to \cite{CS,Hagen-Touikan} or Section~\ref{CCC prelims} below for definitions.}. Indeed, due to the extreme flexibility of cube complexes, it is all too easy to perturb any cubulation by adding ``insignificant noise'' (say, a few loose edges around the space). Essential, hyperplane-essential cube complexes are those from which all ``noise'' has been removed. A simple procedure for this removal is provided by Section~3 in \cite{CS} and Theorem~A in \cite{Hagen-Touikan}. 

More precisely, \emph{every} cocompact group action on a $\CAT$ cube complex $X$ can be collapsed to an action on an essential, hyperplane-essential $\CAT$ cube complex $X_{\bullet}$. This procedure will preserve most additional properties of the original action. In particular, the collapsing map ${X\twoheadrightarrow X_{\bullet}}$ has uniformly bounded fibres and it is an equivariant quasi-isometry.

Essential hyperplane-essential cube complexes are the appropriate setting to study cross ratios. Our first result is the following:

\begin{thmintro}\label{hyp CR}
Let $G$ be a non-elementary Gromov hyperbolic group.
\begin{enumerate}
\item Every proper cocompact action of $G$ on an essential $\CAT$ cube complex $X$ canonically determines an invariant cross ratio
\[\Cr_X\colon\partial_{\infty}G^{(4)}\ra\Z.\] 
There exists a co-meagre\footnote{A set is \emph{co-meagre} if its complement is a countable union of sets whose closures have empty interior. By Baire's theorem, co-meagre subsets of $\partial_{\infty}G$ are dense.} subset $\mc{C}\cu\partial_{\infty}G$ such that $\Cr_X$ is continuous at all points of $\mc{C}^{(4)}\cu\partial_{\infty}G^{(4)}$.
\item Let in addition $X$ be hyperplane-essential and consider another action $G\acts Y$ satisfying the same hypotheses. If there exists a co-meagre subset $\mc{D}\cu\partial_{\infty}G$ such that the cross ratios $\Cr_X$ and $\Cr_Y$ coincide on $\mc{D}^{(4)}\cu\partial_{\infty}G^{(4)}$, then $X$ and $Y$ are $G$--equivariantly isomorphic.
\end{enumerate}
\end{thmintro}

In particular, two essential hyperplane-essential cubulations yield the same boundary cross ratio if and only if they are $G$--equivariantly isomorphic. 

Essentiality and hyperplane-essentiality are absolutely crucial to part~(2) of Theorem~\ref{hyp CR}. Examples~\ref{need essential} and~\ref{need hyperplane-essential} show that --- in a very strong sense --- neither of these assumptions can be dropped.

It is not surprising that $\Cr_X$ takes \emph{integer} values in Theorem~\ref{hyp CR}, after all cube complexes are fundamentally discrete objects. Our cross ratio can be regarded as an \emph{exact} discretisation of Paulin's \emph{coarse} cross ratio on Gromov boundaries of arbitrary Gromov hyperbolic spaces \cite{Paulin-cr}.

\medskip
The main ingredient in the proof of Theorem~\ref{hyp CR} is Theorem~\ref{ext Moeb intro} below. We will discuss this result at length later in the introduction, but let us first describe one more of its applications. 

Let us endow our $\CAT$ cube complexes with their $\ell^1$ (aka \emph{combinatorial}) metric $d$ and let us associate to each action $G\acts X$ the function $\ell_X\colon G\ra\N$ given by:
\[\ell_X(g)=\inf_{x\in X} d(x,gx).\]
This is normally known as \emph{length function}, or \emph{marked length spectrum} by analogy with the corresponding notion in the setting of Riemannian manifolds. Theorem~\ref{ext Moeb intro} below will also have the following consequence.

\begin{corintro}\label{hyp MLSR}
Let a Gromov hyperbolic group $G$ act properly and cocompactly on essential, hyperplane-essential $\CAT$ cube complexes $X$ and $Y$. The two actions have the same $\ell^1$ length function if and only if $X$ and $Y$ are $G$--equivariantly isomorphic.
\end{corintro}

The same result is conjectured to hold for actions of $G$ on Hadamard manifolds (Problems~3.1 and~3.7 in \cite{Burns-Katok}). This is known as the \emph{``marked length-spectrum rigidity conjecture''} and it is a notorious open problem. Progress on the conjecture has been remarkably limited, with most results only handling $2$--dimensional spaces \cite{Otal-Ann,Croke,Croke-Fathi-Feldman}, or extremely rigid settings such as symmetric spaces \cite{Hamenstaedt-GAFA,Dal'Bo-Kim}. 

In this perspective, Corollary~\ref{hyp MLSR} is particularly interesting as --- along with our previous work in \cite{BF1} --- it is the first length-spectrum rigidity result to cover such a broad family of non-positively curved spaces. The proof of Corollary~\ref{hyp MLSR} relies on a reduction --- obtained in \cite{BF1} --- to the problem of extending certain boundary maps to isomorphisms of cube complexes. However, we stress that the core argument in the proof of Corollary~\ref{hyp MLSR} lies in the ensuing extension procedure, and this requires completely different techniques from those in \cite{BF1}. See the statement of Theorem~\ref{ext Moeb intro} below and the subsequent discussion for a detailed description.

When $X$ and $Y$ have no free faces (i.e.\ when their $\CAT$ metrics are geodesically complete), Corollary~\ref{hyp MLSR} follows from Theorem~A in \cite{BF1}. Having no free faces, however, is an extremely strong restriction when studying cubulations of \emph{hyperbolic} groups, as most known cubulating procedures will not yield spaces satisfying this requirement. As an example, consider the case when $G$ is the fundamental group of a closed, oriented surface $S$ of genus $\geq 2$. It is well-known that every finite filling collection of closed curves on $S$ gives rise to an essential, hyperplane-essential cubulation of $G$ \cite{Sageev,Sag97,Bergeron-Wise}. On the other hand, most cube complexes resulting from this construction will have dimension $\geq 3$, which forces the existence of free faces\footnote{More generally, given a $\CAT$ cube complex $X$ with no free faces and \emph{any} group $G$ acting on $X$ cocompactly and with virtually cyclic hyperplane-stabilisers, we necessarily have $\dim X\leq 2$. This can be shown by noticing that $\R$ is the only cube complex with no free faces that admits a cocompact action of the group $\Z$.}.

Examples~\ref{need essential} and~\ref{need hyperplane-essential} show that essentiality and hyperplane-essentiality are \emph{necessary} assumptions on the $\CAT$ cube complex $X$ for any form of length-spectrum rigidity to hold. Thus, Corollary~\ref{hyp MLSR} is the \emph{optimal} result of this type for cubulations of hyperbolic groups. In addition, note that any cubulation can be made essential and hyperplane-essential by means of the collapsing procedure of \cite{CS,Hagen-Touikan}. In many settings, however, the two assumptions are automatically satisfied, even without resorting to any collapsing: for instance, this is the case for any cubulation of a hyperbolic $3$--manifold group arising from Sageev's construction applied to quasi-Fuchsian immersed surfaces \cite{Kahn-Markovic}.

\medskip \noindent
{\bf Remark.}
Although we have preferred to state Theorem~\ref{hyp CR} and Corollary~\ref{hyp MLSR} for \emph{cube} complexes, they more generally hold for \emph{cuboid} complexes. In such complexes, edges can have arbitrary (positive) real lengths, so the cross ratio $\Cr_X$ and the length function $\ell_X$ will take arbitrary real values. The price to pay is that, in both results, we can only conclude that actions with the same cross ratio/length function are $G$--equivariantly \emph{isometric}, i.e.\ said isometries will in general not take vertices to vertices. 

All results in this paper equally apply to $\CAT$ cuboid complexes, without requiring any significant changes to proofs --- although of course all cubical isomorphisms need to be replaced with mere isometries. The reader can consult Section~\ref{cuboids sect} for a brief discussion of this. 

\medskip \noindent
{\bf On the proofs of Theorem~\ref{hyp CR} and Corollary~\ref{hyp MLSR}.}
As mentioned, the core result of this paper is an extension procedure for certain partially-defined, cross-ratio preserving boundary maps (Theorem~\ref{ext Moeb intro} below). In order to make things precise, let us introduce some terminology.

The horofunction boundary of the cube complex $(X,d)$ is known as the \emph{Roller boundary} $\partial X$. In our setting, this space is always compact and totally disconnected --- unlike the Gromov/visual\footnote{For the visual boundary of a $\CAT$ space and Gromov boundary of a Gromov hyperbolic space, we refer the reader, respectively, to Chapters~II.8 and~III.H.3 in \cite{BH}. We denote both boundaries by $\partial_{\infty}X$, as these coincide whenever both defined.} boundary $\partial_{\infty}X$. As we observed in \cite{BFI-new,BF1}, the Roller boundary is naturally endowed with a continuous, $\Z$--valued cross ratio:
\[\Cr(x,y,z,w)=\#\mscr{W}(x,z|y,w)-\#\mscr{W}(x,w|y,z).\]
Here, the notation $\mscr{W}(x,z|y,w)$ refers to the collection of hyperplanes of $X$ that separate $x,z\in\partial X$ from $y,w\in\partial X$.

When $X$ is Gromov hyperbolic, the two boundaries $\partial X$ and $\partial_{\infty}X$ share a ``large'' subset. More precisely, a co-meagre subset of $\partial_{\infty}X$ is naturally identified with a subset of $\partial X$ and therefore inherits the cross ratio of $\partial X$. We will denote this common subset by $\partial_{\rm nt}X$, as it coincides with the collection of \emph{non-terminating ultrafilters} introduced in \cite{Nevo-Sageev}. Equivalently, we can describe $\partial_{\rm nt}X\cu\partial_{\infty}X$ as the subset of points that do not lie in the Gromov boundary of any hyperplane of $X$ (Lemma~\ref{nt vs hyp}). 

The following is the crucial ingredient in the proofs of part~(2) of Theorem~\ref{hyp CR} and of Corollary~\ref{hyp MLSR}.

\begin{thmintro}\label{ext Moeb intro}
Let a non-elementary Gromov hyperbolic group $G$ act properly and cocompactly on essential, hyperplane-essential $\CAT$ cube complexes $X$ and $Y$. Let $f\colon\partial_{\infty}X\ra\partial_{\infty}Y$ be the unique $G$--equivariant homeomorphism. Suppose that there exists a nonempty, $G$--invariant subset $\Omega\cu\partial_{\rm nt}X$ such that $f(\Omega)\cu\partial_{\rm nt}Y$ and such that cross ratios of elements of $\Omega^4$ are preserved by $f$. Then, there exists a unique $G$--equivariant isomorphism $F\colon X\ra Y$ extending $f$. 
\end{thmintro}

It is possible that Theorem~\ref{ext Moeb intro} will find further application in the proof of rigidity results for certain classes of cubulated hyperbolic groups. Indeed, since $\ell^1$ metrics on cube complexes fall in the setting of Section~5.1 of \cite{Haissinsky-Bourbaki}, cross-ratio preserving boundary maps should arise naturally from commensurations or quasi-isometries between cubulated hyperbolic groups with suitable properties. At present, however, a major obstruction to pursuing approaches of this kind is that not much is known on conformal dimension and Loewner property for boundaries of cubulated hyperbolic groups. A notable exception are Bourdon groups \cite{Bourdon-GAFA}; also see \cite{Bourdon-Kleiner1,Bourdon-Kleiner2}.

\medskip
It is interesting to compare Theorem~\ref{ext Moeb intro} with two old results obtained in different settings. The first is Paulin's classical theorem that homeomorphisms of Gromov boundaries arise from \emph{quasi-isometries} if and only if they \emph{almost preserve} the boundary cross ratio (Theorem~1.2 in \cite{Paulin-cr}). Paulin's techniques are of no help in our context, as we want our extensions to be genuine \emph{isometries}.

A more fitting comparison is with Proposition~2.4.7 in \cite{Bourdon-GAFA}, whose statement strikingly resembles that of Theorem~\ref{ext Moeb intro}. This is not a coincidence, as Bourdon's buildings $I_{p,q}$ --- and, more generally, all Fuchsian buildings without triangular chambers --- can be given a natural structure of $\CAT$ square complex; see e.g.\ Section~2.2 in \cite{Genevois-Martin}. The hyperplanes of the $\CAT$ square complex correspond to Bourdon's \emph{arbre-murs}, along with a choice of a preferred side. With these observations in mind, Bourdon's \emph{birapport combinatoire} on $\partial_{\infty}I_{p,q}\simeq\partial_{\infty}\G_{p,q}$ becomes a special case of part~(1) of our Theorem~\ref{hyp CR}. Proposition~2.4.7 in \cite{Bourdon-GAFA} and part of Theorem~1.5 in \cite{Xie} become a special case of Theorem~\ref{ext Moeb intro} above.

It is important to remark that, unlike the $2$--dimensional setting of Fuchsian buildings, the cube complexes in Theorem~\ref{ext Moeb intro} can have arbitrarily high dimension. This will seriously complicate proofs due to a strictly $3$--dimensional phenomenon which we now describe. 

Given points $x,y,z,w\in\partial X$, the three sets of hyperplanes $\mscr{W}(x,y|z,w)$, $\mscr{W}(x,z|y,w)$ and $\mscr{W}(x,w|y,z)$ are pairwise transverse. If $\dim X\leq 2$, one of these sets must be empty and their three cardinalities can be deduced from their respective differences, i.e.\ $\Cr(x,y,z,w)$, $\Cr(y,z,x,w)$ and $\Cr(z,x,y,w)$. On the other hand, when $\dim X\geq 3$, it may be impossible to recover all three cardinalities just from cross ratios of $4$--tuples involving only the points $x,y,z,w$ (see e.g.\ Figure~1 in \cite{BF1} and the related discussion). 

In order to resolve part of this issue, we will be led to consider \emph{trustworthy} $4$--tuples $(x,y,z,w)\in(\partial_{\rm nt}X)^4$, i.e.\ those $4$--tuples for which one of the three sets $\mscr{W}(x,y|z,w)$, $\mscr{W}(x,z|y,w)$ and $\mscr{W}(x,w|y,z)$ is empty. A key point will be that, even in boundaries of high-dimensional cube complexes, it is always possible to find several trustworthy $4$--tuples (Lemma~\ref{trust lemma}).

\medskip
We now briefly sketch the proof of Theorem~\ref{ext Moeb intro}, denoting by $\mscr{W}(X)$ and $\mscr{H}(X)$, respectively, the collections of all hyperplanes and all halfspaces of the cube complex $X$. The rough idea is that it should be possible to reconstruct the structure of the \emph{halfspace pocset} $(\mscr{H}(X),\cu,\ast)$ simply by looking at the Gromov boundary $\partial_{\infty}X$ and the cross ratio (where defined). 

Overlooking various complications, there are two (bipartite) steps.
\begin{enumerate}
\item[(Ia)] For every $\mf{h}\in\mscr{H}(X)$, we have $\partial_{\infty}\mf{h}\setminus\partial_{\infty}\mf{h}^*\neq\emptyset$. 
\item[(Ib)] Given $\mf{h},\mf{k}\in\mscr{H}(X)$, we have $\mf{h}\cu\mf{k}$ if and only if $\partial_{\infty}\mf{h}\cu\partial_{\infty}\mf{k}$. \footnote{This is not true in general, but it is how one should think about things. It only fails when $\partial_{\infty}\mf{h}=\partial_{\infty}\mf{k}$ and $\mf{k}\subsetneq\mf{h}$, in which case $\mf{k}$ and $\mf{h}$ are at finite Hausdorff distance anyway.}
\item[(IIa)] For every $\mf{w}\in\mscr{W}(X)$, there exists $\mf{w}'\in\mscr{W}(Y)$ with $f(\partial_{\infty}\mf{w})=\partial_{\infty}\mf{w}'$.
\item[(IIb)] For every $\mf{h}\in\mscr{H}(X)$, there exists $\mf{h}'\in\mscr{H}(Y)$ with $f(\partial_{\infty}\mf{h})=\partial_{\infty}\mf{h}'$.
\end{enumerate}
The boundary homeomorphism $f$ then induces a $G$--equivariant bijection $f_*\colon\mscr{H}(X)\ra\mscr{H}(Y)$ by Steps~(Ia) and~(IIb). Step~(Ib) shows that $f_*$ preserves inclusion relations and, by general theory of $\CAT$ cube complexes, $f_*$ must be induced by a $G$--equivariant isomorphism $F\colon X\ra Y$.

Steps~(Ia) and~(Ib) are the key points where, respectively, \emph{essentiality} and \emph{hyperplane-essentiality} come into play. Example~\ref{need essential} shows that, if $X$ is not essential, some halfspaces may be invisible in $\partial_{\infty}X$, i.e.\ Step~(Ia) fails. Without hyperplane-essentiality, instead, $\partial_{\infty}X$ may not be able to tell whether two halfspaces are nested or not. This is exactly the problem in Example~\ref{need hyperplane-essential}, where \emph{transverse} halfspaces $\mf{h}$ and $\mf{k}$ have $\partial_{\infty}\mf{h}=\partial_{\infty}\mf{k}$.

Regarding Step~(IIa), it is not hard to use the cross ratio to characterise which pairs of points $\xi,\eta\in\partial_{\infty}X$ lie in the Gromov boundary of a common hyperplane (Proposition~\ref{pts lie in hyp}). This property is then preserved by $f$, which is all one needs if no two hyperplanes share asymptotic directions (e.g.\ in Fuchsian buildings). In general, we will require more elaborate arguments (Lemma~\ref{hyp pairs are preserved} and Proposition~\ref{maximal to maximal}) based on the fact that $\partial_{\infty}G$ cannot be covered by limit sets of infinite-index quasi-convex subgroups.

Finally, there is a deceiving similarity between the statements of Steps~(IIa) and~(IIb), but the proof of the latter is significantly more involved. Given $\mf{w}\in\mscr{W}(X)$ bounding $\mf{h}\in\mscr{H}(X)$, the set $\partial_{\infty}\mf{h}\setminus\partial_{\infty}\mf{w}$ is a union of connected components of $\partial_{\infty}X\setminus\partial_{\infty}\mf{w}$. However, ${\partial_{\infty}X\setminus\partial_{\infty}\mf{w}}$ will in general have many more components than there are halfspaces bounded by $\mf{w}$. 

The case to keep in mind is when $G=\pi_1S$, for a closed oriented surface $S$, and the hyperplane-stabiliser $G_{\mf{w}}<G$ is the fundamental group of a subsurface of $S$ with at least $3$ boundary components. Not all $G_{\mf{w}}$--invariant partitions of the set of connected components of ${\partial_{\infty}X\setminus\partial_{\infty}\mf{w}}$ arise from a halfspace of $X$. Thus, one cannot recover $\partial_{\infty}\mf{h}$ from the knowledge of $\partial_{\infty}\mf{w}$ purely through topological and dynamical arguments. 

We will instead rely again on the cross ratio in order to circumvent these issues. Step~(IIb) will finally be completed in Theorem~\ref{all WP}.

\medskip \noindent
{\bf On the relationship between $\partial X$ and $\partial_{\infty}X$.} 
We still have not discussed the first half of Theorem~\ref{hyp CR}, which is mostly based on transferring the cross ratio from $\partial X$ to $\partial_{\infty}X$. As the required techniques are quite similar, we do not assume hyperbolicity of the $\CAT$ cube complex $X$ and we more generally describe the relationship between the Roller boundary $\partial X$ and the \emph{contracting boundary} $\partial_cX$. The latter was introduced in \cite{Charney-Sultan}.

Fixing a basepoint $p\in X$, every point of $\partial X$ is represented by a combinatorial ray based at $p$. We denote by $\partial_{\rm cu}X\cu\partial X$ the subset of points that are represented by \emph{contracting} combinatorial rays. We endow $\partial_{\rm cu}X$ with the restriction of the (totally disconnected) topology of $\partial X$. We moreover denote by $\partial_c^{\rm vis}X$ the space obtained by endowing the contracting boundary $\partial_cX$ with the restriction of the visual topology on the visual boundary $\partial_{\infty}X$. 

Responding to a suggestion in the introduction of \cite{Charney-Sultan}, we prove:

\begin{thmintro}\label{R vs c intro}
Let $X$ be a uniformly locally finite $\CAT$ cube complex.
\begin{enumerate}
\item There exists a natural continuous surjection $\Phi\colon\partial_{\rm cu}X\longrightarrow\partial_c^{\rm vis}X$ with finite fibres. Collapsing its fibres, $\Phi$ descends to a homeomorphism.
\item If $X$ is hyperbolic, we have $\partial_{\rm cu}X=\partial X$ and $\partial_c^{\rm vis}X=\partial_{\infty}X$.
\end{enumerate}
\end{thmintro}

The reader will find additional details on the map $\Phi$ in Section~\ref{d_c vs d_R}, especially in Remark~\ref{components are finite} and Theorem~\ref{properties of Phi}. We stress that --- whenever flats are present --- it is not possible to represent the \emph{entire} visual boundary $\partial_{\infty}X$ as a quotient of a subset of the Roller boundary.

Now, part~(1) of Theorem~\ref{hyp CR} is obtained by considering a canonical section to the map $\Phi$. The latter is built through a new construction of barycentres for bounded cube complexes, which we describe in Section~\ref{median barycentres}.

Along with our previous work in \cite{BF1}, Theorem~\ref{R vs c intro} also allows us to extend Theorem~\ref{hyp CR} to the context of non-hyperbolic groups acting on $\CAT$ cube complexes with \emph{no free faces}. 

Recall that the \emph{Morse boundary} of an arbitrary finitely generated group $G$ was introduced in \cite{Cordes}. In accordance with \cite{Cashen-Mackay}, we prefer to refer to it as the \emph{contracting boundary}\footnote{This is justified by the fact that a quasi-geodesic is \emph{Morse} if and only if it is \emph{sublinearly contracting} \cite{ACGH}. However, we acknowledge that the word ``contracting'' is normally taken to mean ``strongly contracting'' (as we also do in Section~\ref{Morse section}) and these quasi-geodesics would not provide a satisfactory notion of boundary for a general group $G$.} of $G$ (denoted $\partial_cG$), as this simplifies notation and terminology (the topology of $\partial_cG$ will not be relevant to us).

\begin{corintro}\label{non-hyp CR}
Let $G$ be a finitely generated, non-virtually-cyclic group.
\begin{enumerate}
\item Every proper cocompact action of $G$ on an irreducible $\CAT$ cube complex with no free faces $X$ canonically determines an invariant cross ratio:
\[\Cr_X\colon\partial_cG^{(4)}\ra\Z.\] 
\item Given another action $G\acts Y$ as above, the cross ratios $\Cr_X$ and $\Cr_Y$ coincide if and only if $X$ and $Y$ are $G$--equivariantly isomorphic.
\end{enumerate}
\end{corintro}

It is worth pointing out that, under the hypotheses of Corollary~\ref{non-hyp CR}, the contracting boundary $\partial_cG$ is always nonempty. When $G$ is hyperbolic, $\partial_cG$ is naturally identified with the Gromov boundary $\partial_{\infty}G$.

The cross ratio provided by Corollary~\ref{non-hyp CR} is again continuous\footnote{Both endowing $\partial_cG$ with the visual topology and with the topologies of \cite{Cordes,Cashen-Mackay}.} at a ``large'' subset of $\partial_cG^{(4)}$, but it does not make sense to speak of \emph{meagre} subsets in this context. Indeed, the entire contracting boundary is often itself meagre, even in the Cashen--Mackay topology. In fact, $\partial_cG$ is likely to be a Baire space if and only if $G$ is hyperbolic. The latter observation follows from Theorem~7.6 in \cite{Cashen-Mackay} when $G$ is a toral relatively hyperbolic group. We thank Chris Cashen for pointing this out to us.

\medskip
We conclude the introduction with the following question, which is naturally brought to mind by Theorem~\ref{hyp CR}.

\begin{quest*}
Let $G$ be a non-elementary hyperbolic group. Is it possible to provide a set of conditions completely characterising which invariant, $\Z$--valued cross ratios on $\partial_{\infty}G$ arise from cubulations of $G$?
\end{quest*}

Theorem~1.1 in \cite{Labourie-IHES} is a result of this type in the context of Hitchin representations. A complete answer to the above question might provide a new procedure to cubulate groups. 

In this regard, note that $\partial_{\infty}G$ is endowed with a continuous, invariant, \emph{$\R$--valued} cross ratio whenever $G$ acts properly and cocompactly on a ${\rm CAT}(-1)$ space. So it would also be interesting to determine under what circumstances an invariant $\R$--valued cross ratio can be discretised to an invariant $\Z$--valued cross ratio. Of course, one should be very careful when making speculations, as, for instance, uniform lattices in $SU(n,1)$ and $Sp(n,1)$ are not cubulable \cite{Delzant-Py,Niblo-Reeves-T}.

\medskip
{\bf Acknowledgements.} We are indebted to Anthony Genevois for bringing Bourdon's paper \cite{Bourdon-GAFA} to our attention and for his many interesting comments on an earlier version of this preprint. We are grateful to Mark Hagen and Alessandro Sisto for contributing some of the ideas in Sections~\ref{CAT vs l^1 geod} and~\ref{median barycentres}, respectively. We also thank Pierre-Emmanuel Caprace, Christopher Cashen, Matt Cordes, Ruth Charney, Cornelia Dru\c{t}u, Talia Fern\'os, Ilya Gekhtman, Tobias Hartnick, John Mackay, Gabriel Pallier, Beatrice Poz\-zet\-ti, Michah Sageev, Viktor Schroeder and Ric Wade for helpful conversations. Finally, we thank Maria Martini for her hedgehog-drawing skills.

We thank the organisers of the following conferences, where part of this work was carried out: \emph{Third GEAR (Junior) Retreat}, Stanford, 2017; \emph{Moduli Spaces}, Ventotene, 2017; \emph{Young Geometric Group Theory VII}, Les Diablerets, 2018; \emph{Topological and Homological Methods in Group Theory}, Bielefeld, 2018; \emph{3--Manifolds and Geometric Group Theory}, Luminy, 2018; \emph{Representation varieties and geometric structures in low dimensions}, Warwick, 2018. EF also thanks M.\ Incerti-Medici, Viktor Schroeder and Anna Wienhard for his visits at UZH and Universit\"at Heidelberg, respectively. 

JB was supported by the Swiss National Science Foundation under Grant 200020$\setminus$175567. EF was supported by the Clarendon Fund and the Merton Moussouris Scholarship.

\addtocontents{toc}{\protect\setcounter{tocdepth}{1}}
\section{Preliminaries.}

\subsection{$\CAT$ cube complexes.}\label{CCC prelims}

We will assume a certain familiarity with basic properties of cube complexes. The reader can consult for instance \cite{Sageev-notes} and the first sections of \cite{Chatterji-Niblo,CS,Nevo-Sageev,CFI} for an introduction to the subject. This subsection is mainly meant to fix notation and recall a few facts that we shall rely on.

Let $X$ be a simply connected cube complex satisfying Gromov's no-$\triangle$-condition; see 4.2.C in \cite{Gromov-HypGps} and Chapter~II.5 in \cite{BH}. The Euclidean metrics on its cubes fit together to yield a $\CAT$ metric on $X$. We can also endow each cube $[0,1]^k\cu X$ with the restriction of the $\ell^1$ metric of $\R^k$ and consider the induced path metric $d(-,-)$. We refer to $d$ as the \emph{combinatorial metric} (or \emph{$\ell^1$ metric}). In finite dimensional cube complexes, the $\CAT$ and combinatorial metrics are bi-Lipschitz equivalent and complete.

The metric space $(X,d)$ is a \emph{median space}. This means that, given any three points $p_1,p_2,p_3\in X$, there exists a unique point ${m=m(p_1,p_2,p_3)\in X}$ such that $d(p_i,p_j)=d(p_i,m)+d(m,p_j)$ for all $1\leq i<j\leq 3$. We refer to $m(p_1,p_2,p_3)$ as the \emph{median} of $p_1$, $p_2$ and $p_3$; if the three points are vertices of $X$, so is $m(p_1,p_2,p_3)$. The map $m\colon X^3\ra X$ endows $X$ (and its $0$--skeleton) with a structure of \emph{median algebra}. We refer the reader to \cite{Roller,CDH,Fio1} for a definition of the latter and more on median geometry.

We will use the more familiar and concise expression ``$\CAT$ cube complex'' with the meaning of ``simply connected cube complex satisfying Gromov's no-$\triangle$-con\-di\-tion''. However, unless specified otherwise, all our cube complexes $X$ will be endowed with the combinatorial metric, all points of $X$ will be implicitly assumed to be vertices and all geodesics will be combinatorial geodesics contained in the $1$--skeleton. In some situations, especially in Sections~\ref{CAT vs l^1 geod} and~\ref{d_c vs d_R}, we will also need to consider geodesics with respect to the $\CAT$ metric. In this case, we will use the terminology ``combinatorial geodesic/segment/ray/line'' and ``$\CAT$ geodesic/\allowbreak segment/\allowbreak ray/\allowbreak line''.

We denote by $X'$ the \emph{cubical subdivision} of $X$. This is the $\CAT$ cube complex obtained by adding a vertex $v(c)$ at the centre of each cube $c\cu X$; we then join the vertices $v(c)$ and $v(c')$ by an edge if $c$ is a codimension-one face of $c'$ or vice versa. Each $k$--cube of $X$ gives rise to $2^k$ $k$--cubes of $X'$.

We write $\mscr{W}(X)$ and $\mscr{H}(X)$, respectively, for the sets of hyperplanes and halfspaces of $X$. Given a halfspace $\mf{h}\in\mscr{H}(X)$, we denote its complement by $\mf{h}^*$. Endowing $\mscr{H}(X)$ with the order relation given by inclusions, the involution $*$ is order-reversing. The triple $(\mscr{H}(X),\cu,*)$ is thus a \emph{po{\bf c}set}, in the sense of \cite{Sageev-notes}. 

We say that two distinct hyperplanes are \emph{transverse} if they cross. Similarly, we say that two halfspaces --- or a halfspace and hyperplane --- are transverse if the corresponding hyperplanes are. Halfspaces $\mf{h}$ and $\mf{k}$ are transverse if and only if all four intersections $\mf{h}\cap\mf{k}$, $\mf{h}^*\cap\mf{k}$, $\mf{h}\cap\mf{k}^*$ and $\mf{h}^*\cap\mf{k}^*$ are nonempty. We say that subsets $A,B\cu\mscr{W}(X)$ are transverse if every element of $A$ is transverse to every element of $B$. 

Every hyperplane $\mf{w}$ can itself be regarded as a $\CAT$ cube complex; its cells are precisely the intersections $\mf{w}\cap c$, where $c\cu X$ is a cube. The set of hyperplanes of the cube complex $\mf{w}$ is naturally identified with the set of hyperplanes of $X$ that are transverse to $\mf{w}$. We thus denote by $\mscr{W}(\mf{w})$ this subset of $\mscr{W}(X)$. We also denote by $C(\mf{w})$ the \emph{carrier} of $\mf{w}$, i.e.\ its neighbourhood of radius $\frac{1}{2}$ in $X$.

Three hyperplanes $\mf{w}_1$, $\mf{w}_2$ and $\mf{w}_3$ form a \emph{facing triple} if we can choose pairwise disjoint sides $\mf{h}_1$, $\mf{h}_2$ and $\mf{h}_3$; the three halfspaces are then also said to form a \emph{facing triple}. Halfspaces $\mf{h}$ and $\mf{k}$ are \emph{nested} if either $\mf{h}\cu\mf{k}$ or $\mf{k}\cu\mf{h}$.

Given a vertex $p\in X$, we denote by $\s_p\cu\mscr{H}(X)$ the set of all halfspaces containing $p$. It satisfies the following properties:
\begin{enumerate}
\item any two halfspaces in $\s_p$ intersect non-trivially;
\item for every hyperplane $\mf{w}\in\mscr{W}(X)$, a side of $\mf{w}$ lies in $\s_p$;
\item every descending chain of halfspaces in $\s_p$ is finite. 
\end{enumerate}
Subsets $\s\cu\mscr{H}(X)$ satisfying $(1)$--$(3)$ are known as \emph{DCC ultrafilters}. If a set $\s\cu\mscr{H}(X)$ only satisfies $(1)$ and $(2)$, we refer to it simply as an \emph{ultrafilter}. 

Let $\iota\colon X\ra 2^{\mscr{H}(X)}$ denote the map that takes each vertex $p$ to the set $\s_p$. Its image $\iota (X)$ coincides with the collection of all DCC ultrafilters. Endowing $2^{\mscr{H}(X)}$ with the product topology, we can consider the closure $\overline{\iota(X)}$, which happens to coincide with the set of all ultrafilters. Equipped with the subspace topology, this is a totally disconnected, compact, Hausdorff space known as the \emph{Roller compactification} of $X$ \cite{Nevo-Sageev}; we denote it by $\overline X$. 

The \emph{Roller boundary} $\partial X$ is defined as the difference $\overline X\setminus X$. The inclusion $\iota\colon X\ra\overline X$ is always continuous\footnote{Note that $\iota$ is only defined on the $0$--skeleton, which has the discrete topology.}. If, moreover, $X$ is locally finite, then $\iota$ is a topological embedding, $X$ is open in $\overline X$ and $\partial X$ is compact. Even though elements of $\partial X$ are technically just sets of halfspaces, we will rather think of them as points at infinity. In analogy with vertices of $X$, we will then write $x\in\partial X$ and reserve the notation $\s_x$ for the ultrafilter representing $x$.

According to an unpublished result of U.\ Bader and D.\ Guralnik, the identity of $X$ extends to a homeomorphism between $\partial X$ and the horofunction boundary of $(X,d)$; also see \cite{Caprace-Lecureux,Fernos-Lecureux-Matheus}. However, the characterisation of $\overline X$ in terms of ultrafilters additionally provides a natural structure of median algebra on $\overline X$, corresponding to the map
\[m(\s_x,\s_y,\s_z)=(\s_x\cap\s_y)\cup(\s_y\cap\s_z)\cup(\s_z\cap\s_x).\]
Under the identification of $p\in X$ and $\s_p\cu\mscr{H}(X)$, this map $m\colon{\overline{X}}^3\ra\overline X$ restricts to the usual median-algebra structure on $X$. Given $x,y\in\overline X$, the \emph{interval} between $x$ and $y$ is the set $I(x,y)=\{z\in\overline X\mid m(x,y,z)=z\}$. If $x,y,z\in\overline X$, the median $m(x,y,z)$ is the only point of $\overline X$ that lies in all three intervals $I(x,y)$, $I(y,z)$ and $I(z,x)$. Observe that $I(p,q)\cu X$ if $p,q\in X$.

In some instances, we will also have to consider the visual boundary of $X$ associated to the $\CAT$ metric. To avoid confusion with the Roller boundary $\partial X$, we will denote the visual boundary by $\partial_{\infty}X$; note that $\partial_{\infty}X$ is the horofunction boundary of $X$ with respect to the $\CAT$ metric. When $X$ is Gromov hyperbolic, $\partial_{\infty}X$ is also naturally identified with the Gromov boundary of $X$, for which we will adopt the same notation.

Given $p\in X$ and $\mf{h}\in\mscr{H}(X)$, we have $p\in\mf{h}$ if and only if $\mf{h}\in\s_p$. By analogy, we say that a point $x\in\overline X$ lies in a halfspace $\mf{h}\in\mscr{H}(X)$ (written $x\in\mf{h}$), if the halfspace $\mf{h}$ is an element of the ultrafilter $\s_x$. This should be regarded as a way of extending halfspaces into the boundary, yielding a partition $\overline X=\mf{h}\sqcup\mf{h}^*$ for every $\mf{h}\in\mscr{H}(X)$. 

Given subsets $A,B\cu\overline X$, we adopt the notation:
\begin{align*}
&\mscr{H}(A|B)=\{\mf{h}\in\mscr{H}(X)\mid B\cu\mf{h},~A\cu\mf{h}^*\}, \\
&\mscr{W}(A|B)=\{\mf{w}\in\mscr{W}(X)\mid\text{one side of }\mf{w}\text{ lies in } \mscr{H}(A|B)\}.
\end{align*}
If $\mf{w}\in\mscr{W}(A|B)$, we say that $\mf{w}$ \emph{separates} $A$ and $B$. We denote by $\mscr{W}(A)$ the set of all hyperplanes separating two points of $A$. To avoid possible ambiguities, we adopt the convention that hyperplanes $\mf{w}$ are not contained in either of their sides; in particular, $\mf{w}\not\in\mscr{W}(\mf{w}|A)$ for every $A\cu\overline X$. 

\begin{lem}\label{three pwt}
Given points $x,y,z,w\in\overline X$, the sets $\mscr{W}(x,y|z,w)$, $\mscr{W}(x,z|y,w)$ and $\mscr{W}(x,w|y,z)$ are pairwise transverse.
\end{lem}
\begin{proof}
Consider $\mf{h}\in\mscr{H}(x,y|z,w)$ and $\mf{k}\in\mscr{H}(x,z|y,w)$. Since we have $x\in\mf{h}^*\cap\mf{k}^*$, $y\in\mf{h}^*\cap\mf{k}$, $z\in\mf{h}\cap\mf{k}^*$ and $w\in\mf{h}\cap\mf{k}$, we conclude that $\mf{h}$ and $\mf{k}$ are transverse. Hence $\mscr{W}(x,y|z,w)$ and $\mscr{W}(x,z|y,w)$ are transverse; the same argument shows that they are also transverse to $\mscr{W}(x,w|y,z)$.
\end{proof}

We will generally conflate all geodesics (and quasi-geodesics) with their images in $X$. Every geodesic $\g\cu X$ can be viewed as a collection of edges; distinct edges $e,e'\cu\g$ must cross distinct hyperplanes. We write $\mscr{W}(\g)$ for the collection of hyperplanes crossed by (the edges of) $\g$. If two geodesics $\g$ and $\g'$ share an endpoint $p\in X$, their union $\g\cup\g'$ is again a geodesic if and only if $\mscr{W}(\g)\cap\mscr{W}(\g')=\emptyset$.

Given a ray $r\cu X$, we denote by $r(0)$ its origin and, given $n\in\N$, by $r(n)$ the only point of $r$ with $d(r(0),r(n))=n$. Given a hyperplane $\mf{w}\in\mscr{W}(X)$, there is exactly one side of $\mf{w}$ that has unbounded intersection with $r$. The collection of all these halfspaces is an ultrafilter on $\mscr{H}(X)$ and it therefore determines a point $r(+\infty)\in\partial X$.

Fixing a basepoint $p\in X$, every point of $\partial X$ is of the form $r(+\infty)$ for a ray $r$ with $r(0)=p$. We obtain a bijection between points of $\partial X$ and rays based at $p$, where we need to identify the rays $r_1$ and $r_2$ if $\mscr{W}(r_1)=\mscr{W}(r_2)$. See Proposition~A.2 in \cite{Genevois} for details.

Given two vertices $p,q\in X$, we have $d(p,q)=\#\mscr{W}(p|q)$. By analogy, we can define $d(x,y)=\#\mscr{W}(x|y)$ for all points $x,y\in\overline X$. The resulting function ${d\colon\overline X\x\overline X\ra\N\cup\{+\infty\}}$ satisfies all axioms of a metric, except that it can indeed take the value $+\infty$. We write $x\sim y$ if $x$ and $y$ satisfy $d(x,y)<+\infty$. This is an equivalence relation on $\overline X$; we refer to its equivalence classes as \emph{components}.

Given $x\in\overline X$, we denote by $Z(x)$ the only component of $\overline X$ that contains the point $x$. When $x\in X$, we have $Z(x)=X$. For every component $Z\cu\overline X$, the pair $(Z,d)$ is a metric space. Joining points of $Z$ by an edge whenever they are at distance $1$ and adding $k$--cubes whenever we see their $1$--skeleta, we can give $(Z,d)$ a structure of $\CAT$ cube complex.

We obtain here a couple of simple results which will be needed later on.

\begin{lem}\label{components vs separation}
Suppose that $\dim X<+\infty$. Let $\mf{h}$ and $\mf{k}$ be disjoint halfspaces with $\mf{h}\neq\mf{k}^*$. Suppose that there exist points $x\in\mf{h}\cap\partial X$ and $y\in\mf{k}\cap\partial X$ with $x\sim y$. There exists an infinite chain $\mf{j}_0\supsetneq\mf{j}_1\supsetneq ...$ of halfspaces of $X$ transverse to both $\mf{h}$ and $\mf{k}$.
\end{lem}
\begin{proof}
Since $\mf{h}\neq\mf{k}^*$, there exists a point $p\in X$ lying in $\mf{h}^*\cap\mf{k}^*$; see Figure~\ref{lemma picture}. The set $\mscr{W}(p|x,y)$ is infinite, as $\mscr{W}(p|x)$ is infinite and $\mscr{W}(x|y)$ is finite. Since $\s_p$ is a DCC ultrafilter, the sets $\mscr{W}(p|\mf{h})$ and $\mscr{W}(p|\mf{k})$ are finite and the set $A=\mscr{H}(p|x,y)\setminus(\mscr{H}(p|\mf{h})\cup\mscr{H}(p|\mf{k}))$ is infinite. Any halfspace in $A$ is transverse to both $\mf{h}$ and $\mf{k}$. Any two elements of $\mscr{H}(p|x,y)$ are either transverse or nested, and any subset of pairwise transverse halfspaces has cardinality at most $\dim X$. The required chain is thus obtained by applying Ramsey's theorem to $A$.
\end{proof}

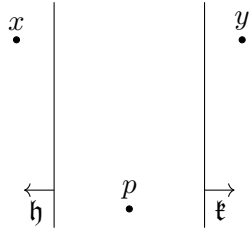
\begin{figure} 
\begin{tikzpicture}
\draw [fill] (-1,-0.5) -- (-1,2.5);
\draw [fill] (1,-0.5) -- (1,2.5);
\draw [->] (-1,0) -- (-1.4,0);
\draw [->] (1,0) -- (1.4,0);
\node [below left] at (-1,0) {$\mf{h}$};
\node [below right] at (1,0) {$\mf{k}$};
\draw[fill] (0,-0.25) circle [radius=0.04cm];
\draw[fill] (-1.5,2) circle [radius=0.04cm];
\draw[fill] (1.5,2) circle [radius=0.04cm];
\node [above] at (-1.5,2) {$x$};
\node [above] at (1.5,2) {$y$};
\node [above] at (0,-0.25) {$p$};
\end{tikzpicture}
\caption{The setup in the proof of Lemma~\ref{components vs separation}.}
\label{lemma picture} 
\end{figure}

\begin{lem}\label{rays to nearby points}
Let $r\cu X$ be a ray and set $x=r(+\infty)$. Given a point $x'\in\partial X$ with $x'\sim x$, there exists a ray $r'$ satisfying $r'(0)=r(0)$, $r'(+\infty)=x'$ and such that the Hausdorff distance $d_{\rm Haus}(r,r')$ is at most $d(x,x')$.
\end{lem}
\begin{proof}
It suffices to consider the case when $d(x,x')=1$. Let $\mf{w}$ be the only hyperplane separating $x$ and $x'$; let $\mf{h}$ denote the side of $\mf{w}$ containing $x$. Note that $\mf{w}$ must be transverse to all but finitely many hyperplanes in $\mscr{W}(r)$. Thus, $r$ intersects the carrier $C(\mf{w})$ in a sub-ray $\g\cu r$.

If $r(0)\in\mf{h}$, the ray $\g$ does not cross $\mf{w}$. Let $\g'$ be the ray such that $\mscr{W}(\g(n)|\g'(n))=\{\mf{w}\}$ for all $n\geq 0$. In this case, the ray $r'$ is obtained by following $r$ up to $\g(0)$, crossing $\mf{w}$, and finally following $\g'$ all the way to $x'$. 

If instead $r(0)\in\mf{h}^*$, there exists $k\geq 0$ such that $\mscr{W}(\g(k)|\g(k+1))=\{\mf{w}\}$. Let $\g''$ be the ray with $\mscr{W}(\g''(n)|\g(n+k+1))=\{\mf{w}\}$ for all $n\geq 0$; in particular, $\g''(0)=\g(k)$ and $\g''(+\infty)=x'$. We construct $r'$ by following $r$ up to $\g(k)$ and then, rather than crossing $\mf{w}$, following $\g''$ until $x'$.
\end{proof}

We say that a subset $C\cu X$ is \emph{convex} if every geodesic with endpoints in $C$ is entirely contained in $C$; equivalently, $I(x,y)\cu C$ for every ${x,y\in C}$. Halfspaces are precisely those nonempty convex subsets of $X$ whose complement is convex and nonempty. Given any subset $A\cu X$, we denote by $\hull(A)$ the smallest convex subset of $X$ that contains $A$. This coincides with the intersection of all halfspaces containing $A$. The subcomplex $\hull(A)$ is itself a $\CAT$ cube complex and its hyperplane set is identified with $\mscr{W}(A)$.

Given pairwise-intersecting convex subsets $C_1,...,C_k\cu X$, we always have ${C_1\cap ...\cap C_k\neq\emptyset}$. This is known as Helly's lemma (Theorem~2.2 in \cite{Roller}).

Every convex subset $C\cu X$ comes equipped with a $1$--Lipschitz projection $\pi_C\colon X\ra C$, with the property that $\pi_C(x)\in I(x,y)$ for every $x\in X$ and every $y\in C$. We refer to $\pi_C$ as the \emph{gate-projection} to $C$. For all $x,y\in X$, we have $\mscr{W}(x|\pi_C(x))=\mscr{W}(x|C)$ and $\mscr{W}(x|y)\cap\mscr{W}(C)=\mscr{W}(\pi_C(x)|\pi_C(y))$, so $\pi_C$ is the nearest-point projection with respect to the combinatorial metric.

Consider now two disjoint halfspaces $\mf{h}$ and $\mf{k}$. We set
\[M(\mf{h},\mf{k})=\{(x,y)\in\mf{h}\x\mf{k}\mid d(x,y)=d(\mf{h},\mf{k})\}\]
and denote by $B(\mf{h},\mf{k})$ the union of all intervals $I(x,y)$ with $(x,y)\in M(\mf{h},\mf{k})$. The set $B=B(\mf{h},\mf{k})$ is usually known as the \emph{bridge} and it is a convex subcomplex of $X$. Let $S_1$ and $S_2$ denote the projections of $M(\mf{h},\mf{k})\cu\mf{h}\x\mf{k}$ to the factors $\mf{h}$ and $\mf{k}$, respectively. We refer to $S_1$ and $S_2$ as the \emph{shores}; note that $S_1=B\cap\mf{h}$ and $S_2=B\cap\mf{k}$, so shores are also convex subcomplexes.

The restrictions $\pi_{S_1}|_{S_2}$ and $\pi_{S_1}|_{S_2}$ define cubical isomorphisms between $S_1$ and $S_2$. In fact, the intervals $I(x,y)$ associated to pairs $(x,y)\in M(\mf{h},\mf{k})$ are pairwise disjoint and all isomorphic to each other, giving rise to isometric splittings $B\simeq I(x,y)\x S$. Here $S_1$ corresponds to $\{x\}\x S$ and $S_2$ corresponds to $\{y\}\x S$.

We refer to the cube complex $S=S(\mf{h},\mf{k})$ simply as the \emph{abstract shore} when we do not want to identify it with any specific subcomplex of $X$. It is precisely the restriction quotient of $X$ (in the sense of p.\ 860 of \cite{CS}) associated to the set of hyperplanes transverse to both $\mf{h}$ and $\mf{k}$. Finally, we remark that, for every $x\in\mf{h}$ and every $y\in\mf{k}$, we have 
\[\tag{$\ast$}\label{distance formula} d(x,y)=d(x,S_1)+d(\pi_{S_1}(x),\pi_{S_1}(y))+d(\mf{h},\mf{k})+d(y,S_2).\]
The reader can consult for instance Section~2.G of \cite{CFI} or Section~2.2 of \cite{Fio3} for a more detailed treatment of bridges and shores.

We say that two disjoint halfspaces $\mf{h}$ and $\mf{k}$ are \emph{strongly separated} if the corresponding shores are singletons. Equivalently, no hyperplane of $X$ is transverse to both $\mf{h}$ and $\mf{k}$. Similarly, we say that two hyperplanes are strongly separated if they bound strongly separated halfspaces.

The cube complex $X$ is \emph{irreducible} if it cannot be split as a product of lower-dimensional cube complexes. Every finite dimensional cube complex admits a canonical decomposition as product of irreducible cube complexes (Proposition~2.6 in \cite{CS}); we refer to it as the \emph{De Rham decomposition}.

Throughout the paper, all groups will be implicitly assumed to be finitely generated. When a group $G$ acts on a $\CAT$ cube complex $X$, we will assume that the action is by \emph{cubical automorphisms}, i.e.\ by isometries taking vertices to vertices. We say that $G\acts X$ is \emph{essential} if no $G$--orbit is contained in a metric neighbourhood of a halfspace. Similarly, we say that $X$ is \emph{essential} if no halfspace is contained in a metric neighbourhood of the corresponding hyperplane. If $G$ acts cocompactly, $X$ is essential if and only if the action $G\acts X$ is essential.

\begin{rmk}\label{hyp vs irr}
Every essential, Gromov hyperbolic $\CAT$ cube complex is irreducible. Indeed, essentiality guarantees that $X$ has no bounded factors in its De Rham decomposition, whereas hyperbolicity implies that there is at most one unbounded factor.
\end{rmk} 

The action $G\acts X$ is \emph{hyperplane-essential} if each hyperplane-stabiliser acts essentially on the corresponding hyperplane. Similarly, $X$ is \emph{hyperplane-essential} if all its hyperplanes are essential cube complexes. Again, if $G$ acts cocompactly, $X$ is hyperplane-essential if and only if the action $G\acts X$ is hyperplane-essential. This follows from Exercise~1.6 in \cite{Sageev-notes}, which we record here for later use:

\begin{lem}\label{hyp cocpt}
Let $\mf{w}\in\mscr{W}(X)$ be a hyperplane and $G_{\mf{w}}<G$ its stabiliser. If $G\acts X$ is cocompact, the action $G_{\mf{w}}\acts\mf{w}$ also is.
\end{lem}

The following notion appeared in Definition~7.3 in \cite{Fernos} and Definition~5.8 in \cite{Fernos-Lecureux-Matheus}; also see Proposition~7.5 in \cite{Fernos}.

\begin{defn}\label{regular defn}
Let $X$ be irreducible. A point $x\in\partial X$ is \emph{regular} if, for every $\mf{h}\in\s_x$, there exists $\mf{k}\in\s_x$ such that $\mf{k}$ and $\mf{h}^*$ are strongly separated. Equivalently, $\s_x$ contains a chain $\mf{h}_0\supsetneq\mf{h}_1\supsetneq ...$ such that $\mf{h}_n^*$ and $\mf{h}_{n+1}$ are strongly separated for every $n\geq 0$. We refer to the latter as a \emph{strongly separated chain} and denote by $\partial_{\rm reg}X\cu\partial X$ the subset of regular points.
\end{defn}

With reference to the proof sketch of Theorem~\ref{ext Moeb intro} in the introduction, the following is the formulation of Steps~(Ia) and~(Ib) that we will actually use. Note that hyperbolicity is not required here. 

\begin{prop}\label{ample}
Let $X$ be irreducible, essential and endowed with a proper cocompact action $G\acts X$ of a non-virtually-cyclic group. Consider two halfspaces $\mf{h}_1,\mf{h}_2$ and a nonempty $G$--invariant subset $\mc{A}\cu\partial_{\rm reg}X$.
\begin{enumerate}
\item The intersections $\mf{h}_i\cap\mc{A}$ and $\mf{h}_i^*\cap\mc{A}$ are always nonempty.
\end{enumerate}
If moreover $X$ is hyperplane-essential, the following also hold.
\begin{enumerate}
\setcounter{enumi}{1}
\item The halfspaces $\mf{h}_1$ and $\mf{h}_2$ are transverse if and only if the set $\mc{A}$ intersects each of the four sectors $\mf{h}_1\cap\mf{h}_2$, $\mf{h}_1^*\cap\mf{h}_2$, $\mf{h}_1\cap\mf{h}_2^*$ and $\mf{h}_1^*\cap\mf{h}_2^*$.
\item If we have $\mf{h}_1\cap\mc{A}\subsetneq\mf{h}_2\cap\mc{A}$, then $\mf{h}_1\subsetneq\mf{h}_2$.
\end{enumerate}
\end{prop}
\begin{proof} 
Part~(1) follows from Lemmas~2.9 and 2.18 in \cite{BF1}. If $\mf{h}_1$ and $\mf{h}_2$ are not transverse, one of the four intersections $\mf{h}_1\cap\mf{h}_2$, $\mf{h}_1^*\cap\mf{h}_2$, $\mf{h}_1\cap\mf{h}_2^*$, $\mf{h}_1^*\cap\mf{h}_2^*$ is empty by definition; in particular, it cannot contain any point of $\mc{A}$. This proves one implication of part~(2), while the other follows from part~(1) and Proposition~2.11 in \cite{BF1}. We conclude by proving part~(3).

If $\mf{h}_1\cap\mc{A}\subsetneq\mf{h}_2\cap\mc{A}$, part~(2) shows that $\mf{h}_1$ and $\mf{h}_2$ cannot be transverse. We then have either $\mf{h}_1\subsetneq\mf{h}_2$, or $\mf{h}_2\cu\mf{h}_1$, or $\mf{h}_2^*\cu\mf{h}_1$ or $\mf{h}_2\cu\mf{h}_1^*$. In the second case we would have $\mf{h}_1\cap\mc{A}\subsetneq\mf{h}_2\cap\mc{A}\cu\mf{h}_1\cap\mc{A}$, and in the third case $\mf{h}_1\cap\mf{h}_2^*\cap\mc{A}\supseteq\mf{h}_2^*\cap\mc{A}\neq\emptyset$, which both lead to contradictions. In the fourth case, taking complements we obtain $\mf{h}_1\cu\mf{h}_2^*$, hence $\mf{h}_1\cap\mf{h}_2^*\cap\mc{A}\supseteq\mf{h}_1\cap\mc{A}\neq\emptyset$, which is also a contradiction. We conclude that $\mf{h}_1\subsetneq\mf{h}_2$.
\end{proof}

In relation to part~(3) of Proposition~\ref{ample}, note however that $\mf{h}_1\cap\mc{A}\cu\mf{h}_2\cap\mc{A}$ \emph{does not} imply $\mf{h}_1\cu\mf{h}_2$, as we might actually have $\mf{h}_2\subsetneq\mf{h}_1$ in this case.

\subsection{Combinatorial geodesics vs $\CAT$ geodesics.}\label{CAT vs l^1 geod}

The next result is probably well-known to experts, but a proof does not seem to appear in the literature. We provide it in this subsection for completeness.

We will always specify whether geodesics are meant with respect to the $\CAT$ metric on $X$, or rather with respect to the combinatorial metric $d$. We stress that Hausdorff distances, however, will always be calculated with respect to the combinatorial metric.

\begin{prop}\label{CAT near l1}
Let $X$ be a $D$--dimensional $\CAT$ cube complex. Every $\CAT$ ray based at a vertex of $X$ is at Hausdorff distance at most $D$ from a combinatorial ray with the same origin. 
\end{prop}

Given a combinatorial geodesic $\g$, the hyperplanes of $\mscr{W}(\g)$ can be arranged in a sequence $(\mf{w}_n)_{n\geq 0}$ according to the order in which they are crossed by $\g$ after $\g(0)$. We denote this sequence by $\mf{s}(\g)$.

\begin{lem}\label{hyperplane orders}
Let $(\mf{w}_n)_{n\geq 0}$ be a (finite or infinite) sequence of pairwise distinct hyperplanes of $X$ and let $p\in X$ be a vertex. There exists a combinatorial geodesic $\g$ based at $p$ such that $\mf{s}(\g)=(\mf{w}_n)_{n\geq 0}$ if and only if, for every $n\geq 0$, we have $\mscr{W}(p|\mf{w}_n)=\{\mf{w}_0,...,\mf{w}_{n-1}\}\setminus\mscr{W}(\mf{w}_n)$.
\end{lem}
\begin{proof}
If there exists a geodesic $\g$ such that $\g(0)=p$ and $\mf{s}(\g)=(\mf{w}_n)_{n\geq 0}$, each point $\g(n)$ lies in the carrier $C(\mf{w}_n)$. Thus $\mscr{W}(p|\mf{w}_n)\cu\mscr{W}(p|\g(n))$ and every element of $\mscr{W}(p|\g(n))$ either crosses $\mf{w}_n$ or lies in $\mscr{W}(p|\mf{w}_n)$. We conclude that ${\mscr{W}(p|\mf{w}_n)=\{\mf{w}_0,...,\mf{w}_{n-1}\}\setminus\mscr{W}(\mf{w}_n)}$ for all $n\geq 0$.

Assuming instead that the sequence $(\mf{w}_n)_{n\geq 0}$ satisfies the latter condition, we are going to construct points $p_n\in C(\mf{w}_n)$ with $\mscr{W}(p|p_n)=\{\mf{w}_0,...,\mf{w}_{n-1}\}$. We then obtain $\g$ by setting $\g(n)=p_n$. We proceed by induction on $n\geq 0$, observing that the case $n=0$ immediately follows from $\mscr{W}(p|\mf{w}_0)=\emptyset$.

Given $p_n\in C(\mf{w}_n)$ with $\mscr{W}(p|p_n)=\{\mf{w}_0,...,\mf{w}_{n-1}\}$, let $p_{n+1}\in X$ be the only point with $\mscr{W}(p_n|p_{n+1})=\{\mf{w}_n\}$. As $\mscr{W}(p|p_{n+1})=\{\mf{w}_0,...,\mf{w}_n\}$, we only need to show that $p_{n+1}$ lies in the carrier $C(\mf{w}_{n+1})$. If this failed, there would exist a hyperplane $\mf{u}\in\mscr{W}(p_{n+1}|\mf{w}_{n+1})$ and we would have either $\mf{u}\in\mscr{W}(p,p_{n+1}|\mf{w}_{n+1})$ or $\mf{u}\in\mscr{W}(p_{n+1}|p,\mf{w}_{n+1})$. The former is forbidden by $\mscr{W}(p|\mf{w}_{n+1})\cu\{\mf{w}_0,...,\mf{w}_n\}$, whereas the latter would clash with the fact that, for $0\leq i\leq n$, each $\mf{w}_i$ either crosses $\mf{w}_{n+1}$ or lies in $\mscr{W}(p|\mf{w}_{n+1})$.
\end{proof} 

\begin{proof}[Proof of Proposition~\ref{CAT near l1}]
Let $\rho$ be a $\CAT$ ray based at a vertex $p\in X$. For every hyperplane $\mf{w}\in\mscr{W}(X)$ there exists exactly one side of $\mf{w}$ that has unbounded intersection with $\rho$. The collection of these halfspaces forms an ultrafilter $\s\cu\mscr{H}(X)$ representing a point $x\in\partial X$. 

A hyperplane is crossed by $\rho$ if and only if it lies in the set $\mscr{W}(p|x)$; thus, $\rho$ is entirely contained in the subcomplex $I(p,x)\cap X$. Let $(\mf{w}_n)_{n\geq 0}$ be an ordering of the elements of $\mscr{W}(p|x)$, so that $m<n$ if $\mf{w}_m$ is crossed by $\rho$ before $\mf{w}_n$. By Lemma~\ref{hyperplane orders}, there exists a combinatorial ray $r$ from $p$ to $x$ such that $\mf{s}(\g)=(\mf{w}_n)_{n\geq 0}$. It remains to prove that $d_{\rm Haus}(r,\rho)\leq D$. 

Let $u\in r$ and $v\in\rho$ be points (not necessarily vertices) that are not separated by any hyperplane of $X$. Note that for every point $u\in r$ there exists such a point $v\in\rho$ and vice versa. Let $\iota\colon I(p,x)\cap X\ra\R^D$ be an $\ell^1$--isometric cubical embedding; it exists for instance by Theorem~1.14 in \cite{BCGNW}. Under the map $\iota$, preimages of convex sets are convex and, therefore, preimages of halfspaces are halfspaces. It follows that the points $\iota(u)$ and $\iota(v)$ are not separated by any hyperplane of $\R^D$, hence they lie in a translate of a unit cube of $\R^D$. Thus $d(u,v)=d(\iota(u),\iota(v))\leq D$. 
\end{proof}

\subsection{Median barycentres.}\label{median barycentres}

Let $S$ be a bounded $\CAT$ cube complex. 

Considering the $\CAT$ metric on $S$, there exists a unique barycentre $c_S\in S$. This is the centre of the unique smallest closed ball containing $S$; see e.g.\ Proposition~II.2.7 in \cite{BH} or Proposition~3.73 in \cite{DK}. However, the point $c_S$ is in general not a vertex of $S$, nor a vertex of any iterated cubical subdivision. This is illustrated in Figure~\ref{barycentre example}.

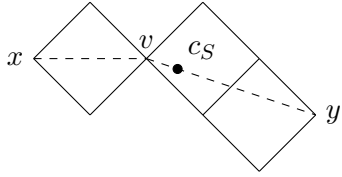
\begin{figure} 
\begin{tikzpicture}[scale=1.50]
\draw [fill] (0,0) -- (-0.5,-0.5);
\draw [fill] (-0.5,-0.5) -- (-1,0);
\draw [fill] (-1,0) -- (-0.5,0.5);
\draw [fill] (0,0) -- (-0.5,0.5);
\draw [fill] (0,0) -- (0.5,-0.5);
\draw [fill] (0.5,-0.5) -- (1,0);
\draw [fill] (1,0) -- (0.5,0.5);
\draw [fill] (0,0) -- (0.5,0.5);
\draw [fill] (0.5,-0.5) -- (1,-1);
\draw [fill] (1,-1) -- (1.5,-0.5);
\draw [fill] (1.5,-0.5) -- (1,0);
\node[right] at (1.5,-0.5) {$y$};
\node[left] at (-1,0) {$x$};
\draw[dashed] (-1,0) -- (0,0);
\draw[dashed] (0,0) -- (1.5,-0.5);
\node[above] at (0,0) {$v$};
\draw[fill] (0.276,-0.092) circle [radius=0.04cm];
\node[above right] at (0.276,-0.092) {$c_S$};
\end{tikzpicture}
\caption{The pictured cube complex $S$ consists of three squares. The $\CAT$ geodesic from $x$ to $y$ is the unique longest geodesic in $S$. The $\CAT$ distance between $v$ and the barycentre $c_S$ is $\frac{1}{2}(\sqrt{5}-\sqrt{2})$. This number is not of the form $\frac{1}{2^n}\sqrt{a^2+b^2}$ for any $a,b,n\in\N$, so $c_S$ is not a vertex of any iterated cubical subdivision of $S$.}
\label{barycentre example} 
\end{figure}

In Section~\ref{main section} we will need a different notion of barycentre, which we now introduce. It will always be a vertex of the first cubical subdivision $S'$. 

In the discussion below, points of $S$ are assumed to be vertices and we only consider the combinatorial metric on $X$, as in the rest of the paper. 

Let $\mf{h}\in\mscr{H}(S)$ be a side of the hyperplane $\mf{w}$ and let $x\in\mf{h}$ and $y\in\mf{h}^*$ be vertices maximising the distance from $\mf{w}$. We say that $\mf{w}$ is \emph{balanced} if $d(x,\mf{w})=d(y,\mf{w})$ and \emph{unbalanced} otherwise. If $d(x,\mf{w})>d(y,\mf{w})$, we call $\mf{h}$ \emph{heavy} and $\mf{h}^*$ \emph{light}. 

\begin{lem}\label{one light}
Given halfspaces $\mf{h},\mf{k}\in\mscr{H}(S)$ with $\mf{h}\cap\mf{k}=\emptyset$ and $\mf{h}\neq\mf{k}^*$, at least one of them is light.
\end{lem}
\begin{proof}
Let $\mf{w}$ and $\mf{u}$ be the hyperplanes associated to $\mf{h}$ and $\mf{k}$, respectively. Pick $x\in\mf{h}$ and $y\in\mf{k}$ maximising the distance from $\mf{w}$ and $\mf{u}$, respectively. Note that $\mf{k}\cu\mf{h}^*$ and $\mf{h}\cu\mf{k}^*$. If $\mf{h}$ and $\mf{k}$ were both not light, we would have 
\[d(x,\mf{w})\geq d(y,\mf{w})\geq d(y,\mf{u})+d(\mf{u},\mf{w})>d(y,\mf{u})\]
and similarly $d(y,\mf{u})>d(x,\mf{w})$. This is a contradiction.
\end{proof}

In particular, any two heavy halfspaces intersect and any two balanced hyperplanes are transverse. By Helly's Lemma, the intersection of all heavy halfspaces is nonempty. It is a cube $c\cu S$ cut by all balanced hyperplanes.

The centre of $c$ is a vertex $m_S$ of the cubical subdivision $S'$. We refer to it as the \emph{median barycentre} of $S$. Note that $m_S$ is a vertex of $S$ if and only if every hyperplane is unbalanced. For instance, $m_S=v$ in Figure~\ref{barycentre example}.

We remark that, given bounded $\CAT$ cube complexes $S_1$ and $S_2$ and an isomorphism $F\colon S_1\ra S_2$, we have $F(m_{S_1})=m_{S_2}$.

\subsection{Cross ratios on cube complexes.}\label{cross ratio prelims}

Let $X$ be a $\CAT$ cube complex. Fixing a vertex $p\in X$, the \emph{Gromov product} of two points $x,y\in\overline X$ is:
\[(x\cdot y)_p=\#\mscr{W}(p|x,y)=d(p,m(p,x,y))\in\N\cup\{+\infty\}.\]
The following is Lemma~2.3 in \cite{BF1}.

\begin{lem}\label{infinite Gromov product}
Consider $x,y,z\in\overline X$ and $p\in X$.
\begin{enumerate}
\item We have $m(x,y,z)\in X$ if and only if each of the three intervals $I(x,y)$, $I(y,z)$, $I(z,x)$ intersects $X$.
\item We have $(x\cdot y)_p<+\infty$ if and only if $I(x,y)$ intersects $X$.
\end{enumerate}
\end{lem}

Let $\mscr{A}\cu(\overline X)^4$ be the subset of $4$--tuples $(x,y,z,w)$ such that at most one of the values $(x\cdot y)_p+(z\cdot w)_p$, $(x\cdot z)_p+(y\cdot w)_p$ and $(x\cdot w)_p+(y\cdot z)_p$ is infinite; by Lemma~\ref{infinite Gromov product}, the set $\mscr{A}$ is independent of the choice of $p$.

In our previous work with Incerti-Medici \cite{BFI-new}, we introduced a cross ratio $\Cr\colon\mscr{A}\ra\Z\cup\{\pm\infty\}$, which admits the following equivalent characterisations:
\begin{enumerate}
\item $\Cr(x,y,z,w)=\#\mscr{W}(x,z|y,w)-\#\mscr{W}(x,w|y,z)$;
\item $\Cr(x,y,z,w)=(x\cdot z)_p+(y\cdot w)_p-(x\cdot w)_p-(y\cdot z)_p$;
\item $\Cr(x,y,z,w)=d(x,w)+d(y,z)-d(x,z)-d(y,w)$, if $x,y,z,w\in X$.
\end{enumerate}
In particular, the second characterisation does not depend on the choice of $p\in X$. Note that $\Cr$ satisfies symmetries (i)--(iv) from the introduction, as long as all involved $4$--tuples lie in $\mscr{A}$. We will sometimes write $\Cr_X$ when we wish to specify the cube complex under consideration.

Endowing $\mscr{A}\cu (\overline X)^4$ with the subspace topology, the following is Proposition~3.3 in \cite{BF1}. Note that $\overline X$ and $\mscr{A}$ are totally disconnected.

\begin{prop}\label{cr continuous}
If $X$ is locally finite, the cross ratio $\Cr$ is continuous.
\end{prop}

\subsection{$\CAT$ cuboid complexes.}\label{cuboids sect}

As mentioned in the introduction, all results in this paper equally hold for cube complexes with variable edge lengths: \emph{cuboid complexes} in our terminology. 

For the sake of simplicity and clarity, we will only treat $\CAT$ cube complexes in most of the paper. Only very minor changes are required in order to adapt our arguments to general $\CAT$ cuboid complexes. We briefly describe them here, along with the relevant definitions.

Consider for a moment a genuine $\CAT$ cube complex $X$. Every function $\mu\colon\mscr{W}(X)\ra\R_{>0}$ determines a \emph{weighted combinatorial metric} $d_\mu$ on $X$. For vertices $v,w\in X$, this is given by:
\[d_{\mu}(v,w)=\sum_{\mf{w}\in\mscr{W}(v|w)}\mu(\mf{w}).\]
For instance, the usual combinatorial metric $d$ arises from the function that assigns value $+1$ to each hyperplane.

\begin{defn}
A \emph{$\CAT$ cuboid complex} $\mbb{X}$ is any metric cell complex $(X,d_{\mu})$ arising from this construction.
\end{defn}

Two cuboid complexes $\mbb{X}=(X,d_{\mu})$ and $\mbb{Y}=(Y,d_{\nu})$ are \emph{isomorphic} if there exists an isometric cellular isomorphism $f\colon\mbb{X}\ra\mbb{Y}$. In other words, $f\colon X\ra Y$ is an isomorphism of $\CAT$ cube complexes inducing a map $f_*\colon\mscr{W}(X)\ra\mscr{W}(Y)$ such that $\mu=\nu\o f_*$. 

Note, however, that there can be isometries $\mbb{X}\ra\mbb{Y}$ that do not preserve the cellular structures. For instance, consider the cuboid complex $\mbb{X}'$ arising from the cubical subdivision $X'$. If we assign each edge of $\mbb{X}'$ half the length of the corresponding edge of $\mbb{X}$, the identity map $\mbb{X}\ra\mbb{X}'$ is a (surjective) isometry, but never an isomorphism.

When dealing with $\CAT$ cuboid complexes, rather than $\CAT$ cube complexes, the following adaptations and conventions are required.
\begin{enumerate}
\item All group actions on cuboid complexes will be assumed to be by \emph{automorphisms} (i.e.\ self-isomorphisms).
\item We consider two actions to be the same when they are equivariantly \emph{isometric}. In Theorems~\ref{hyp CR},~\ref{ext Moeb intro} and in Corollaries~\ref{hyp MLSR},~\ref{non-hyp CR}, equivariant \emph{isomorphisms} of cube complexes need to be replaced with equivariant \emph{isometries} of cuboid complexes. It is easy to see that it is not possible to map vertices to vertices in general.
\item We define hyperplanes and halfspaces of $\mbb{X}=(X,d_{\mu})$ to coincide with hyperplanes and halfspaces of the underlying cube complex $X$. This also explains how to interpret notations like $\mscr{W}(A|B)$ in this context.
\item Given a subset $\mc{U}\cu\mscr{W}(\mbb{X})=\mscr{W}(X)$, the \emph{cardinality} $\#\mc{U}$ should always be replaced by the \emph{weight} $\sum_{\mf{w}\in\mc{U}}\mu(\mf{w})$. Nonempty subsets are still precisely those that have positive weight.
\item The cross ratio $\Cr$ will no longer take values in $\Z\cup\{\pm\infty\}$, but rather in $M\cup\{\pm\infty\}$, where $M$ is the $\Z$--module generated by the image of the map $\mu$. A similar observation applies to length functions.
\end{enumerate}

\addtocontents{toc}{\protect\setcounter{tocdepth}{2}}
\section{The Morse property in cube complexes.}\label{Morse section}

Other than the proof of Theorem~\ref{R vs c intro} (consisting of Theorem~\ref{properties of Phi} and Remark~\ref{components are finite}), most of this section will be devoted to collecting more or less well-known facts from the literature. Throughout:

\begin{ass}
Let the $\CAT$ cube complex $X$ be finite dimensional and locally finite.
\end{ass}

\subsection{Contracting geodesics.}

Recall that we only endow $X$ with its combinatorial metric. All geodesics will be combinatorial in this subsection.

\begin{defn}\label{Morse defn}
Let $Y$ be a proper metric space. Given a closed subset $A\cu Y$, we denote by $\pi_A\colon Y\ra 2^A$ the nearest-point projection to $A$. If $B\cu Y$, we write $\pi_A(B)$ instead of $\bigcup_{b\in B}\pi_A(b)$.

A closed subset $A\cu Y$ is \emph{(strongly) contracting} if there exists $D>0$ such that every metric ball $B$ disjoint from $A$ satisfies $\text{diam}(\pi_{A}(B))\leq D$.

Given a function $M\colon [1,+\infty)\ra [0,+\infty)$, a quasi-geodesic $\g\cu Y$ is \emph{$M$--Morse} if, for every $C>0$ and every $(C,C)$--quasi-geodesic $\eta$ with endpoints on $\g$, the entire $\eta$ is contained in the (open) $M(C)$--neighbourhood of $\g$. We say that $\g$ is \emph{Morse} if it is $M$--Morse for some function $M$.
\end{defn} 

We refer the reader to \cite{ACGH} for a detailed discussion of contracting subsets and the Morse property in general metric spaces.

\begin{defn}
We say that a geodesic $\g\cu X$ is \emph{$C$--lean} if there do not exist transverse subsets $\mc{U}\cu\mscr{W}(\g)$ and $\mc{V}\cu\mscr{W}(X)$ such that $\min\{\#\mc{U},\#\mc{V}\}>C$ and such that $\mc{U}\sqcup\mc{V}$ does not contain facing triples. We say that $\g$ is \emph{lean} if it is $C$--lean for some $C\geq 0$.
\end{defn}

The following is due to A.\ Genevois; see Corollary~3.7 in \cite{Genevois} and Lemma~4.6 in \cite{Genevois-surv}.

\begin{thm}\label{lean vs contracting vs Morse}
For a ray $\g\cu X$, we have: $\text{lean} \Leftrightarrow \text{contracting} \Leftrightarrow \text{Morse}$.
\end{thm}

Given rays $\g$ and $\g'$ at finite Hausdorff distance, it is clear from definitions that $\g$ satisfies the above conditions if and only if $\g'$ does.

\begin{lem}\label{rays to cu}
Let $\alpha$ and $\g$ be rays in $X$ with $\alpha(+\infty)\sim\g(+\infty)$. If $\g$ is contracting, then $\alpha$ is at finite Hausdorff distance from $\g$, hence contracting.
\end{lem}
\begin{proof}
By Lemma~\ref{rays to nearby points}, it suffices to consider the case when $\alpha(+\infty)=\g(+\infty)$. By Theorem~\ref{lean vs contracting vs Morse}, there exists $C>0$ such that $\g$ is $C$--lean. Set $p=\g(0)$, $q=\alpha(0)$, $x=\alpha(+\infty)=\g(+\infty)$ and $I=I(x,p)$. Since $\mscr{W}(q|I)=\mscr{W}(q|x,p)$ is a finite subset of $\mscr{W}(q|x)=\mscr{W}(\alpha)$, the intersection $\alpha\cap I$ is a sub-ray of $\alpha$. We conclude by showing that ${d(u,\g)\leq 2C}$ for every point $u\in I$. 

Pick a point $v\in\g$ with $d(u,p)=d(v,p)$ and set $m=m(p,u,v)$. Note that $\mscr{W}(m|u)$ and $\mscr{W}(m|v)$ are contained in $\mscr{W}(x|p)=\mscr{W}(\g)$. Every halfspace $\mf{h}\in\mscr{H}(m|u)$ is transverse to every halfspace $\mf{k}\in\mscr{H}(m|v)$; indeed, we have $m\in\mf{h}^*\cap\mf{k}^*$, $u\in\mf{h}\cap\mf{k}^*$, $v\in\mf{h}^*\cap\mf{k}$ and $x\in\mf{h}\cap\mf{k}$. Moreover, the sets $\mscr{W}(m|u)$ and $\mscr{W}(m|v)$ have the same size and contain no facing triples. Since $\g$ is $C$--lean, we conclude that $\#\mscr{W}(m|u)=\#\mscr{W}(m|v)\leq C$. This shows that $d(u,v)\leq 2C$.
\end{proof}

\subsection{Roller boundaries vs contracting boundaries.}\label{d_c vs d_R}

Unlike the rest of the paper, this subsection employs both the combinatorial and $\CAT$ metrics on $X$; we will specify each time whether geodesics are meant with respect to the former or latter. Still, the notation $d(\cdot,\cdot)$ will always refer to the combinatorial metric

The \emph{contracting boundary} $\partial_cX$ was introduced in \cite{Charney-Sultan}. Disregarding topologies for the moment, $\partial_cX$ is the subset of the visual boundary $\partial_{\infty}X$ that consists of points represented by contracting $\CAT$ rays.

In order to relate the contracting boundary $\partial_cX$ and the Roller boundary $\partial X$, we introduce the following (see Lemma~\ref{rays to cu} for the equivalence in the definition):

\begin{defn}
We say that a point $x\in\partial X$ is \emph{contracting} (or a \emph{contracting ultrafilter}) if one (equivalently, each) combinatorial ray representing $x$ is contracting. We denote the set of contracting ultrafilters by $\partial_{\rm cu}X\cu\partial X$.
\end{defn}

We stress that our definition of contracting point is not equivalent to the one in Remark~6.7 of \cite{Fernos-Lecureux-Matheus}; in fact, our notion is weaker. 

In general, the inclusion $\partial_{\rm cu}X\cu\partial X$ is strict. If however $X$ is Gromov hyperbolic, every combinatorial ray in $X$ is contracting (see e.g.\ Theorem~3.3 in \cite{Genevois2}) and we have $\partial_{\rm cu}X=\partial X$.

Lemma~\ref{rays to cu} shows that the set $\partial_{\rm cu}X$ is a union of $\sim$--equivalence classes. The following result provides more information.

\begin{lem}\label{bounded components of cu}
\begin{enumerate}
\item Every component of $\partial_{\rm cu}X$ is bounded.
\item Points $x\in\partial_{\rm cu}X$ and $y\in\partial X$ lie in the same component if and only if they satisfy $I(x,y)\cap X=\emptyset$.
\end{enumerate}
\end{lem} 
\begin{proof}
Consider a point $x\in\partial_{\rm cu}X$ and a $C$--lean combinatorial ray $r$ with $x=r(+\infty)$; set $p=r(0)$. We simultaneously prove both parts of the lemma by showing that, for $y\in\partial X$, the condition $(x\cdot y)_p=+\infty$ implies $d(x,y)\leq C$.

Suppose for the sake of contradiction that $\mscr{W}(p|x,y)$ is infinite and $\mscr{W}(x|y)$ contains a finite subset $\mc{U}$ with $\#\mc{U}>C$. Given $\mf{w}\in\mscr{W}(x|y)$ and a halfspace $\mf{h}\in\mscr{H}(p|x,y)$, either $\mf{w}\cu\mf{h}$ or $\mf{w}$ and $\mf{h}$ are transverse. Fixing $\mf{w}$, there are at most $d(p,\mf{w})<+\infty$ halfspaces $\mf{h}\in\mscr{H}(p|\mf{w})$. Thus, all but finitely many hyperplanes in $\mscr{W}(p|x,y)$ are transverse to all elements of $\mc{U}$. As $\mscr{W}(p|x,y)$ and $\mc{U}$ contain no facing triples, this violates $C$--leanness of $r$.
\end{proof}

\begin{lem}\label{properties of cu}
Consider a point $x\in\partial_{\rm cu}X$. There exists an infinite descending chain of halfspaces $\mf{h}_0\supsetneq\mf{h}_1\supsetneq ...$ such that $\bigcap\mf{h}_n=Z(x)$ and such that the shores $S(\mf{h}_n^*,\mf{h}_{n+1})$ are finite cube complexes of uniformly bounded diameter.
\end{lem}
\begin{proof}
Let $r$ be a contracting combinatorial ray with $x=r(+\infty)$. Theorem~3.9 in \cite{Genevois} yields an infinite chain of halfspaces $\mf{h}_0\supsetneq\mf{h}_1\supsetneq ...$ such that the shores $S_n=S(\mf{h}_n^*,\mf{h}_{n+1})$ have uniformly bounded diameter and such that $x\in\mf{h}_n$ for every $n\geq 0$. Since shores embed as subcomplexes of $X$, they are locally finite. Boundedness then implies that each $S_n$ is finite.

Now, observe that $\mf{h}_n^*\cap Z(x)=\emptyset$ for all $n\geq 0$. Otherwise, there would exist an integer $k\geq 0$ and a point $y\in\mf{h}_k^*\cap Z(x)$. We would then have $y\in\mf{h}_n^*\cap Z(x)$ for all $n\geq k$, violating the fact that $d(x,y)<+\infty$.

This shows that $Z(x)$ is contained in each $\mf{h}_n$. Given a point $z\in\bigcap\mf{h}_n$, we have $I(x,z)\cu\bigcap\mf{h}_n$ and hence $I(x,z)\cap X=\emptyset$. Lemma~\ref{bounded components of cu} then shows that $z\in Z(x)$. This proves that $\bigcap\mf{h}_n=Z(x)$ and concludes the proof.
\end{proof}

\begin{rmk}\label{components are finite}
If $X$ is \emph{uniformly} locally finite, part~(1) of Lemma~\ref{bounded components of cu} can actually be promoted to say that components of $\partial_{\rm cu}X$ are finite.

Indeed, let $x$ and $\mf{h}_0\supsetneq\mf{h}_1\supsetneq ...$ be as in the statement of Lemma~\ref{properties of cu}. Given $y\in Z(x)$, every $\mf{w}\in\mscr{W}(x|y)$ must be transverse to all but finitely many of the $\mf{h}_n$. In particular, $\mf{w}$ must be a hyperplane of almost all shores $S(\mf{h}_n^*,\mf{h}_{n+1})$. The latter are uniformly finite, as they embed as uniformly bounded subcomplexes of the uniformly locally finite cube complex $X$. We conclude that only finitely many hyperplanes of $X$ can separate two points of $Z(x)$, i.e.\ that $Z(x)$ is finite.
\end{rmk}

Since $X$ is finite dimensional, its combinatorial and $\CAT$ metrics are quasi-isometric. In particular, the notion of Morse quasi-geodesic is independent of our choice of one of the two metrics. 

We remark that Morse quasi-geodesics in complete $\CAT$ spaces always stay within bounded distance of contracting $\CAT$ geodesics (see e.g.\ Lemma~2.5, Theorem~2.9 and the proof of Corollary~2.10 in \cite{Charney-Sultan}). Along with Theorem~\ref{lean vs contracting vs Morse} and Proposition~\ref{CAT near l1}, this observation yields the following.

\begin{cor}\label{both ways}
Every contracting combinatorial ray is at finite Hausdorff distance from a contracting $\CAT$ ray. Every contracting $\CAT$ ray is at finite Hausdorff distance from a contracting combinatorial ray.
\end{cor}

Every point $x\in\partial_{\rm cu}X$ is represented by combinatorial rays in $X$. These rays are all contracting by Lemma~\ref{rays to cu} and Corollary~\ref{both ways} shows that they are at finite Hausdorff distance from a unique family of pairwise-asymptotic contracting $\CAT$ rays. This yields a map
\[\Phi\colon\partial_{\rm cu}X\longrightarrow\partial_cX.\]
We endow $\partial_{\rm cu}X$ with the restriction of the topology of $\partial X$. We write $\partial_c^{\rm vis}X$ to refer to the contracting boundary $\partial_cX$ endowed with the restriction of the visual topology on $\partial_{\infty}X$. Although this is not one of the standard topologies on $\partial_cX$ (\cite{Charney-Sultan,Cashen-Mackay}), it is all that we will need in most of the paper.

The next result describes a few properties of the map $\Phi$. Recall the standing assumption that $X$ be finite dimensional and locally finite.

\begin{thm}\label{properties of Phi}
The map $\Phi\colon\partial_{\rm cu}X\ra\partial_c^{\rm vis}X$ is a continuous surjection, whose fibres are precisely the $\sim$--equivalence classes in $\partial_{\rm cu}X$. Moreover, $\Phi$ descends to a homeomorphism $\overline\Phi\colon\partial_{\rm cu} X/\mathord\sim\ra\partial_c^{\rm vis}X$.
\end{thm}
\begin{proof}
Surjectivity is immediate from Corollary~\ref{both ways}. Lemma~\ref{rays to cu} shows that $\Phi$ is constant on $\sim$--equivalence classes. On the other hand, if $\g$ and $\g'$ are combinatorial rays with $\g(+\infty)\not\sim\g'(\infty)$, it is clear that the distance $d(\g(n),\g'(n))$ diverges as $n$ goes to infinity. We conclude that the fibres of $\Phi$ are precisely the components of $\partial_{\rm cu}X$.

We now prove continuity of $\Phi$. Fix a vertex $p\in X$ and let $D=\dim X$. Given points $x,y\in\partial_{\rm cu}X$, let $\rho_x$ and $\rho_y$ be the $\CAT$ rays from $p$ to $\Phi(x)$ and $\Phi(y)$, respectively. Let $r_x$ and $r_y$ be combinatorial rays based at $p$ and satisfying $d_{\rm Haus}(r_x,\rho_x)\leq D$, $d_{\rm Haus}(r_y,\rho_y)\leq D$, as provided by Proposition~\ref{CAT near l1}. The points $x'=r_x(+\infty)$ and $y'=r_y(+\infty)$ lie in the components $Z(x)$ and $Z(y)$, respectively. Theorem~\ref{lean vs contracting vs Morse} moreover shows that $r_x$ and $r_y$ are $C$--lean for some constant $C>0$. 

Given an integer $n\geq 0$, we define an open neighbourhood $U_n(x)$ of $x$ in $\partial_{\rm cu}X$ as follows. Pick halfspaces $\mf{h}_3\cu\mf{h}_2\cu\mf{h}_1$ with $Z(x)\cu\mf{h}_3$, ${d(p,\mf{h}_1)>n}$, $\#\mscr{W}(\mf{h}_1^*|\mf{h}_2)>C$ and $\#\mscr{W}(\mf{h}_2^*|\mf{h}_3)>C$. Their existence is ensured by Lemma~\ref{properties of cu}. The neighbourhood $U_n(x)$ is then the subset of $\partial_{\rm cu}X$ consisting of points that lie within $\mf{h}_3$. 

\begin{claim}
For all $y\in U_n(x)$, we have $d(r_x(n),r_y(n))\leq 5C$.
\end{claim}

As $d_{\rm Haus}(r_x,\rho_x)\leq D$ and $d_{\rm Haus}(r_y,\rho_y)\leq D$, the claim shows that, given any open neighbourhood $V$ of $\Phi(x)$ in $\partial_c^{\rm vis}X$, we must have ${\Phi(U_n(x))\cu V}$ for all sufficiently large $n$. We thus complete the proof of continuity of $\Phi$ by proving the claim.

\emph{Proof of Claim.} Let the halfspaces $\mf{h}_1$, $\mf{h}_2$ and $\mf{h}_3$ be as above. Since $y\in\mf{h}_3$, Lemma~\ref{components vs separation} implies that $x',y'\in\mf{h}_2$. Indeed, as $\#\mscr{W}(\mf{h}_2^*|\mf{h}_3)>C$ and $r_x$ is $C$--lean, no infinite chain of halfspaces can be transverse to both $\mf{h}_2$ and $\mf{h}_3$. 

We set $q_x=r_x(n)$, $q_y=r_y(n)$ and $m=m(p,q_x,q_y)$. Note that the points $q_x$, $q_y$ and $m$ lie in $\mf{h}_1^*$ as $d(p,\mf{h}_1)>n$. Let $A_x\cu\mscr{H}(m|q_x)$ and $A_y\cu\mscr{H}(m|q_y)$ be the subsets of halfspaces containing $\mf{h}_2$; set $a_x=\#A_x$ and $a_y=\#A_y$. 

Each $\mf{k}\in\mscr{H}(m|q_x)\setminus A_x$ satisfies $\mf{k}^*\cap\mf{h}_2\neq\emptyset$, but also $m\in\mf{k}^*\cap\mf{h}_2^*$, $q_x\in\mf{k}\cap\mf{h}_2^*$ and $x'\in\mf{k}\cap\mf{h}_2$. We conclude that each halfspace in $\mscr{H}(m|q_x)\setminus A_x$ is transverse to $\mf{h}_2$ and, similarly, to $\mf{h}_1$. Since $\mscr{W}(\mf{h}_1^*|\mf{h}_2)\cu\mscr{W}(r_x)$ and $\#\mscr{W}(\mf{h}_1^*|\mf{h}_2)>C$, the fact that $r_x$ is $C$--lean implies that $d(m,q_x)-a_x\leq C$. Similarly, we obtain $d(m,q_y)-a_y\leq C$.

The sets $A_x$ and $A_y$ are transverse. As $A_x\cu\mscr{W}(r_x)$, leanness of $r_x$ implies that $\min\{a_x,a_y\}\leq C$. Since $d(m,q_x)=d(m,q_y)$ by construction, we have
\[|a_x-a_y|=\left|(d(m,q_x)-a_x)-(d(m,q_y)-a_y)\right|\leq C.\]
Hence $\max\{a_x,a_y\}\leq 2C$ and 
\[d(q_x,q_y)=d(m,q_x)+d(m,q_y)\leq (a_x+C)+(a_y+C)\leq 5C.~~~\square\]

In order to prove that $\overline\Phi$ is a homeomorphism, let us first obtain the following property: Given $x,x_n\in\partial_{\rm cu}X$ and $y\in\partial X$ with $\Phi(x_n)\ra\Phi(x)$ and $x_n\ra y$, we must have $y\in Z(x)$. 

Fix a basepoint $p\in X$. Let $\rho_n$ and $\rho$ be $\CAT$ rays from $p$ to $x_n$ and $x$, respectively. Proposition~\ref{CAT near l1} yields combinatorial rays $r_n$ and $r$ based at $p$, with $d_{\rm Haus}(r_n,\rho_n)\leq D$ and $d_{\rm Haus}(r,\rho)\leq D$. By Theorem~\ref{lean vs contracting vs Morse}, there exists $C\geq 0$ such that $r$ is $C$--lean. Let $\mf{h}_0\supsetneq\mf{h}_1\supsetneq ...$ be a chain in ${\s_x\setminus\s_p}$ with $\bigcap\mf{h}_n=Z(x)$, as provided by Lemma~\ref{properties of cu}. Up to passing to a subchain, we can assume that $\#\mscr{W}(\mf{h}_k^*|\mf{h}_{k+1})>C$ for all $k\geq 0$

Given any $k\geq 0$, there exists $N(k)\geq 0$ such that $r_n$ enters $\mf{h}_k$ for all ${n\geq N(k)}$. Indeed, the rays $\rho_n$ converge to $\rho$ uniformly on compact sets and $r_n,r$ are uniformly Hausdorff-close to $\rho_n,\rho$. We conclude that the points $z_n:=r_n(+\infty)$ lie in $\mf{h}_k$ for all $n\geq N(k)$. Since $x_n\in Z(z_n)$, we must also have $x_n\in\mf{h}_{k-1}$. This follows from Lemma~\ref{components vs separation}, as, by leanness of $r$, there is no infinite chain of halfspaces transverse to both $\mf{h}_k$ and $\mf{h}_{k-1}$. We conclude that $y=\lim x_n$ lies in $\bigcap\mf{h}_n=Z(x)$, as required. 

Now, since $\partial_{\infty}X$ is metrisable, so is $\partial_c^{\rm vis}X$. The fact that $\overline\Phi$ is a homeomorphism will thus follow if we prove that $\overline{\Phi}^{-1}$ is sequentially continuous (see e.g.\ Theorem~30.1(b) in \cite{Munkres}). Denote by $p\colon\partial_{\rm cu}X\ra\partial_{\rm cu} X/\mathord\sim$ the quotient projection. Suppose for the sake of contradiction that points $x,x_n\in\partial_{\rm cu}X$ are given so that $\Phi(x_n)\ra\Phi(x)$, but $p(x_n)\not\ra p(x)$. Possibly passing to a subsequence, there exists an open neighbourhood $V$ of $p(x)$ in $\partial_{\rm cu} X/\mathord\sim$ such that no $p(x_n)$ lies in $V$; hence no $x_n$ lies in $p^{-1}(V)$. Since $\partial X$ is compact and metrisable, a subsequence $x_{n_k}$ converges to a point $y\in\partial X$; as $p^{-1}(V)$ is open, we have $y\not\in p^{-1}(V)$. However, the set $p^{-1}(V)$ is a union of $\sim$--equivalence classes and we have shown above that $y\sim x\in p^{-1}(V)$, a contradiction.
\end{proof}

\begin{rmk}\label{topologies on cnt 1}
Given points $x,x_n\in\partial_{\rm cu}X$ satisfying $\Phi(x_n)\ra\Phi(x)$ and $Z(x)=\{x\}$, we always have $x_n\ra x$. Otherwise, compactness of $\partial X$ would yield a subsequence $x_{n_k}$ converging to a point $y\in\partial X$ different from $x$. However, we have shown during the proof of Theorem~\ref{properties of Phi} (right after the claim) that $y\in Z(x)=\{x\}$.
\end{rmk}

When $X$ is Gromov hyperbolic, we have $\partial_{\rm cu}X=\partial X$ and $\partial_c^{\rm vis}X$ coincides with the Gromov boundary $\partial_{\infty}X$. This case is worth highlighting:

\begin{cor}\label{properties of Phi hyp}
If $X$ is Gromov hyperbolic, the map $\Phi\colon\partial X\ra\partial_{\infty}X$ descends to a homeomorphism $\overline\Phi\colon\partial X/\mathord\sim\ra\partial_{\infty}X$.
\end{cor}

\subsection{Contracting non-terminating ultrafilters.}\label{cnt sect}

In this subsection:

\begin{ass}
Let, in addition, $X$ be irreducible. 
\end{ass}

The following definition is originally due to Nevo and Sageev (cf.\ Section~3.1 in \cite{Nevo-Sageev}).

\begin{defn}\label{nt defn}
A point $x\in\partial X$ is \emph{non-terminating} if $Z(x)=\{x\}$. Equivalently\footnote{This is because, given an ultrafilter $\s$ and a halfspace $\mf{h}\in\s$, the set $(\s\setminus\{\mf{h}\})\cup\{\mf{h}^*\}$ is an ultrafilter (at distance $1$ from $\s$) if and only if $\mf{h}$ is minimal in $(\s,\cu)$.}, the poset $(\s_x,\cu)$ does not have minimal elements. 
\end{defn}

We denote by $\partial_{\rm nt}X\cu\partial X$ the subset of non-terminating ultrafilters and set $\partial_{\rm cnt}X=\partial_{\rm cu}X\cap\partial_{\rm nt}X$. By Theorem~\ref{properties of Phi}, the restriction to $\partial_{\rm cnt}X$ of the map $\Phi\colon\partial_{\rm cu}X\ra\partial_cX$ is injective. Therefore, we will generally identify the sets $\partial_{\rm cnt}X\cu\partial_{\rm nt}X$ and $\Phi(\partial_{\rm cnt}X)\cu\partial_cX$, even writing simply $\partial_{\rm cnt}X\cu\partial_cX$.

\begin{rmk}\label{topologies on cnt 2}
It follows from Theorem~\ref{properties of Phi} and Remark~\ref{topologies on cnt 1} that $\partial_{\rm cnt}X$ inherits the same topology from $\partial_{\rm cu}X\cu\partial X$ and $\partial_c^{\rm vis}X$.
\end{rmk}

It will be useful to make the following observation.

\begin{lem}\label{regular vs NT} 
We have $\partial_{\rm cnt}X=\partial_{\rm cu}X\cap\partial_{\rm reg}X$.
\end{lem}
\begin{proof}
We have $\partial_{\rm cnt}X\supseteq\partial_{\rm cu}X\cap\partial_{\rm reg}X$ since it is clear from Definition~\ref{regular defn} that $\partial_{\rm reg}X\cu\partial_{\rm nt}X$. For the other inclusion, consider a point $x\in\partial_{\rm cnt}X$. Lemma~\ref{properties of cu} yields an infinite descending chain of halfspaces $\mf{h}_0\supsetneq\mf{h}_1\supsetneq ...$ such that the shores $S(\mf{h}_n^*,\mf{h}_{n+1})$ are finite and $\bigcap\mf{h}_n=\{x\}$. In particular, for every $n\geq 0$, only finitely many hyperplanes are transverse to both $\mf{h}_n$ and $\mf{h}_{n+1}$. We are going to show that, for every $k\geq 0$, there exists $N>k$ such that the halfspaces $\mf{h}_k^*$ and $\mf{h}_N$ are strongly separated. This implies that $\mf{h}_k^*$ and $\mf{h}_n$ are strongly separated for every $n\geq N$, providing a strongly separated chain in $\s_x$ and hence showing that $x$ is regular.

Suppose for the sake of contradiction that, for some $k\geq 0$ and every $n>k$, there exists a hyperplane $\mf{w}_n$ transverse to both $\mf{h}_k$ and $\mf{h}_n$. Note that each $\mf{w}_n$ is in particular transverse to $\mf{h}_k$ and $\mf{h}_{k+1}$ and that there are only finitely many such hyperplanes. We conclude that $\mf{w}_n=\mf{w}$ for a hyperplane $\mf{w}$ and infinitely many values of $n$. In particular, $\mf{w}$ is transverse to all $\mf{h}_n$ with $n\geq k$. It follows that there exists a point $y\in\partial X$ such that $y\in\mf{h}_n$ for all $n\geq 0$ and $\mf{w}\in\mscr{W}(x|y)$, violating the fact that $\bigcap\mf{h}_n=\{x\}$.
\end{proof}

Contracting non-terminating ultrafilters are plentiful, as they for instance arise in relation to the following notion (see Propositions~\ref{cnt +-} and~\ref{sc exist}). 

\begin{defn}[Definition~4.1 in \cite{BF1}]\label{neatly contracting defn}
We say that $g\in\Aut(X)$ is \emph{neatly contracting} if there exist halfspaces $\mf{h}_1$ and $\mf{h}_2$ such that ${g\mf{h}_1\cu\mf{h}_2\cu\mf{h}_1}$ and both pairs $(\mf{h}_2,\mf{h}_1^*)$ and $(g\mf{h}_1,\mf{h}_2^*)$ are strongly separated.
\end{defn}

We collect here various facts on neatly contracting automorphisms that will be needed later on.

\begin{prop}\label{cnt +-}
Every neatly contracting automorphism $g\in\Aut(X)$ has exactly two fixed points $g^{\pm}\in\partial_{\rm cnt}X$. Given any $x\in\overline{X}\setminus\{g^{\pm}\}$, we have $g^nx\ra g^+$ and $g^{-n}x\ra g^-$ for $n\ra+\infty$.
\end{prop}
\begin{proof}
This is essentially Proposition~4.3 in \cite{BF1}; we only need to prove that the points $g^{\pm}$ are represented by contracting rays. By Propositions~2.7 and~4.4 in \cite{BF1}, there exists a $\langle g\rangle$--invariant line $\g\cu X$ with endpoints $g^{\pm}$. As $\g$ must cross all hyperplanes bounding the halfspaces $g^n\mf{h}_1$, Theorem~3.9 in \cite{Genevois} shows that $\g$ is contracting. This concludes the proof.
\end{proof}

We will write $g^{\pm}_X\in\partial_{\rm cnt}X$ when it is necessary to specify the cube complex. The next proposition follows from Lemmas~2.9 and~4.7 in \cite{BF1}, although the main ingredients are actually from \cite{Fernos-Lecureux-Matheus}. The same result holds for any finite collection of (irreducible, essential) cubulations.

\begin{prop}\label{sc exist}
Let a non-virtually-cyclic group $G$ act properly and cocompactly on irreducible, essential $\CAT$ cube complexes $X$ and $Y$. There exists $g\in G$ simultaneously acting as a neatly contracting automorphism on $X$ and $Y$.
\end{prop}

Given a finitely generated group $G$, its contracting boundary\footnote{This is also known as \emph{Morse boundary}. See the discussion in the introduction for a justification of our terminology.} $\partial_cG$ was introduced in \cite{Cordes,Cashen-Mackay}; we will not endow $\partial_cG$ with any topology.

Fixing a word metric on $G$, we say that $g\in G$ is \emph{Morse} if $n\mapsto g^n$ is a Morse quasi-geodesic in $G$. This notion is independent of the chosen word metric. Every Morse element $g\in G$ fixes exactly two points $g^{\pm\infty}\in\partial_cG$ (see e.g.\ Definition~9.1, Theorem~2.2 and Theorem~9.4 in \cite{Cashen-Mackay}). 

Given an action $G\acts X$ as in the statement of Proposition~\ref{sc exist}, the Milnor--Schwarz lemma shows that orbit maps $G\ra X$ are $G$--equivariant quasi-isometries. They are at finite distance from each other and all yield the same $G$--equivariant bijection $o_X\colon\partial_cG\ra\partial_cX$. This follows from either Proposition~4.2 in \cite{Cordes} or Corollary~6.2 in \cite{Cashen-Mackay}. 

Now, if $g\in G$ acts on $X$ as a neatly contracting automorphism, we can consider the points $g^{\pm}\in\partial_{\rm cnt}X\cu\partial_cX$ introduced in Proposition~\ref{cnt +-}. Note that, in this case, $g$ is also a Morse element in $G$ by Theorem~\ref{lean vs contracting vs Morse}.

\begin{lem}\label{north south}
Let $G\acts X$ be as in Proposition~\ref{sc exist}. For every neatly contracting element $g\in G$, we have $o_X(g^{+\infty})=g^+$ and $o_X(g^{-\infty})=g^-$.
\end{lem}
\begin{proof}
The discussion above already shows that $o_X(\{g^{\pm\infty}\})=\{g^{\pm}\}$ as these are the only two $G$--fixed points in $\partial_cG$ and $\partial_cX$, respectively. Let $\partial_c^{\mc{FQ}}X$ denote the contracting boundary of $X$, endowed with the Cashen--Mackay topology (Definition~5.4 in \cite{Cashen-Mackay}). By Theorem~9.4 and Corollary~6.2 in \emph{op.\ cit.}, the element $g$ acts on $\partial_c^{\mc{FQ}}X$ with north-south dynamics. Note that $g$ also acts on $\partial_c^{\rm vis}X$ with north-south dynamics, by Proposition~\ref{cnt +-} and Theorem~\ref{properties of Phi}. Since the identity map $\partial_c^{\mc{FQ}}X\ra\partial_c^{\rm vis}X$ is continuous (Section~7 in \cite{Cashen-Mackay}), we conclude that the attracting/repelling fixed points of $g$ are the same for the two topologies.
\end{proof}

If $G$ is Gromov hyperbolic, every infinite-order element $g\in G$ is Morse and acts on the Gromov boundary $\partial_{\infty}G$ with north-south dynamics.

\section{Cubulations of hyperbolic groups.}\label{hyp sect}

This section is devoted to the proof of Theorem~\ref{ext Moeb intro}. Referring to the sketch given in the introduction, we only need to carry out Steps~(IIa) and (IIb), as Step~(I) is provided by Proposition~\ref{ample}. The crucial results are thus Proposition~\ref{maximal to maximal} in Section~\ref{traces vs cross ratios} and Theorem~\ref{all WP} in Section~\ref{concluding}. Before that, Sections~\ref{traces sect},~\ref{counting sect} and~\ref{trust sect} develop all necessary ingredients.

For the convenience of the reader, we recall here a couple of standard lemmas which do not require any action on a cube complex.

\begin{lem}\label{nowhere-dense}
Let $G$ be Gromov hyperbolic and let $H<G$ be an infinite-index quasi-convex subgroup. The limit set $\partial_{\infty}H\cu\partial_{\infty}G$ is nowhere-dense.
\end{lem}
\begin{proof}
By Lemma~2.9 in \cite{GMRS}, there exists a point $\xi\in\partial_{\infty}G\setminus\partial_{\infty}H$. We can assume that $H$ is non-elementary, otherwise the lemma is trivial. Given any infinite-order element $h\in H$, the orbit $\langle h\rangle\cdot\xi$ accumulates on the point $h^{+\infty}\in\partial_{\infty}H$. Since the action $H\acts\partial_{\infty}H$ is minimal, the subset $H\cdot h^{+\infty}$ is dense in $\partial_{\infty}H$. Thus $\partial_{\infty}H\cu\overline{H\cdot\xi}$, while $H\cdot\xi\cu\partial_{\infty}G\setminus\partial_{\infty}H$. Along with the fact that $\partial_{\infty}H$ is closed in $\partial_{\infty}G$, this concludes the proof.
\end{proof}

\begin{lem}\label{limit set commensurable}
Let $G$ be Gromov hyperbolic and let $H$ and $K$ be quasi-convex subgroups. If $H$ and $K$ have the same limit set, they are commensurable.
\end{lem}
\begin{proof}
Let $L\leq G$ denote the stabiliser of the limit set $\Lambda=\partial_{\infty}H=\partial_{\infty}K$. By Lemma~3.8 in \cite{Kapovich-Short}, the limit set of $L$ coincides with $\Lambda$. Since $H\leq L$ and $H$ is quasi-convex in $G$, Proposition~3.4 in \emph{loc.\ cit.}\ shows that $L$ is quasi-convex in $G$. Finally, Lemma~2.9 in \cite{GMRS} implies that $H$ and $K$ have finite index in $L$, hence $H\cap K$ has finite index in both $H$ and $K$.
\end{proof}

\subsection{Traces at infinity.}\label{traces sect}

Throughout this subsection:

\begin{ass}
Let $G$ be a Gromov hyperbolic group (not necessarily non-elementary) with a proper cocompact action on a $\CAT$ cube complex $X$. We denote by $o_X\colon\partial_{\infty}G\ra\partial_{\infty}X$ the only $G$--equivariant homeomorphism and by $\Phi\colon\partial X\ra\partial_{\infty}X$ the map from Corollary~\ref{properties of Phi hyp}. As usual, we will only endow $X$ with its \emph{combinatorial} metric. 
\end{ass}

Note that every (combinatorial) ray $r\cu X$ determines a point of $\partial X$, but, by hyperbolicity, it also determines a point of $\partial_{\infty}X$.

The carrier $C(\mf{w})$ of each hyperplane $\mf{w}\in\mscr{W}(X)$ is a convex subcomplex of $X$; in particular, it is itself hyperbolic and we can consider its boundary $\partial_{\infty}\mf{w}\cu\partial_{\infty}X$. A point $\xi\in\partial_{\infty}X$ lies in $\partial_{\infty}\mf{w}$ if and only if one/all rays representing $\xi$ are at finite Hausdorff distance from $\mf{w}$. 

\begin{defn}
We refer to $\partial_{\infty}\mf{w}\cu\partial_{\infty}X$ as the \emph{trace at infinity} of $\mf{w}$.
\end{defn}

Denoting by $G_{\mf{w}}\leq G$ the stabiliser of $\mf{w}$, Lemma~\ref{hyp cocpt} guarantees that $G_{\mf{w}}$ acts properly and cocompactly on $C(\mf{w})$. In particular, the homeomorphism $o_X^{-1}\colon\partial_{\infty}X\ra\partial_{\infty}G$ takes $\partial_{\infty}\mf{w}$ to the limit set $\partial_{\infty}G_{\mf{w}}\cu\partial_{\infty}G$.

We now make a few simple observations on traces at infinity.

\begin{lem}\label{traces of halfspaces}
Consider a hyperplane $\mf{w}\in\mscr{W}(X)$ and its two sides $\mf{h}$ and $\mf{h}^*$.
\begin{enumerate}
\item The subsets $\partial_{\infty}\mf{h}$ and $\partial_{\infty}\mf{h}^*$ are closed in $\partial_{\infty}X$. Moreover, we have $\partial_{\infty}X=\partial_{\infty}\mf{h}\cup\partial_{\infty}\mf{h}^*$ and $\partial_{\infty}\mf{h}\cap\partial_{\infty}\mf{h}^*=\partial_{\infty}\mf{w}$. 
\item If $X$ is essential, the sets $\partial_{\infty}\mf{h}\setminus\partial_{\infty}\mf{w}$ and $\partial_{\infty}\mf{h}^*\setminus\partial_{\infty}\mf{w}$ are nonempty.
\end{enumerate}
\end{lem} 
\begin{proof}
Since $\mf{h}$ and $\mf{h}^*$ are closed convex subsets of $X$, their boundaries are well-defined closed subsets of $\partial_{\infty}X$. Every ray in $X$ intersects either $\mf{h}$ or $\mf{h}^*$ in a sub-ray, so every point of $\partial_{\infty}X$ lies in either $\partial_{\infty}\mf{h}$ or $\partial_{\infty}\mf{h}^*$. 

It is clear that $\partial_{\infty}\mf{w}\cu\partial_{\infty}\mf{h}\cap\partial_{\infty}\mf{h}^*$. Conversely, if $\xi\in\partial_{\infty}\mf{h}\cap\partial_{\infty}\mf{h}^*$, we can consider rays $r$ and $r^*$ representing $\xi$ and entirely contained, respectively, in $\mf{h}$ and $\mf{h}^*$. The function $t\mapsto d(r(t),r^*(t))$ must be uniformly bounded and, observing that $d(r(t),C(\mf{w}))<d(r(t),\mf{h}^*)\leq d(r(t),r^*(t))$, we see that $r$ is contained in a metric neighbourhood of $C(\mf{w})$. Hence $\xi\in\partial_{\infty}\mf{w}$.

Finally, regarding part~(2), consider a sequence of points $x_n\in\mf{h}$ whose distances from $\mf{h}^*$ diverge. Lemma~\ref{hyp cocpt} ensures that $x_n$ can be chosen so that their gate-projections $\overline x_n$ to $C(\mf{w})$ lie in a compact set. The Arzel\`a--Ascoli theorem guarantees that, up to passing to a subsequence, the geodesics joining $\overline x_n$ and $x_n$ converge to a ray that is contained in $\mf{h}$ but in no metric neighbourhood of $\mf{h}^*$. 
The corresponding point of $\partial_{\infty}X$ lies in $\partial_{\infty}\mf{h}\setminus\partial_{\infty}\mf{w}$. Similarly, $\partial_{\infty}\mf{h}^*\setminus\partial_{\infty}\mf{w}$ is also nonempty.
\end{proof}

\begin{rmk}
In the setting of Lemma~\ref{traces of halfspaces}, it is interesting to observe that $\partial_{\infty}\mf{h}\setminus\partial_{\infty}\mf{w}$ is a union of connected components of $\partial_{\infty}X\setminus\partial_{\infty}\mf{w}$. This is because the partition $\partial_{\infty}X\setminus\partial_{\infty}\mf{w}=(\partial_{\infty}\mf{h}\setminus\partial_{\infty}\mf{w})\sqcup(\partial_{\infty}\mf{h}^*\setminus\partial_{\infty}\mf{w})$ consists of two disjoint closed subsets.
\end{rmk}

\begin{rmk}\label{trivial boundary}
By Lemma~\ref{hyp cocpt}, we have $\partial_{\infty}\mf{w}=\emptyset$ if and only if $\mf{w}$ is compact. If $X$ is hyperplane-essential, this means that $\mf{w}$ consists of a single point, i.e.\ $\mf{w}$ is dual to an edge that disconnects $X$.
\end{rmk}

We remind the reader that, due to Corollary~\ref{properties of Phi hyp} and Remark~\ref{topologies on cnt 2}, we feel entitled to implicitly identify the sets $\partial_{\rm nt}X\cu\partial X$ and $\Phi(\partial_{\rm nt}X)\cu\partial_{\infty}X$.

\begin{lem}\label{nt vs hyp}
A point $\xi\in\partial_{\infty}X$ lies in $\partial_{\rm nt}X$ if and only if there does not exist any $\mf{w}\in\mscr{W}(X)$ with $\xi\in\partial_{\infty}\mf{w}$. 
\end{lem}
\begin{proof}
By part~(1) of Lemma~\ref{traces of halfspaces}, a point $\xi\in\partial_{\infty}X$ lies in some trace at infinity $\partial_{\infty}\mf{w}$ if and only if there exists a halfspace $\mf{h}\in\mscr{H}(X)$ such that $\xi\in\partial_{\infty}\mf{h}\cap\partial_{\infty}\mf{h}^*$. Equivalently, there exist two rays representing $\xi$, one contained in $\mf{h}$ and the other contained in $\mf{h}^*$. Thus, $\xi$ lies in a trace at infinity if and only if there exist distinct points $x,y\in\partial X$ with $\Phi(x)=\Phi(y)=\xi$. On the other hand, a point $\xi\in\partial_{\infty}X$ is non-terminating if and only if $\Phi^{-1}(\xi)$ is a singleton (cf.\ Theorem~\ref{properties of Phi} and Definition~\ref{nt defn}).
\end{proof}

In the rest of the subsection, we collect various facts that will be needed in the proof of Step~(IIa). They key ingredient is Proposition~\ref{generics exist} below.

Given hyperplanes $\mf{u},\mf{w}\in\mscr{W}(X)$, we denote by $B(\mf{u},\mf{w})\cu X$ the bridge associated to their carriers. When $\mf{u}$ and $\mf{w}$ are transverse, $B(\mf{u},\mf{w})$ is simply the intersection of the two carriers. 

\begin{prop}\label{intersections of traces}
Given $\mf{u},\mf{w}\in\mscr{W}(X)$, the group $G_{\mf{u}}\cap G_{\mf{w}}$ acts properly and cocompactly on $B(\mf{u},\mf{w})$. Furthermore, $\partial_{\infty}B(\mf{u},\mf{w})=\partial_{\infty}\mf{u}\cap\partial_{\infty}\mf{w}$.
\end{prop}
\begin{proof}
Let $S$ denote the intersection of $B=B(\mf{u},\mf{w})$ with the carrier of $\mf{u}$ and consider its stabiliser $G_S\leq G$. The action $G_S\acts S$ is cocompact by Proposition~2.7 in \cite{Hagen-Susse}. There exists a constant $D>0$ such that $S$ is contained in the $D$--neighbourhood of both $\mf{u}$ and $\mf{w}$. The finitely many hyperplanes that contain $S$ in their $D$--neighbourhood are permuted by $G_S$ and it follows that $G_{\mf{u}}\cap G_{\mf{w}}$ sits in $G_S$ as a finite-index subgroup. Hence $G_{\mf{u}}\cap G_{\mf{w}}$ acts cocompactly on $S$ and it also acts cocompactly on the bridge $B$, which is at finite Hausdorff distance from $S$. Finally, as $B\cu X$ is a subcomplex and $G\acts X$ is proper, the action $G_{\mf{u}}\cap G_{\mf{w}}\acts B$ is also proper.

It is clear that $\partial_{\infty}B=\partial_{\infty}S\cu\partial_{\infty}\mf{u}$. The same argument applied to the other shore of $B$ shows that $\partial_{\infty}B\cu\partial_{\infty}\mf{u}\cap\partial_{\infty}\mf{w}$. Now, consider a point $\xi\in\partial_{\infty}\mf{u}\setminus\partial_{\infty}S$ and a geodesic $\g\cu C(\mf{u})$ representing $\xi$; note that the function $t\mapsto d(\g(t),S)$ must diverge with $t$. As soon as $\g(t)\not\in S$, equation~\eqref{distance formula} in Section~\ref{CCC prelims} implies that $d(\g(t),C(\mf{w}))=d(\g(t),S)+d(S,C(\mf{w}))$. In particular, the function $t\mapsto d(\g(t),C(\mf{w}))$ must also diverge and $\xi\not\in\partial_{\infty}\mf{w}$. This shows that we have $\partial_{\infty}\mf{u}\cap\partial_{\infty}\mf{w}\cu\partial_{\infty}S=\partial_{\infty}B$, concluding the proof.
\end{proof}

We now introduce a preorder on $\mscr{W}(X)$. We write $\mf{u}_1\preceq\mf{u}_2$ if $\partial_{\infty}\mf{u}_1\cu\partial_{\infty}\mf{u}_2$ and $\mf{u}_1\sim\mf{u}_2$ if $\partial_{\infty}\mf{u}_1=\partial_{\infty}\mf{u}_2$. The latter is an equivalence relation.

Recall that $\mscr{W}(\mf{w})$ denotes the set of hyperplanes transverse to $\mf{w}\in\mscr{W}(X)$.

\begin{prop}\label{on the preorder}
Let $X$ be hyperplane-essential.
\begin{enumerate}
\item We have $\mf{u}_1\preceq\mf{u}_2$ if and only if $\mscr{W}(\mf{u}_1)\cu\mscr{W}(\mf{u}_2)$. Equivalently, $\mf{u}_1$ and $\mf{u}_2$ are not transverse and $\mf{u}_1$ stays at bounded distance from $\mf{u}_2$.
\item If there exists $\mf{w}$ with $\mf{u}_1\preceq\mf{w}$ and $\mf{u}_2\preceq\mf{w}$, the hyperplanes $\mf{u}_1$ and $\mf{u}_2$ are not transverse.
\item There exists $N=N(X)\geq 0$ such that, given any hyperplane $\mf{u}$ with ${\partial_{\infty}\mf{u}\neq\emptyset}$, there exist at most $N$ hyperplanes $\mf{w}$ with $\mf{u}\preceq\mf{w}$.
\end{enumerate}
\end{prop}
\begin{proof}
We start by proving part~(1). Let $B$ denote the bridge $B(\mf{u}_1,\mf{u}_2)$ and let $S_1$ be the shore $B\cap C(\mf{u}_1)$. If $\mscr{W}(\mf{u}_1)\cu\mscr{W}(\mf{u}_2)$, we have $\mf{u}_2\not\in\mscr{W}(\mf{u}_1)$ and the carrier $C(\mf{u}_1)$ is at finite Hausdorff distance from $B$. Thus, $\mf{u}_1$ is contained in a metric neighbourhood of $\mf{u}_2$ and this implies that $\mf{u}_1\preceq\mf{u}_2$.

Conversely, suppose that $\mf{u}_1\preceq\mf{u}_2$. Proposition~\ref{intersections of traces} implies $\partial_{\infty}\mf{u}_1=\partial_{\infty}S_1$. Since $S_1$ is a convex subcomplex of $C(\mf{u}_1)$, we will also denote by $S_1$ its projection to a convex subcomplex of the $\CAT$ cube complex $\mf{u}_1$. Note that $S_1=\mf{u}_1$, or there would exist a halfspace $\mf{k}$ of $\mf{u}_1$ with $S_1\cu\mf{k}$. In this case, applying part~(2) of Lemma~\ref{traces of halfspaces} to $\mf{u}_1$, we would obtain a point of $\partial_{\infty}\mf{k}^*\setminus\partial_{\infty}\mf{k}\cu\partial_{\infty}\mf{u}_1\setminus\partial_{\infty}S_1$, a contradiction. The fact that $S_1=\mf{u}_1$ yields $\mscr{W}(\mf{u}_1)=\mscr{W}(S_1)=\mscr{W}(\mf{u}_1)\cap\mscr{W}(\mf{u}_2)\cu\mscr{W}(\mf{u}_2)$, completing part~(1).

If $\mf{u}_1$ and $\mf{u}_2$ are transverse and $\mf{u}_1\preceq\mf{w}$, it follows from part~(1) that $\mf{u}_2$ and $\mf{w}$ are transverse. In particular $\mf{u}_2\not\preceq\mf{w}$ and this proves part~(2).

We finally address part~(3). By hyperbolicity of $X$, there exists a constant $D=D(X)>0$ such that, for any two hyperplanes $\mf{w}_1,\mf{w}_2$ with $\partial_{\infty}B(\mf{w}_1,\mf{w}_2)\neq\emptyset$, we have $d(C(\mf{w}_1),C(\mf{w}_2))\leq D$. If $\mf{u}\preceq\mf{w}$, Proposition~\ref{intersections of traces} ensures that $\partial_{\infty}B(\mf{u},\mf{w})=\partial_{\infty}\mf{u}\neq\emptyset$, hence $d(C(\mf{u}),C(\mf{w}))\leq D$. Part~(1) now shows that $C(\mf{u})$ is contained in the $(D+1)$--neighbourhood of $C(\mf{w})$. Since $X$ admits a proper cocompact action, it is uniformly locally finite and there exist only uniformly finitely many such hyperplanes $\mf{w}$.
\end{proof}

\begin{defn}
Given a hyperplane $\mf{w}$, a point $\xi\in\partial_{\infty}\mf{w}$ is \emph{generic} if the only hyperplanes $\mf{u}$ with $\xi\in\partial_{\infty}\mf{u}$ are those that satisfy $\mf{w}\preceq\mf{u}$.
\end{defn}

\begin{prop}\label{generics exist}
For every $\mf{w}\in\mscr{W}(X)$, the subset of generic points is dense in $\partial_{\infty}\mf{w}$. In particular, generic points exist as soon as $\partial_{\infty}\mf{w}\neq\emptyset$.
\end{prop}
\begin{proof}
Let $\mscr{B}$ be the family of subsets of $\partial_{\infty}\mf{w}$ of the form $\partial_{\infty}\mf{w}\cap\partial_{\infty}\mf{u}$, with $\mf{u}\in\mscr{W}(X)$ and $\mf{w}\not\preceq\mf{u}$. By definition, a point $\xi\in\partial_{\infty}\mf{w}$ is generic if and only if it does not lie in the union of the elements of $\mscr{B}$. Since $\mscr{W}(X)$ is countable, so is $\mscr{B}$. The proposition then follows from Baire's theorem, if we show that every $\mf{B}\in\mscr{B}$ is nowhere-dense in $\partial_{\infty}\mf{w}$.

If $\mf{B}=\partial_{\infty}\mf{w}\cap\partial_{\infty}\mf{u}$, set $H=G_{\mf{w}}\cap G_{\mf{u}}$. Proposition~\ref{intersections of traces} shows that the homeomorphism $o_X\colon\partial_{\infty}G\ra\partial_{\infty}X$ takes $\partial_{\infty}H\cu\partial_{\infty}G_{\mf{w}}$ to $\mf{B}\cu\partial_{\infty}\mf{w}$. Since $\mf{w}\not\preceq\mf{u}$, the difference $\partial_{\infty}\mf{w}\setminus\mf{B}$ is nonempty and $H$ must have infinite index in $G_{\mf{w}}$. As $H$ is quasi-convex, we conclude via Lemma~\ref{nowhere-dense}.
\end{proof}

\begin{rmk}\label{generic vs nt}
Assume that $X$ is hyperplane-essential and consider a generic point $\xi\in\partial_{\infty}\mf{w}$. Viewing $\mf{w}$ itself as a $\CAT$ cube complex, we have $\xi\in\partial_{\rm nt}\mf{w}$. This follows from Lemma~\ref{nt vs hyp} since, if $\mf{u}\in\mscr{W}(\mf{w})$, part~(1) of Proposition~\ref{on the preorder} shows that $\mf{w}\not\preceq\mf{u}$ and hence $\xi\not\in\partial_{\infty}\mf{u}$. 
\end{rmk}

This is a good point to make the following simple observation, which will only be needed in the proof of Proposition~\ref{delta preserved} later on.

\begin{lem}\label{ss in quasi-lines}
Let $X$ be essential and suppose that $G=\langle g\rangle\simeq\Z$. Then:
\begin{enumerate}
\item every hyperplane of $X$ is compact;
\item for every $\mf{h}\in\mscr{H}(X)$, there exists $N>0$ such that $g^N\mf{h}$ and $\mf{h}^*$ are strongly separated.
\end{enumerate}
\end{lem}
\begin{proof}
Since $\partial_{\infty}X$ contains only two points $\xi$ and $\eta$, part~(2) of Lemma~\ref{traces of halfspaces} shows that $\partial_{\infty}\mf{w}=\emptyset$ for every $\mf{w}\in\mscr{W}(X)$. Lemma~\ref{hyp cocpt} then yields part~(1).

Given any halfspace $\mf{h}$, we have $\partial_{\infty}\mf{h}\cap\partial_{\infty}\mf{h}^*=\emptyset$. Up to swapping $\xi$ and $\eta$, we have $\Phi^{-1}(\xi)\cu\mf{h}$ and $\Phi^{-1}(\eta)\cu\mf{h}^*$. Squaring $g$ if necessary, we can assume that $g$ fixes $\xi$ and $\eta$. Since $X$ admits a cocompact action, it is finite dimensional and there exists $n>0$ such that $g^n\mf{h}$ and $\mf{h}$ are not transverse. Observe that $g^n\mf{h}\cap\mf{h}$ and $g^n\mf{h}^*\cap\mf{h}^*$ are both nonempty, as they contain $\Phi^{-1}(\xi)$ and $\Phi^{-1}(\eta)$, respectively. Replacing $g$ with its inverse if necessary, we can assume that $g^n\mf{h}\cu\mf{h}$. As the hyperplane bounding $\mf{h}$ is compact, there exist only finitely many hyperplanes transverse to $\mf{h}$; they are all compact. Choosing a sufficiently large $k>0$, we can thus ensure that none of them is transverse to $g^{kn}\mf{h}\subsetneq\mf{h}$, hence $g^{kn}\mf{h}$ and $\mf{h}^*$ are strongly separated.
\end{proof}

\subsection{Towards hyperplane recognition.}\label{counting sect}

This subsection is devoted to Pro\-positions~\ref{pts lie in hyp} and~\ref{simplification}. The latter, in particular, will be our main tool in overcoming the difficulties, described in the introduction, regarding the passage from Step~(IIa) to Step~(IIb).

\begin{ass}
Let again $G$ be Gromov hyperbolic and let the action $G\acts X$ be proper and cocompact. Let $o_X\colon\partial_{\infty}G\ra\partial_{\infty}X$ denote the only $G$--equivariant homeomorphism. 
\end{ass}

\begin{prop}\label{pts lie in hyp}
Consider two distinct points ${\xi,\eta\in\partial_{\infty}X}$ and four sequences $x_n,y_n,z_n,w_n\in\partial_{\rm nt}X$, where $x_n$ and $z_n$ converge to $\xi$ while $y_n$ and $w_n$ converge to $\eta$. There exists $N\geq 0$ such that, for every $n\geq N$ and every $\mf{w}\in\mscr{W}(x_n,y_n|z_n,w_n)$, the points $\xi$ and $\eta$ lie in $\partial_{\infty}\mf{w}$.
\end{prop}
\begin{proof}
Consider a geodesic $\alpha_n$ whose endpoints in $\partial X$ are $x_n$ and $y_n$; similarly, let $\beta_n$ be a geodesic with endpoints $z_n$ and $w_n$. We also pick a geodesic $\g$ whose endpoints in $\partial_{\infty}X$ are precisely $\xi$ and $\eta$. The convergence of the four sequences yields a constant $D>0$ such that, for every point $p\in\g$, there exists $N\geq 0$ and points $a_n\in\alpha_n$, $b_n\in\beta_n$ with $d(a_n,p)\leq D$ and $d(b_n,p)\leq D$ for all $n\geq N$. Every hyperplane $\mf{w}$ in $\mscr{W}(x_n,y_n|z_n,w_n)$ separates $\alpha_n$ and $\beta_n$ and must thus satisfy ${d(p,\mf{w})\leq D}$. 

Since there are only finitely many hyperplanes at distance at most $D$ from $p$, we are only left to show that every element of $\limsup\mscr{W}(x_n,y_n|z_n,w_n)$ contains $\xi$ and $\eta$ in its trace at infinity. Let us consider a hyperplane $\mf{w}$ that separates $\alpha_{n_k}$ and $\beta_{n_k}$ for a diverging sequence of integers $n_k$. Up to passing to a further subsequence, the Arzel\`a--Ascoli theorem allows us to assume that the geodesics $\alpha_{n_k}$ converge locally uniformly to a geodesic $\alpha$. The discussion above then shows that $\alpha$ is contained in a metric neighbourhood of $C(\mf{w})$ and has endpoints $\xi$ and $\eta$ in $\partial_{\infty}X$. Hence $\xi$ and $\eta$ lie in $\partial_{\infty}\mf{w}$.
\end{proof}

\begin{lem}\label{W(k)}
For an infinite-order element $k\in G$ and a hyperplane $\mf{w}$, the following are equivalent:
\begin{enumerate}
\item a non-trivial power of $k$ preserves $\mf{w}$;
\item the points $o_X(k^{\pm\infty})$ lie in $\partial_{\infty}\mf{w}$.
\end{enumerate}
Fixing $k$, there are only finitely many hyperplanes satisfying these conditions.
\end{lem}
\begin{proof}
Since $X$ is hyperbolic, there exists a constant $D=D(X)>0$ such that any two geodesic lines in $X$ with the same endpoints in $\partial_{\infty}X$ are at Hausdorff distance at most $D$. If $\g\cu X$ is a geodesic with $o_X(k^{\pm\infty})$ as endpoints at infinity, every hyperplane satisfying condition~(2) must contain $\g$ in its $D$--neighbourhood. It follows that only finitely many hyperplanes of $X$ satisfy condition~(2). If $\mf{w}$ is such a hyperplane, $k^n\mf{w}$ also satisfies condition~(2) for all $n>0$, so we must have $k^n\mf{w}=\mf{w}$ for some $n>0$. Conversely, if $k^n\in G_{\mf{w}}$ for some $n>0$, we have $k^{\pm\infty}=(k^n)^{\pm\infty}\in\partial_{\infty}G_{\mf{w}}$. 
\end{proof}

\begin{defn}\label{good defn}
If $k\in G$ is infinite-order, we denote by $\mscr{W}(k)$ the set of hyperplanes satisfying the equivalent conditions in Lemma~\ref{W(k)}. We say that $k$ is \emph{good} if it preserves every halfspace bounded by an element of $\mscr{W}(k)$.
\end{defn}

\begin{rmk}\label{goods exist}
Every infinite-order element has a good power, as the set $\mscr{W}(k)$ is always finite by Lemma~\ref{W(k)}.
\end{rmk}

We are interested in good elements because of the following result.

\begin{prop}\label{simplification}
Consider distinct points $x,y\in\partial_{\rm nt}X$ and a good infinite-order element $k\in G$ that fixes neither of them. There exists $N\geq 0$ such that, for every $n\geq N$, we have:
\[\mscr{W}(x|y)\cap\mscr{W}(k)=\mscr{W}(k^nx, k^{-n}x|k^ny, k^{-n}y).\]
\end{prop}
\begin{proof}
Let us set $\mscr{W}_n=\mscr{W}(k^nx, k^{-n}x|k^ny, k^{-n}y)$ for the sake of simplicity. Since $\langle k\rangle$ acts with north-south dynamics on $\partial_{\infty}X$, the points $k^nx$ and $k^ny$ converge to $o_X(k^{+\infty})$, while $k^{-n}x$ and $k^{-n}y$ tend to $o_X(k^{-\infty})$ for $n\ra +\infty$. Proposition~\ref{pts lie in hyp} then yields an integer $N\geq 0$ such that $\mscr{W}_n\cu\mscr{W}(k)$ for all $n\geq N$. On the other hand, it is clear from the fact that $k$ is good that $\mscr{W}_n\cap\mscr{W}(k)=\mscr{W}(x|y)\cap\mscr{W}(k)$ and this concludes the proof.
\end{proof}

\subsection{Trust issues.}\label{trust sect}

Throughout this subsection:

\begin{ass}
Let the Gromov hyperbolic group $G$ act properly and cocompactly on $\CAT$ cube complexes $X$ and $Y$. We fix two subsets $\mc{A}\cu\partial_{\rm nt}X$ and $\mc{B}\cu\partial_{\rm nt}Y$. 
\end{ass}

Given $\mc{U}\cu\mscr{W}(X)$, we employ the notation $\mc{U}(A|B)=\mscr{W}(A|B)\cap\mc{U}$ and
\[\Cr_{\mc{U}}(x,y,z,w)=\#\mc{U}(x,z|y,w)-\#\mc{U}(x,w|y,z)\]
for all subsets $A,B\cu\overline X$ and pairwise distinct points $x,y,z,w\in\partial_{\rm nt}X$. Given $\mc{V}\cu\mscr{W}(Y)$, the same notation applies to subsets of $\overline Y$ and points of $\partial_{\rm nt}Y$.

As mentioned in the introduction, a key complication in the proof of Theorem~\ref{ext Moeb intro} is the fact that $\Cr_{\mc{U}}(x,y,z,w)$ does not provide any direct information on $\#\mc{U}(x,z|y,w)$ and $\#\mc{U}(x,w|y,z)$, only on their difference. The following notion is devised to address this problem.

\begin{defn}
We say that a $4$--tuple $(a,b,c,d)\in(\partial_{\rm nt}X)^4$ is \emph{$\mc{U}$--trustworthy} if at least one among the sets $\mc{U}(a,b|c,d)$, $\mc{U}(a,c|b,d)$ and $\mc{U}(a,d|b,c)$ is empty. If $\mc{U}=\mscr{W}(X)$, we just say that $(a,b,c,d)$ is \emph{trustworthy}.
\end{defn}

\begin{lem}\label{trust explained}
Given subsets $\mc{U}\cu\mscr{W}(X)$ and $\mc{V}\cu\mscr{W}(Y)$, consider pairwise distinct points $x_1,x_2,x_3,x_4\in\mc{A}$ and a bijection $f\colon\mc{A}\ra\mc{B}$ satisfying:
\[\Cr_{\mc{U}}\left(x_{\s(1)},x_{\s(2)},x_{\s(3)},x_{\s(4)}\right)=\Cr_{\mc{V}}\left(f(x_{\s(1)}),f(x_{\s(2)}),f(x_{\s(3)}),f(x_{\s(4)})\right)\]
for every permutation $\s\in\mf{S}_4$. If the 4--tuples $(f(x_1),f(x_2),f(x_3),f(x_4))$ and $(x_1,x_2,x_3,x_4)$ are, respectively, $\mc{V}$--trustworthy and $\mc{U}$--trust\-worthy, then, for every $\s\in\mf{S}_4$, we have:
\[\#\mc{U}\left(x_{\s(1)},x_{\s(2)}| x_{\s(3)},x_{\s(4)}\right)=\#\mc{V}\left(f(x_{\s(1)}),f(x_{\s(2)})|f(x_{\s(3)}),f(x_{\s(4)})\right).\]
\end{lem}
\begin{proof}
Since $(x_1,x_2,x_3,x_4)$ is $\mc{U}$--trustworthy, we can permute the four points so that $\mc{U}(x_1,x_2|x_3,x_4)=\emptyset$. This implies that $\Cr_{\mc{U}}(x_1,x_3,x_4,x_2)\geq 0$ and $\Cr_{\mc{U}}(x_1,x_4,x_3,x_2)\geq 0$. We then have $\Cr_{\mc{V}}(f(x_1),f(x_3),f(x_4),f(x_2))\geq 0$ and $\Cr_{\mc{V}}(f(x_1),f(x_4),f(x_3),f(x_2))\geq 0$, which imply that the cardinality $\#\mc{V}(f(x_1),f(x_2)|f(x_3),f(x_4))$ is at most as large as the minimum between $\#\mc{V}(f(x_1),f(x_3)|f(x_2),f(x_4))$ and $\#\mc{V}(f(x_1),f(x_4)|f(x_2),f(x_3))$. 

Since the $4$--tuple $(f(x_1),f(x_2),f(x_3),f(x_4))$ is $\mc{V}$--trustworthy, this means that $\mc{V}(f(x_1),f(x_2)|f(x_3),f(x_4))=\emptyset$. We thus have:
\begin{align*}
\#\mc{U}&(x_1,x_3|x_2,x_4)=\Cr_{\mc{U}}(x_1,x_4,x_3,x_2) \\
& =\Cr_{\mc{V}}(f(x_1),f(x_4),f(x_3),f(x_2))=\#{\mc{V}}(f(x_1),f(x_3)|f(x_2),f(x_4)).
\end{align*}
The equality between $\#\mc{U}(x_1,x_4|x_2,x_3)$ and $\#{\mc{V}}(f(x_1),f(x_4)|f(x_2),f(x_3))$ is obtained similarly.
\end{proof}

The following will be our main source of trustworthy $4$--tuples.

\begin{defn}\label{trustworthy pairs defn}
Let $\xi$ and $\eta$ be distinct points of $\partial_{\infty}X$. We say that the pair $(\xi,\eta)$ is \emph{trustworthy} if there do not exist transverse hyperplanes $\mf{u}_1$ and $\mf{u}_2$ with $\xi,\eta\in\partial_{\infty}\mf{u}_1\cap\partial_{\infty}\mf{u}_2$.
\end{defn}

\begin{lem}\label{trust lemma}
Consider two distinct points $\xi,\eta\in\partial_{\infty}X$ and sequences ${a_n,b_n,c_n,d_n\in\partial_{\rm nt}X}$, where $a_n$ and $c_n$ converge to $\xi$, while $b_n$ and $d_n$ converge to $\eta$. If $(\xi,\eta)$ is trustworthy, then there exists $N$ such that $(a_n,b_n,c_n,d_n)$ is trustworthy for all $n\geq N$.
\end{lem}
\begin{proof}
If $(a_n,b_n,c_n,d_n)$ is not trustworthy for infinitely many values of $n$, we can pass to a subsequence in order to ensure that the sets $\mscr{W}(a_n,b_n|c_n,d_n)$ and $\mscr{W}(a_n,d_n|b_n,c_n)$ are all nonempty. Proposition~\ref{pts lie in hyp} provides $N$ such that, for all $n\geq N$, every element of $\mscr{W}(a_n,b_n|c_n,d_n)\cup\mscr{W}(a_n,d_n|b_n,c_n)$ contains $\xi$ and $\eta$ in its trace at infinity. Lemma~\ref{three pwt}, however, shows that the sets $\mscr{W}(a_n,b_n|c_n,d_n)$ and $\mscr{W}(a_n,d_n|b_n,c_n)$ are transverse, contradicting the fact that $(\xi,\eta)$ is trustworthy.
\end{proof}

The next result applies e.g.\ to the case when no three elements of $\mc{U}$ are transverse (Lemma~\ref{three pwt}). It will only be needed in the proof of Theorem~\ref{all WP}.

\begin{lem}\label{all trustworthy}
Suppose that, for a subset $\mc{U}\cu\mscr{W}(X)$, every element of $\mc{A}^4$ is $\mc{U}$--trustworthy. Consider a partition $\mc{A}=\mc{P}\sqcup\mc{Q}$ with $\#\mc{P},\#\mc{Q}\geq 2$ and such that $\mc{U}(x,y|z,w)\neq\emptyset$ for all $x,y\in\mc{P}$ and $z,w\in\mc{Q}$. Then, there exist $\overline x,\overline y\in\mc{P}$ and $\overline z,\overline w\in\mc{Q}$ such that $\mc{U}(\overline x,\overline y|\overline z,\overline w)=\mc{U}(\mc{P}|\mc{Q})$. 
\end{lem}
\begin{proof}
Pick points $\overline x,\overline y\in\mc{P}$ and $\overline z,\overline w\in\mc{Q}$ so as to minimise the cardinality of $\mc{U}(\overline x,\overline y|\overline z,\overline w)$; since $\#\mc{P},\#\mc{Q}\geq 2$, this set is finite. It is clear that $\mc{U}(\overline x,\overline y|\overline z,\overline w)$ contains $\mc{U}(\mc{P}|\mc{Q})$. Consider any hyperplane $\mf{w}\in\mc{U}(\overline x,\overline y|\overline z,\overline w)$ and let $\mf{h}$ be its side containing $\overline x$ and $\overline y$. We are going to show that $\mf{h}\cap\mc{A}\cu\mc{P}$ and the same argument will yield $\mf{h}^*\cap\mc{A}\cu\mc{Q}$. This will conclude the proof as then $\mf{h}\cap\mc{A}=\mc{P}$ and $\mf{h}^*\cap\mc{A}=\mc{Q}$, which shows that $\mf{w}$ separates $\mc{P}$ and $\mc{Q}$.

Suppose for the sake of contradiction that a point $u\in\mf{h}\cap\mc{A}$ lies in $\mc{Q}$. By hypothesis, there exist hyperplanes $\mf{u}_1\in\mc{U}(\overline x,\overline y|u,\overline z)$ and ${\mf{u}_2\in\mc{U}(\overline x,\overline y|u,\overline w)}$. As $\mc{U}(\overline x,\overline y|u,\overline z)$ and $\mc{U}(\overline x,\overline y|u,\overline w)$ do not contain $\mf{w}$ and cannot have fewer elements than $\mc{U}(\overline x,\overline y|\overline z,\overline w)$, we must be able to choose $\mf{u}_1$ and $\mf{u}_2$ outside $\mc{U}(\overline x,\overline y|\overline z,\overline w)$. In conclusion:
\[\mf{w}\in\mc{U}(\overline x,\overline y,u|\overline z,\overline w)\cu\mc{U}(\overline y,u|\overline z,\overline w),\]
\[\mf{u}_1\in\mc{U}(\overline x,\overline y,\overline w|u,\overline z)\cu\mc{U}(\overline y,\overline w|u,\overline z),\]
\[\mf{u}_2\in\mc{U}(\overline x,\overline y,\overline z|u,\overline w)\cu\mc{U}(\overline y,\overline z|u,\overline w),\]
which violates the assumption that every $4$--tuple, in particular $(u,\overline y,\overline z,\overline w)$, be $\mc{U}$--trustworthy. 
\end{proof}

\subsection{Traces vs cross ratios.}\label{traces vs cross ratios}

Throughout this subsection:

\begin{ass}
We now assume that the hyperbolic group $G$ is non-elementary. We consider proper cocompact actions of $G$ on essential, hyperplane-essential $\CAT$ cube complexes $X$ and $Y$. These are irreducible by Remark~\ref{hyp vs irr}. 

Let $\mc{A}\cu\partial_{\rm nt}X$ and $\mc{B}\cu\partial_{\rm nt}Y$ be nonempty $G$--invariant subsets with $f(\mc{A})=\mc{B}$, where $f$ is the homeomorphism $o_Y\o o_X^{-1}\colon\partial_{\infty}X\ra\partial_{\infty}Y$.

Throughout the subsection, we also fix subsets $\mc{U}\cu\mscr{W}(X)$ and $\mc{V}\cu\mscr{W}(Y)$ such that:
\[\Cr_{\mc{U}}(x,y,z,w)=\Cr_{\mc{V}}(f(x),f(y),f(z),f(w))\]
for all pairwise distinct points $x,y,z,w\in\mc{A}$. It will become clear in Section~\ref{concluding} how this relates to the proof of Theorem~\ref{ext Moeb intro}.
\end{ass}

The reader can think of the case where $\mc{U}=\mscr{W}(X)$ and $\mc{V}=\mscr{W}(Y)$, although we will need the full generality of the previous setup in Section~\ref{concluding}.

\begin{lem}\label{hyp pairs are preserved}
Given a hyperplane $\mf{u}\in\mc{U}$, a generic point $\xi\in\partial_{\infty}\mf{u}$ and an arbitrary point $\eta\in\partial_{\infty}\mf{u}\setminus\{\xi\}$, there exists a hyperplane $\mf{v}\in\mc{V}$ with $f(\xi),f(\eta)\in\partial_{\infty}\mf{v}$.
\end{lem}
\begin{proof}
Consider a (combinatorial) line $\g\cu C(\mf{u})$ with endpoints $\xi$ and $\eta$. Pick $\mf{w}_0\in\mscr{W}(\g)$ and label the other elements of $\mscr{W}(\g)$ as $\mf{w}_n$, $n\in\Z$, according to the order in which they are crossed by $\g$ and so that positive indices correspond to the half of $\g$ ending at $\xi$. Let $\mf{h}$ be any side of the hyperplane $\mf{u}$ and let $\mf{h}_n$ be the side of $\mf{w}_n$ that contains the positive half of $\g$. For all $n>0$, part~(2) of Proposition~\ref{ample} allows us to pick points of $\mc{A}$ as follows: $x_n\in\mf{h}\cap\mf{h}_n$, $y_n\in\mf{h}\cap\mf{h}_{-n}^*$, $z_n\in\mf{h}^*\cap\mf{h}_n$ and $w_n\in\mf{h}^*\cap\mf{h}_{-n}^*$. 

Note that every limit point of the $x_n$ within $\partial X$ must have infinite Gromov product with $\g(+\infty)\in\partial X$. Since $\Phi(\g(+\infty))=\xi$, Lemma~\ref{bounded components of cu} and Corollary~\ref{properties of Phi hyp} show that $x_n\in\partial_{\rm nt}X$ converge to $\xi$. Similarly, $z_n$ converge to $\xi$, while $y_n$ and $w_n$ converge to $\eta$.

By Remark~\ref{generic vs nt} and Lemma~\ref{regular vs NT}, the generic point $\xi\in\partial_{\infty}\mf{u}$ is represented by a regular point in the Roller boundary of the $\CAT$ cube complex $\mf{u}$. For all sufficiently large $n>0$, it follows that $\mf{w}_n\cap\mf{u}$ and $\mf{w}_{-n}\cap\mf{u}$ are strongly separated as hyperplanes of the cube complex $\mf{u}$. In other words, no hyperplane of $X$ is transverse to $\mf{u}\in\mc{U}(x_n,y_n|z_n,w_n)$ and $\mf{w}_n,\mf{w}_{-n}\in\mscr{W}(x_n,z_n|y_n,w_n)$ at the same time. By Lemma~\ref{three pwt}, the sets $\mscr{W}(x_n,w_n|y_n,z_n)$ are then all empty for large $n>0$. It follows that $\Cr_{\mc{U}}(x_n,z_n,y_n,w_n)>0$, hence $\Cr_{\mc{V}}(f(x_n),f(z_n),f(y_n),f(w_n))>0$ and $\mc{V}(f(x_n),f(y_n)|f(z_n),f(w_n))\neq\emptyset$. Since $f(x_n)$ and $f(z_n)$ converge to $f(\xi)$, while $f(y_n)$ and $f(w_n)$ converge to $f(\eta)$, Proposition~\ref{pts lie in hyp} yields the required hyperplane $\mf{v}$.
\end{proof}

\begin{defn}
A hyperplane $\mf{w}\in\mc{U}$ is \emph{$\mc{U}$--boundary-maximal} if $\partial_{\infty}\mf{w}\neq\emptyset$ and every hyperplane $\mf{u}\in\mc{U}$ with $\mf{w}\preceq\mf{u}$ actually satisfies $\mf{u}\sim\mf{w}$. When $\mc{U}=\mscr{W}(X)$, we simply speak of \emph{boundary-maximal} hyperplanes.
\end{defn}

Part~(3) of Proposition~\ref{on the preorder} and Remark~\ref{trivial boundary} show that boundary-maximal hyperplanes exist as soon as $X$ is not a tree. In fact, for every hyperplane $\mf{u}$ there exists a boundary-maximal hyperplane $\mf{w}$ with $\mf{u}\preceq\mf{w}$. 

The next result can be viewed as Step~(IIa) from the introduction.

\begin{prop}\label{maximal to maximal}
If $\mf{u}\in\mc{U}$ is $\mc{U}$--boundary-maximal, there exists a $\mc{V}$--boun\-dary-maximal hyperplane $\mf{v}\in\mc{V}$ with $\partial_{\infty}\mf{v}=f(\partial_{\infty}\mf{u})$.
\end{prop}
\begin{proof}
Proposition~\ref{generics exist} allows us to pick a generic point $\xi\in\partial_{\infty}\mf{u}$. Note that $\mf{u}$ is acted on properly and cocompactly by its stabiliser (Lemma~\ref{hyp cocpt}), so $\partial_{\infty}\mf{u}\setminus\{\xi\}\neq\emptyset$; in particular, Lemma~\ref{hyp pairs are preserved} shows that there exists a hyperplane $\mf{v}\in\mc{V}$ with $f(\xi)\in\partial_{\infty}\mf{v}$. By part~(3) of Proposition~\ref{on the preorder}, we can take $\mf{v}$ to be $\mc{V}$--boundary-maximal. Let $\zeta\in\partial_{\infty}\mf{v}\setminus\{f(\xi)\}$ be generic.

Since $f(\xi)$ and $\zeta$ both lie in $\partial_{\infty}\mf{v}$, Lemma~\ref{hyp pairs are preserved} shows that $\xi$ and $f^{-1}(\zeta)$ lie in the trace at infinity of some $\mf{u}'\in\mc{U}$. Since $\xi$ is generic, we have $\mf{u}\preceq\mf{u}'$ and, as $\mf{u}$ is $\mc{U}$--boundary-maximal, we conclude that $\partial_{\infty}\mf{u}=\partial_{\infty}\mf{u}'$. In particular $f^{-1}(\zeta)\in\partial_{\infty}\mf{u}$, showing that the closed subset $f(\partial_{\infty}\mf{u})$ contains every generic point in $\partial_{\infty}\mf{v}\setminus\{f(\xi)\}$. By Proposition~\ref{generics exist}, we have $\partial_{\infty}\mf{v}\cu f(\partial_{\infty}\mf{u})$ and we will now obtain the opposite inclusion with a similar argument.

Given generic points $\zeta\in\partial_{\infty}\mf{v}$ and $\eta\in\partial_{\infty}\mf{u}\setminus\{f^{-1}(\zeta)\}$, we apply Lem\-ma~\ref{hyp pairs are preserved} to the points $\eta,f^{-1}(\zeta)\in\partial_{\infty}\mf{u}$. This shows that $f(\eta)$ and $\zeta$ lie in the trace at infinity of an element of $\mc{V}$, which, by genericity of $\zeta$ and $\mc{V}$--boundary-maximality of $\mf{v}$, can be taken to coincide with $\mf{v}$. Hence $f(\eta)\in\partial_{\infty}\mf{v}$ and, by density of generic points, we have $f(\partial_{\infty}\mf{u})\cu\partial_{\infty}\mf{v}$.
\end{proof}

We now prepare for Step~(IIb), which will be completed in Theorem~\ref{all WP}. Given a hyperplane $\mf{w}\in\mscr{W}(X)$ with $\partial_{\infty}\mf{w}\neq\emptyset$, we denote by $\mscr{T}_{\mc{U}}(\mf{w})$ the set of hyperplanes $\mf{u}\in\mc{U}$ satisfying $\mf{u}\sim\mf{w}$. We also write $\wt{\mscr{T}_{\mc{U}}}(\mf{w})\cu\mscr{H}(X)$ for the set of halfspaces bounded by elements of $\mscr{T}_{\mc{U}}(\mf{w})$. By Proposition~\ref{on the preorder}, the set $\mscr{T}_{\mc{U}}(\mf{w})$ is finite and no two of its elements are transverse.

Consider for a moment a subset $\mc{W}\cu\mscr{W}(X)$ and the collection $\wt{\mc{W}}\cu\mscr{H}(X)$ of halfspaces bounded by the elements of $\mc{W}$. Recall from \cite[p.\ 860]{CS} that each such subset $\mc{W}$ determines a quotient $\CAT$ cube complex whose halfspace-pocset is isomorphic to $\wt{\mc{W}}$. As customary, we will refer to this as the \emph{restriction quotient} of $X$ determined by $\mc{W}$.

\begin{defn}
The \emph{dual tree} $T(\mf{w})$ is the restriction quotient determined by the subset $\mscr{T}_{\mc{U}}(\mf{w})\cu\mc{U}\cu\mscr{W}(X)$.
\end{defn}

Note that, since $\mscr{T}_{\mc{U}}(\mf{w})$ is finite and no two of its elements are transverse, the cube complex $T(\mf{w})$ is a finite tree. Being a restriction quotient of $X$, it comes equipped with a projection $\pi_{\mf{w}}\colon\overline X\ra\overline{T(\mf{w})}=T(\mf{w})$. More precisely, $\pi_{\mf{w}}$ takes the point $x\in\overline X$ to the point of $T(\mf{w})$ represented by the ultrafilter $\s_x\cap\wt{\mscr{T}_{\mc{U}}}(\mf{w})$. We will be interested in the pseudo-metric $\delta_{\mf{w}}$ defined on the set $\mc{A}\cu\partial_{\rm nt}X$ by the formula
\[\delta_{\mf{w}}(x,y)=d(\pi_{\mf{w}}(x),\pi_{\mf{w}}(y))=\#(\mscr{T}_{\mc{U}}(\mf{w})\cap\mc{U}(x|y)).\]
The same construction can be applied to hyperplanes $\mf{w}\in\mscr{W}(Y)$, yielding a subset $\mscr{T}_{\mc{V}}(\mf{w})\cu\mc{V}$, a projection to a finite tree $\pi_{\mf{w}}\colon\overline Y\ra T(\mf{w})$ and a pseudo-metric $\delta_{\mf{w}}$ on $\mc{B}\cu\partial_{\rm nt}Y$.

The next result is the main ingredient of Step~(IIb) (cf.\ Theorem~\ref{all WP}). We will employ Definitions~\ref{good defn} and~\ref{trustworthy pairs defn} in its proof. 

\begin{prop}\label{delta preserved}
Let $\mf{u}\in\mc{U}$ and $\mf{v}\in\mc{V}$ be, respectively, $\mc{U}$--boundary-maximal and $\mc{V}$--boundary-maximal, with $f(\partial_{\infty}\mf{u})=\partial_{\infty}\mf{v}$. Given any two points $x,y\in\mc{A}$, we have $\delta_{\mf{v}}(f(x),f(y))=\delta_{\mf{u}}(x,y)$.
\end{prop}
\begin{proof}
By Lemma~\ref{limit set commensurable}, the subgroups $G_{\mf{u}}$ and $G_{\mf{v}}$ are commensurable. Let $H$ be a common finite-index subgroup; in particular, $H$ acts properly and cocompactly on both $\mf{u}$ and $\mf{v}$. Given an infinite-order element $h\in H$, we write $\mc{U}(h)$ and $\mc{V}(h)$ for the subsets of $\mc{U}$ and $\mc{V}$, respectively, corresponding to hyperplanes that contain $o_X(h^{\pm\infty})$ or $o_Y(h^{\pm\infty})$ in their trace at infinity. Note that we have $\mscr{T}_{\mc{U}}(\mf{u})\cu\mc{U}(h)$ and $\mscr{T}_{\mc{V}}(\mf{v})\cu\mc{V}(h)$ by Lemma~\ref{W(k)}.

\smallskip
{\bf Claim.} \emph{There exists an infinite-order element $h_0\in H$ that is good in both $X$ and $Y$, for which the pairs $(o_X(h_0^{-\infty}),o_X(h_0^{+\infty}))$ and $(o_Y(h_0^{-\infty}),o_Y(h_0^{+\infty}))$ are trustworthy, and for which we have $\mscr{T}_{\mc{U}}(\mf{u})\cap\mc{U}(x|y)=\mc{U}(h_0)\cap\mc{U}(x|y)$ and $\mscr{T}_{\mc{V}}(\mf{v})\cap\mc{V}(f(x)|f(y))=\mc{V}(h_0)\cap\mc{V}(f(x)|f(y))$.}

\smallskip
The claim concludes the proof as follows. Lemma~\ref{trust lemma} shows that the $4$--tuples $(h_0^nx,h_0^{-n}x|h_0^ny,h_0^{-n}y)$ are trustworthy for large $n>0$; the same is true of $(h_0^nf(x),h_0^{-n}f(x)|h_0^nf(y),h_0^{-n}f(y))$. Equivariance of $f$ and Lem\-ma~\ref{trust explained} then imply that the set $\mc{U}(h_0^nx,h_0^{-n}x|h_0^ny,h_0^{-n}y)$ has the same cardinality as $\mc{V}(h_0^nf(x),h_0^{-n}f(x)|h_0^nf(y),h_0^{-n}f(y))$ for all large $n>0$. 

Since $x,y\in\partial_{\rm nt}X$, Lemma~\ref{nt vs hyp} shows that $x,y\not\in\partial_{\infty}\mf{u}$. On the other hand $h_0$ preserves $\mf{u}$ and its two fixed points in $\partial_{\infty}X$ lie within $\partial_{\infty}\mf{u}$ by Lemma~\ref{W(k)}. We conclude that $h_0$ does not fix $x,y$ and Proposition~\ref{simplification} finally yields $\#(\mc{U}(h_0)\cap\mc{U}(x|y))=\#(\mc{V}(h_0)\cap\mc{V}(f(x)|f(y)))$. By our choice of $h_0$, these two cardinalities are precisely $\delta_{\mf{u}}(x,y)$ and $\delta_{\mf{v}}(f(x),f(y))$.

\smallskip
\emph{Proof of claim.}
The points $x,y$ are regular by Lemma~\ref{regular vs NT}. It follows that, among hyperplanes of $X$ separating $x$ and $y$, only finitely many are not strongly separated from $\mf{u}$. In light of Proposition~\ref{intersections of traces}, there are only finitely many hyperplanes $\mf{u}'$ that separate $x$ and $y$ and satisfy $\partial_{\infty}\mf{u}'\cap\partial_{\infty}\mf{u}\neq\emptyset$; let us denote by $\mf{u}_1,...,\mf{u}_r$ those that lie in $\mc{U}\setminus\mscr{T}_{\mc{U}}(\mf{u})$. Similarly let $\mf{v}_1,...,\mf{v}_s$ be the elements of $\mc{V}\setminus\mscr{T}_{\mc{V}}(\mf{v})$ that separate $f(x)$ and $f(y)$ and satisfy $\partial_{\infty}\mf{v}_i\cap\partial_{\infty}\mf{v}\neq\emptyset$. 

By boundary-maximality of $\mf{u}$ and $\mf{v}$, we have $\partial_{\infty}\mf{u}_i\cap\partial_{\infty}\mf{u}\subsetneq\partial_{\infty}\mf{u}$ and $\partial_{\infty}\mf{v}_j\cap\partial_{\infty}\mf{v}\subsetneq\partial_{\infty}\mf{v}$ for all $i$ and $j$. Lemma~\ref{nowhere-dense} and Proposition~\ref{intersections of traces} then show that each of these subsets is nowhere-dense. Hence, the union $K\cu\partial_{\infty}H$ of all sets $o_X^{-1}(\partial_{\infty}\mf{u}_i\cap\partial_{\infty}\mf{u})$ and $o_Y^{-1}(\partial_{\infty}\mf{v}_j\cap\partial_{\infty}\mf{v})$ is nowhere-dense. 

Observe that there exists $h\in H$ that acts as a neatly contracting automorphism on both $\mf{u}$ and $\mf{v}$. This follows from Proposition~\ref{sc exist} if $H$ is not virtually cyclic and from Lemma~\ref{ss in quasi-lines} otherwise. Since $K$ is nowhere-dense, there exists a conjugate $h_0$ of $h$ in $H$ such that $h_0^{+\infty}\not\in K$.

Now, we have $o_X(h_0^{+\infty})\not\in\partial_{\infty}\mf{u}_i$ and $o_Y(h_0^{+\infty})\not\in\partial_{\infty}\mf{v}_j$ for all $i$ and $j$, which implies that the inclusions $\mscr{T}_{\mc{U}}(\mf{u})\cap\mc{U}(x|y)\cu\mc{U}(h_0)\cap\mc{U}(x|y)$ and $\mscr{T}_{\mc{V}}(\mf{v})\cap\mc{V}(f(x)|f(y))\cu\mc{V}(h_0)\cap\mc{V}(f(x)|f(y))$ are equalities. 

Since $h_0$ acts as a neatly contracting automorphism on $\mf{u}$, we have $o_X(h_0^{\pm\infty})\in\partial_{\rm nt}\mf{u}$. Lemma~\ref{nt vs hyp} then shows that the pair $(o_X(h_0^{-\infty}),o_X(h_0^{+\infty}))$ is trustworthy. Similarly, we see that $(o_Y(h_0^{-\infty}),o_Y(h_0^{+\infty}))$ is trustworthy. Finally, as $h_0$ has infinite order, Remark~\ref{goods exist} ensures that a power of $h_0$ is good. This concludes the proof of the claim and of the proposition.
\end{proof}

In the setting of Proposition~\ref{delta preserved}, we get the commutative diagram:
\[ \begin{tikzcd}
\mc{A} \arrow[twoheadrightarrow]{r}{\pi_{\mf{u}}} \arrow{d}{f} & \pi_{\mf{u}}(\mc{A}) \arrow{d}{\psi} \arrow[r, hook] & T(\mf{u}) \arrow[dashed]{d}{\Psi} \\
\mc{B} \arrow[twoheadrightarrow]{r}{\pi_{\mf{v}}} & \pi_{\mf{v}}(\mc{B}) \arrow[r, hook] & T(\mf{v}). 
\end{tikzcd}
\]
Here $\pi_{\mf{u}}(\mc{A})$ and $\pi_{\mf{v}}(\mc{B})$ are precisely the quotient metric spaces associated to the pseudo-metric spaces $(\mc{A},\delta_{\mf{u}})$ and $(\mc{B},\delta_{\mf{v}})$. The map $\psi$ is then provided by Proposition~\ref{delta preserved} and it is an isometry.

By part~(1) of Proposition~\ref{ample}, we see that $\pi_{\mf{u}}(\mc{A})$ and $\pi_{\mf{v}}(\mc{B})$ contain all vertices of degree one in $T(\mf{u})$ and $T(\mf{v})$, respectively. The dashed arrow $\Psi$ is then obtained through the following general fact about trees.

\begin{lem}\label{fact on finite trees}
Let $T_1$ and $T_2$ be finite trees with all edges of length one\footnote{When proving Theorem~\ref{ext Moeb intro} for \emph{cuboid} complexes, one should allow edges of arbitrary length in $T_1$ and $T_2$. The proof of the lemma does not change.}. Let $V_i\cu T_i$ be sets of vertices containing all degree-one vertices of $T_i$. Every distance preserving bijection $\psi\colon V_1\ra V_2$ uniquely extends to an isometry $\Psi\colon T_1\ra T_2$. 
\end{lem}
\begin{proof}
Let $V_i'\cu V_i$ denote the subsets of vertices of degree one. If $\#V_i'\leq 2$, the tree $T_i$ is a segment and the lemma is clear. Let us therefore assume that $V_1'$ and $V_2'$ both contain at least three elements. Observe that a point $x\in V_i$ lies outside $V_i'$ precisely when we can find $y,z\in V_i\setminus\{x\}$ with $x=m(x,y,z)$. Since $\psi$ preserves distances, it follows that $\psi(V_1')=V_2'$.

Extending every leaf of $T_i$ to a ray, we embed $T_i$ in a geodesically complete tree $\mc{T}_i$ with a natural bijection $\phi_i\colon V_i'\ra\partial T_i$. Observing that the maps $\phi_i$ preserve cross ratios and $\psi$ is an isometry, we conclude that the bijection $\phi_2\psi\phi_1^{-1}\colon\partial T_1\ra\partial T_2$ preserves cross ratios. It then follows from Theorem~4.3 in \cite{Beyrer-Schroeder} that $\phi_2\psi\phi_1^{-1}$ admits a unique extension to an isometry $\Psi\colon\mc{T}_1\ra \mc{T}_2$.

Note that $m(x,y,z)=m(\phi_i(x),\phi_i(y),\phi_i(z))$ for all pairwise distinct points $x,y,z\in V_i'$. It follows that, given pairwise distinct points $x,y,z\in V_1'$, the map $\Psi$ takes $m_1:=m(x,y,z)$ to $m_2:=m(\psi(x),\psi(y),\psi(z))$. Moreover, $\Psi$ takes the ray from $m_1$ to $\phi_1(x)$ to the ray from $m_2$ to $\phi_2(\psi(x))$. Since 
\begin{align*}
d(m_2,&\Psi(x))=d(m_1,x)=\tfrac{1}{2}\cdot[d(x,y)+d(x,z)-d(y,z)] \\
&=\tfrac{1}{2}\cdot[d(\psi(x),\psi(y))+d(\psi(x),\psi(z))-d(\psi(y),\psi(z))]=d(m_2,\psi(x)),
\end{align*} 
we conclude that $\Psi(x)=\psi(x)$. Thus $\Psi$ and $\psi$ coincide on $V_1'$. For every other point $w\in V_1$, we can find $x,y\in V_1'$ such that $d(x,y)=d(x,w)+d(w,y)$. It is then clear that $\Psi$ and $\psi$ coincide on the entire $V_1$. 
\end{proof}

\begin{cor}\label{key corollary}
Let $\mf{u}$ and $\mf{v}$ be as in Proposition~\ref{delta preserved}. For all ${\mf{h}\in\mscr{H}(X)}$:
\[\#\{\mf{j}\in\wt{\mscr{T}_{\mc{U}}}(\mf{u})\mid\mf{j}\cap\mc{A}=\mf{h}\cap\mc{A}\}=\#\{\mf{m}\in\wt{\mscr{T}_{\mc{V}}}(\mf{v})\mid\mf{m}\cap\mc{B}=f(\mf{h}\cap\mc{A})\}.\]
\end{cor}
\begin{proof}
We begin by observing that it suffices to prove the inequality
\[\tag{$\ast$}\label{inequality} \#\{\mf{j}\in\wt{\mscr{T}_{\mc{U}}}(\mf{u})\mid\mf{j}\cap\mc{A}=\mf{h}\cap\mc{A}\}\leq\#\{\mf{m}\in\wt{\mscr{T}_{\mc{V}}}(\mf{v})\mid\mf{m}\cap\mc{B}=f(\mf{h}\cap\mc{A})\}.\]
This already yields an equality if the right-hand side vanishes. Otherwise, there exists $\mf{m}'\in\wt{\mscr{T}_{\mc{V}}}(\mf{v})$ with $\mf{m}'\cap\mc{B}=f(\mf{h}\cap\mc{A})$ and we can apply the same argument to $f^{-1}$ and $\mf{m}'$ to obtain the opposite inequality.
 
Now, if the left-hand side of~\eqref{inequality} vanishes, there is nothing to prove. Otherwise, there exists a halfspace $\mf{j}'\in\wt{\mscr{T}_{\mc{U}}}(\mf{u})$ with $\mf{j}'\cap\mc{A}=\mf{h}\cap\mc{A}$ and our setup is not affected if we replace $\mf{h}$ with $\mf{j}'$. We can thus assume that $\mf{h}\in\wt{\mscr{T}_{\mc{U}}}(\mf{u})$.

The projection $\pi_{\mf{u}}(\mf{h})$ is a halfspace of the tree $T(\mf{u})$, with complement $\pi_{\mf{u}}(\mf{h}^*)$. Let $C$ and $C^*$ be the convex hulls, respectively, of $\pi_{\mf{u}}(\mf{h}\cap\mc{A})$ and $\pi_{\mf{u}}(\mf{h}^*\cap\mc{A})$. These are disjoint subtrees of $T(\mf{u})$ and the union $C\sqcup C^*$ contains all degree-one vertices of $T(\mf{u})$. The complement of $C\sqcup C^*$ must be an open arc $\alpha\cu T(\mf{u})$ such that every vertex in $\alpha$ has degree two and lies outside $\pi_{\mf{u}}(\mc{A})$. Thus, the hyperplanes associated to the set $\{\mf{j}\in\wt{\mscr{T}_{\mc{U}}}(\mf{u})\mid\mf{j}\cap\mc{A}=\mf{h}\cap\mc{A}\}$ are precisely the elements of $\mscr{T}_{\mc{U}}(\mf{u})$ that are dual to edges of $\alpha$. 

Observing that $\Psi\o\pi_{\mf{u}}=\pi_{\mf{v}}\o f$, we see that the sets $\Psi(C)$ and $\Psi(C^*)$ are the convex hulls, respectively, of $\pi_{\mf{v}}(f(\mf{h}\cap\mc{A}))$ and $\pi_{\mf{v}}(f(\mf{h}^*\cap\mc{A}))$ in $T(\mf{v})$. Moreover, the arc $\Psi(\alpha)$ has the same length of $\alpha$ and its edges correspond to pairwise distinct elements of the set $\{\mf{m}\in\wt{\mscr{T}_{\mc{V}}}(\mf{v})\mid\mf{m}\cap\mc{B}=f(\mf{h}\cap\mc{A})\}$. This yields the desired inequality.
\end{proof}

\subsection{Concluding the proof.}\label{concluding}

In this subsection:

\begin{ass}
Let the group $G$, the cube complexes $X$ and $Y$ and the map $f$ be as in the statement of Theorem~\ref{ext Moeb intro}. We set $\mc{A}=\Om$ and $\mc{B}=f(\Om)$, so that we are in the general setup of Section~\ref{traces vs cross ratios}. Here, however, we do not fix sets $\mc{U}$ and $\mc{V}$. Instead, we observe that, by the hypotheses of Theorem~\ref{ext Moeb intro}, we have $\Cr_X(x,y,z,w)=\Cr_Y(f(x),f(y),f(z),f(w))$ for all pairwise distinct points $x,y,z,w\in\mc{A}$.
\end{ass}

Given subsets $\mc{U}\cu\mscr{W}(X)$ and $\mc{V}\cu\mscr{W}(Y)$, we denote by $\wt{\mc{U}}\cu\mscr{H}(X)$ and $\wt{\mc{V}}\cu\mscr{H}(Y)$ the collections of halfspaces bounded, respectively, by the elements of $\mc{U}$ and $\mc{V}$.

\begin{defn}\label{WP defn}
We say that subsets $\mc{U}\cu\mscr{W}(X)$ and $\mc{V}\cu\mscr{W}(Y)$ are \emph{well-paired} if, for every $\mf{h}\in\mscr{H}(X)$ and $\mf{k}\in\mscr{H}(Y)$, we have:
\[\#\{\mf{j}\in\wt{\mc{U}}\mid\mf{j}\cap\mc{A}=\mf{h}\cap\mc{A}\}=\#\{\mf{m}\in\wt{\mc{V}}\mid\mf{m}\cap\mc{B}=f(\mf{h}\cap\mc{A})\},\]
\[\#\{\mf{m}\in\wt{\mc{V}}\mid\mf{m}\cap\mc{B}=\mf{k}\cap\mc{B}\}=\#\{\mf{j}\in\wt{\mc{U}}\mid\mf{j}\cap\mc{A}=f^{-1}(\mf{k}\cap\mc{B})\}.\] 
In particular, if $\mf{h}\in\wt{\mc{U}}$ and $\mf{k}\in\wt{\mc{V}}$, the right-hand sides must be nonempty.
\end{defn}

The following observation is immediate from definitions.

\begin{lem}\label{WP lemma}
If the sets $\mc{U}\cu\mscr{W}(X)$ and $\mc{V}\cu\mscr{W}(Y)$ are well-paired and the points $x,y,z,w\in\mc{A}$ are pairwise distinct, we have:
\[\#\mc{U}(x,y|z,w)=\#\mc{V}(f(x),f(y)|f(z),f(w)),\] 
\[\Cr_{\mc{U}}(x,y,z,w)=\Cr_{\mc{V}}(f(x),f(y),f(z),f(w)).\]
\end{lem}

The next result provides the correct formulation of Step~(IIb) from the introduction, as traces at infinity need to be counted ``with multiplicity''.

\begin{thm}\label{all WP}
The sets $\mscr{W}(X)$ and $\mscr{W}(Y)$ are well-paired.
\end{thm}
\begin{proof}
We set $\mc{U}_0=\mscr{W}(X)$ and define inductively $\mc{U}_{k+1}\subsetneq\mc{U}_k$ as the subset of hyperplanes that are not $\mc{U}_k$--boundary-maximal. Let $\mc{U}_k^c=\mscr{W}(X)\setminus\mc{U}_k$ and similarly define the sets $\mc{V}_k$ and $\mc{V}_k^c$ starting from $\mc{V}_0=\mscr{W}(Y)$. 

We will now show by induction on $k$ that the sets $\mc{U}_k^c$ and $\mc{V}_k^c$ are well-paired. The base case $k=0$ is trivial, as $\mc{U}_0^c$ and $\mc{V}_0^c$ are empty. Assuming that $\mc{U}_k^c$ and $\mc{V}_k^c$ are well-paired for some $k\geq 0$, Lemma~\ref{WP lemma} yields
\[\Cr_{\mc{U}_k^c}\left(x,y,z,w\right)=\Cr_{\mc{V}_k^c}\left(f(x),f(y),f(z),f(w)\right),\]
whenever $x,y,z,w\in\mc{A}$ are pairwise distinct. On the other hand, observing that $\Cr_X=\Cr_{\mc{U}_k}+\Cr_{\mc{U}_k^c}$ and $\Cr_Y=\Cr_{\mc{V}_k}+\Cr_{\mc{V}_k^c}$, the fact that $f$ takes $\Cr_X$ to $\Cr_Y$ implies that
\[\Cr_{\mc{U}_k}\left(x,y,z,w\right)=\Cr_{\mc{V}_k}\left(f(x),f(y),f(z),f(w)\right).\]
We are now in the setting of Section~\ref{traces vs cross ratios} and can apply Proposition~\ref{maximal to maximal}. It follows that, for each $\mf{u}\in\mc{U}_k\setminus\mc{U}_{k+1}$, there exists $\mf{v}\in\mc{V}_k\setminus\mc{V}_{k+1}$ with $\partial_{\infty}\mf{v}=f(\partial_{\infty}\mf{u})$. Applying Corollary~\ref{key corollary} to both $f$ and $f^{-1}$, we see that the sets $\mscr{T}_{\mc{U}_k}(\mf{u})$ and $\mscr{T}_{\mc{V}_k}(\mf{v})$ are well-paired. Letting $\mf{u}$ vary, these sets partition $\mc{U}_k\setminus\mc{U}_{k+1}$ and $\mc{V}_k\setminus\mc{V}_{k+1}$; we conclude that the latter are also well-paired. Observing that $\mc{U}_{k+1}^c=\mc{U}_k^c\sqcup(\mc{U}_k\setminus\mc{U}_{k+1})$ and $\mc{V}_{k+1}^c=\mc{V}_k^c\sqcup(\mc{V}_k\setminus\mc{V}_{k+1})$, we have finally shown that the sets $\mc{U}_{k+1}^c$ and $\mc{V}_{k+1}^c$ are well-paired, completing the proof of the inductive step.

Part~(3) of Proposition~\ref{on the preorder} shows that, for sufficiently large values of $k$, the sets $\mc{U}_k$ and $\mc{V}_k$ are reduced to the subsets $\overline{\mc{U}}\cu\mscr{W}(X)$ and $\overline{\mc{V}}\cu\mscr{W}(Y)$ of hyperplanes with empty trace at infinity. We conclude the proof by showing that $\overline{\mc{U}}$ and $\overline{\mc{V}}$ are well-paired. Note that the arguments above already yield 
\[\Cr_{\overline{\mc{U}}}\left(x,y,z,w\right)=\Cr_{\overline{\mc{V}}}\left(f(x),f(y),f(z),f(w)\right)\]
for all pairwise distinct points $x,y,z,w\in\mc{A}$. By Remark~\ref{trivial boundary}, no two elements of $\overline{\mc{U}}$ or $\overline{\mc{V}}$ are transverse; in particular, Lemma~\ref{three pwt} shows that every element of $\mc{A}^4$ is $\overline{\mc{U}}$--trustworthy and every element of $\mc{B}^4$ is $\overline{\mc{V}}$--trustworthy. Lem\-ma~\ref{trust explained} guarantees that
\[\#\overline{\mc{U}}\left(x,y|z,w\right)=\#\overline{\mc{V}}\left(f(x),f(y)|f(z),f(w)\right).\] 
Thus, if $\mf{h}\in\mscr{H}(X)$ is bounded by an element of $\overline{\mc{U}}$, the set $\overline{\mc{V}}(x',y'|z',w')$ is nonempty for all $x',y'\in f(\mf{h}\cap\mc{A})$ and $z',w'\in f(\mf{h}^*\cap\mc{A})$. Lemma~\ref{all trustworthy} then provides $\overline x,\overline y\in f(\mf{h}\cap\mc{A})$ and $\overline z,\overline w\in f(\mf{h}^*\cap\mc{A})$ with
\begin{align*}
&\#\overline{\mc{V}}\big(f(\mf{h}\cap\mc{A})\big|f(\mf{h}^*\cap\mc{A})\big)=\#\overline{\mc{V}}\big(\overline x,\overline y\big|\overline z,\overline w\big) \\ 
=&\#\overline{\mc{U}}\big(f^{-1}(\overline x),f^{-1}(\overline y)\big|f^{-1}(\overline z),f^{-1}(\overline w)\big)\geq\#\overline{\mc{U}}\big(\mf{h}\cap\mc{A}\big|\mf{h}^*\cap\mc{A}\big)>0.
\end{align*}
The opposite inequality is obtained with a similar argument and we conclude that $\#\overline{\mc{U}}\left(\mf{h}\cap\mc{A}|\mf{h}^*\cap\mc{A}\right)=\#\overline{\mc{V}}\left(f(\mf{h}\cap\mc{A})|f(\mf{h}^*\cap\mc{A})\right)$. This shows that $\overline{\mc{U}}$ and $\overline{\mc{V}}$ are well-paired, thus completing the proof of the theorem.
\end{proof}

\begin{proof}[Proof of Theorem~\ref{ext Moeb intro}]
We are going to show that $f\colon\partial_{\infty}X\ra\partial_{\infty}Y$ induces a $G$--equivariant pocset isomorphism $f_*\colon(\mscr{H}(X),\cu,\ast)\ra(\mscr{H}(Y),\cu,\ast)$. This will then yield a $G$--equivariant cubical isomorphism $F\colon X\ra Y$ by Roller duality; see e.g.\ \cite{Nica,Chatterji-Niblo,Sageev-notes}. Uniqueness and the fact that $F$ and $f$ coincide on $\Omega$ will be clear from the construction.

We start by observing that the set $\mc{H}(\mf{h}):=\{\mf{j}\in\mscr{H}(X)\mid\mf{j}\cap\mc{A}=\mf{h}\cap\mc{A}\}$ is totally ordered by inclusion for each $\mf{h}\in\mscr{H}(X)$. Indeed, given $\mf{j}_1,\mf{j}_2\in\mc{H}(\mf{h})$, the intersections $\mf{j}_1\cap\mf{j}_2\supseteq\mf{h}\cap\mc{A}$ and $\mf{j}_1^*\cap\mf{j}_2^*\supseteq\mf{h}^*\cap\mc{A}$ are nonempty by part~(1) of Proposition~\ref{ample}. Moreover, $\mf{j}_1$ and $\mf{j}_2$ cannot be transverse, or part~(2) of Proposition~\ref{ample} would yield $\mf{h}\cap\mf{h}^*\supseteq\mf{j}_1\cap\mf{j}_2^*\cap\mc{A}\neq\emptyset$. Hence $\mf{j}_1\cu\mf{j}_2$ or $\mf{j}_2\cu\mf{j}_1$.

Note that $\mc{H}(\mf{h})$ is finite as, by part~(2) of Lemma~\ref{bounded components of cu}, the set $\mscr{W}(x,y|z,w)$ is finite for all distinct points $x,y\in\mf{h}\cap\mc{A}$, $z,w\in\mf{h}^*\cap\mc{A}$. By Theorem~\ref{all WP}, the set $f_*\mc{H}(\mf{h}):=\{\mf{m}\in\mscr{H}(Y)\mid\mf{m}\cap\mc{B}=f(\mf{h}\cap\mc{A})\}$ is a chain of the same length as $\mc{H}(\mf{h})$. Thus, there exists a unique order-preserving bijection between $\mc{H}(\mf{h})$ and $f_*\mc{H}(\mf{h})$ and this is exactly how we define $f_*$ on $\mc{H}(\mf{h})$.

Since the sets $\mc{H}(\mf{h})$ partition $\mscr{H}(X)$, we have actually defined a map $f_*\colon\mscr{H}(X)\ra\mscr{H}(Y)$. This is a bijection, an inverse being provided by the same construction applied to $f^{-1}$. It is also clear that $f_*(\mf{h}^*)=f_*(\mf{h})^*$, as $\mc{H}(\mf{h}^*)$ is exactly the set of complements of the elements of $\mc{H}(\mf{h})$. We are thus only left to show that $f_*$ preserves inclusions.

Consider $\mf{h},\mf{k}\in\mscr{H}(X)$ with $\mf{h}\cu\mf{k}$. We can assume that $\mf{k}\not\in\mc{H}(\mf{h})$, as we already know that $f$ is order-preserving on $\mc{H}(\mf{h})$. Hence $\mf{h}\cap\mc{A}\subsetneq\mf{k}\cap\mc{A}$ and, by construction, $f_*(\mf{h})\cap\mc{B}=f(\mf{h}\cap\mc{A})\subsetneq f(\mf{k}\cap\mc{A})=f_*(\mf{k})\cap\mc{B}$. Part~(3) of Proposition~\ref{ample} finally yields $f_*(\mf{h})\subsetneq f_*(\mf{k})$, concluding the proof.
\end{proof}

\begin{rmk}
When dealing with \emph{cuboid} complexes $\mbb{X}$ and $\mbb{Y}$, the proof of Theorem~\ref{ext Moeb intro} needs to be slightly adapted. In general, the sets $\mc{H}(\mf{h})$ and $f_*\mc{H}(\mf{h})$ will have the same \emph{weight}, but not the same \emph{cardinality}. This prevents us from defining an isomorphism $f_*$ between the halfspace pocsets of $\mbb{X},\mbb{Y}$.

One should instead observe that $\CAT$ cuboid complexes are \emph{median spaces} and, thus, naturally endowed with a structure of \emph{space with measured walls} \cite{CDH}. The map $f\colon\partial_{\infty}\mbb{X}\ra\partial_{\infty}\mbb{Y}$ then induces a $G$--equivariant isomorphism of their \emph{measured halfspace pocsets} (see Sections~2.2 and~3.1 in \cite{Fio1}). In this context, an analogue of Roller duality is provided by Corollaries~3.12 and~3.13 in \cite{Fio1} and we obtain a $G$--equivariant isometry $F\colon\mbb{X}\ra\mbb{Y}$. In general, $F$ will not take vertices of $\mbb{X}$ to vertices of $\mbb{Y}$.
\end{rmk}

\section{Epilogue.}\label{main section}

In this section we apply Theorem~\ref{ext Moeb intro} to obtain Theorem~\ref{hyp CR} and Corollary~\ref{hyp MLSR}. Relying on \cite{BF1}, we also prove Corollary~\ref{non-hyp CR}.

\subsection{Cross ratios on contracting boundaries.}\label{epilogue 1}

By Remark~\ref{hyp vs irr}, the following are common hypotheses to Theorem~\ref{hyp CR} and Corollary~\ref{non-hyp CR}.

\begin{ass}
Let $G$ be a non-virtually-cyclic group. We consider proper cocompact actions of $G$ on irreducible, essential $\CAT$ cube complexes $X$ and $Y$. 
\end{ass}

In Section~\ref{cnt sect}, we defined a $G$--equivariant bijection $o_X\colon\partial_cG\ra\partial_cX$ arising from orbit maps. We now want to exploit the map $\Phi$ introduced in Section~\ref{d_c vs d_R} to transfer the cross ratio on $\partial X$ to a $G$--invariant cross ratio on $\partial_cX$ and $\partial_cG$. To this end, we will need the following notion. 

\begin{defn}
A subset $\mc{A}\cu\partial_{\rm cu}X$ is a \emph{section} (of the map $\Phi$) if $\mc{A}$ intersects each fibre of $\Phi$ at exactly one point. In particular, $\partial_{\rm cnt}X\cu\mc{A}$.
\end{defn}

\begin{rmk}\label{finite Gromov product}
If $\mc{A}$ is a section of $\Phi$ and $x,y\in\mc{A}$ are distinct points, Lemma~\ref{bounded components of cu} shows that $x$ and $y$ have finite Gromov product. In particular, if ${x,y,z,w\in\mc{A}}$ are pairwise distinct, the median $m(x,y,z)$ lies in $X$ and the cross ratio $\Cr(x,y,z,w)$ is well-defined and finite (cf.\ Lemma~\ref{infinite Gromov product}).
\end{rmk}

Note that it is always possible to find a $G$--invariant section $\mc{A}\cu\partial_{\rm cu}X$, up to subdividing $X$. Indeed, since every component of $\partial_{\rm cu}X$ is finite by Remark~\ref{components are finite}, we can consider the first cubical subdivision $X'$ and pick the median barycentre (cf.\ Section~\ref{median barycentres}) of each component of $\partial_{\rm cu}X'$. The resulting subset $\mc{A}\cu\partial_{\rm cu}X'$ is a $G$--invariant section of $\Phi'\colon\partial_{\rm cu}X'\ra\partial_cX'\simeq\partial_cX$.

We stress that we assign length $1$ to every edge of $X'$. In particular, the inclusion $X\hookrightarrow X'$ is a homothety doubling distances.

Restricting the cross ratio on $\partial X'$ to the set $\mc{A}$ and identifying $\mc{A}\simeq\partial_cG$ via the composition $o_X^{-1}\o\Phi'$, we obtain an invariant cross ratio 
\[\Cr_X\colon\partial_cG^{(4)}\longrightarrow\Z.\]
By Remark~\ref{finite Gromov product}, the value $\Cr_X(x,y,z,w)$ is well-defined and finite as soon as $x,y,z,w\in\partial_cG$ are pairwise distinct. We refer to $\Cr_X$ as the \emph{cubical cross ratio} on $\partial_cG$ associated to the cubulation $G\acts X$.

We can extend $\Cr_X$ to $4$--tuples $(x,y,z,w)$ with $\#\{x,y,z,w\}\leq 3$ as long as no three of the four points coincide. If $x\neq y$ and $x\neq z$, we set $\Cr_X(x,x,y,z)=0$ and $\Cr_X(x,y,x,z)=+\infty$, while the other values of $\Cr_X$ can be recovered using its symmetries.

We can endow $\partial_cG$ with the pull-back of the topology of $\partial_c^{\rm vis}X$, or with one of the topologies of \cite{Cordes,Cashen-Mackay}, but $\Cr_X$ will rarely be continuous on its entire domain. Indeed, $\Cr_X$ takes integer values, whereas $\partial_cG$ can be connected (for instance when $G$ is a one-ended hyperbolic group). 

Nevertheless, we have the following:

\begin{prop}\label{continuity of cr_X}
Endowing $\partial_cG$ with any of the three topologies above, $\Cr_X$ is continuous at every $4$--tuple with coordinates in $\partial_{\rm cnt}X\cu\partial_cX$.
\end{prop}
\begin{proof} 
It suffices to consider $\partial_cG$ endowed with the pull-back of the topology of $\partial_c^{\rm vis}X$, as this is the coarsest of the three (see Section~3.1 in \cite{Charney-Sultan} and Section~7 in \cite{Cashen-Mackay}). In this case, the result follows from Proposition~\ref{cr continuous} and Remark~\ref{topologies on cnt 2}. 
\end{proof}

\begin{proof}[Proof of Theorem~\ref{hyp CR}]
Since $X$ is essential and $G$ acts cocompactly, every hyperplane-stabiliser has infinite index in $G$. Lemmas~\ref{nowhere-dense} and~\ref{hyp cocpt} thus imply that the union of the Gromov boundaries of the hyperplanes of $X$ is meagre in $\partial_{\infty}X$. By Lemma~\ref{nt vs hyp}, this means that $\mc{C}:=\partial_{\rm nt}X$ is co-meagre in $\partial_{\infty}X$. Part~(1) now follows from Proposition~\ref{continuity of cr_X}. 

Regarding part~(2), suppose in addition that $X$ is hyperplane-essential. Let ${G\acts Y}$ be another proper cocompact action on an essential, hyperplane-essential $\CAT$ cube complex and consider the $G$--equi\-var\-i\-ant homeomorphism $f=o_Y\o o_X^{-1}\colon\partial_{\infty}X\ra\partial_{\infty}Y$. Suppose that $\Cr_X$ and $\Cr_Y$ coincide on $\mc{D}^{(4)}\cu\partial_{\infty}G^{(4)}$, for a co-meagre subset $\mc{D}\cu\partial_{\infty}G$. The set $\mc{D}\cap\partial_{\rm nt}X\cap f^{-1}(\partial_{\rm nt}Y)$ is co-meagre and so is the intersection of all its $G$--translates, which we denote by $\Om$. By Baire's theorem, $\Om$ is nonempty and we conclude by applying Theorem~\ref{ext Moeb intro}.
\end{proof}

\begin{proof}[Proof of Corollary~\ref{non-hyp CR}]
Part~(1) is immediate from the previous discussion, as cube complexes with no free faces are essential. 

Let now $G\acts X$ and $G\acts Y$ be two cubulations satisfying the hypotheses of the theorem and inducing the same cubical cross ratio on $\partial_cG$. Proposition~\ref{sc exist} yields an element $g\in G$ that acts as a neatly contracting automorphism on both $X$ and $Y$. The $G$--equivariant bijection $f=o_Y\o o_X^{-1}\colon\partial_cX\ra\partial_cY$ takes $g^+_X\in\partial_{\rm cnt}X$ to $g^+_Y\in\partial_{\rm cnt}Y$ by Lemma~\ref{north south}. Setting $\mc{A}=G\cdot g^+_X\cu\partial_{\rm cnt}X$, we have $f(\mc{A})=G\cdot g^+_Y\cu\partial_{\rm cnt}Y$ and cross ratios of these points are preserved. The sets $\mc{A}\cu\partial_{\rm cnt}X$ and $f(\mc{A})\cu\partial_{\rm cnt}Y$ consist of regular points by Lemma~\ref{regular vs NT}. We conclude by Theorem~E in \cite{BF1}, observing that the actions of $G$ on $X$ and $Y$ are non-elementary (in the sense of \emph{op.\ cit.}) by Lemma~2.9 in \emph{op.\ cit.}.
\end{proof}

\subsection{Marked length-spectrum rigidity.}

\begin{proof}[Proof of Corollary~\ref{hyp MLSR}]
If $G$ is non-elementary, Theorem~D in \cite{BF1} provides a $G$--equivariant, cross-ratio preserving bijection $f\colon\mc{A}\ra\mc{B}$, where ${\mc{A}\cu\partial_{\rm reg}X}$ and $\mc{B}\cu\partial_{\rm reg}Y$ are nonempty $G$--invariant subsets. The reader will not have trouble realising that this map is a restriction of the unique $G$--equivariant homeomorphism $f\colon\partial_{\infty}X\ra\partial_{\infty}Y$ (see Section~4.2 in \cite{BF1} for details). Regular points are non-terminating and we conclude by Theorem~\ref{ext Moeb intro}.

We are left to consider the case when $G$ is virtually cyclic. If $G$ is finite, $X$ and $Y$ must be single points, by essentiality; so let us assume that $G$ is virtually isomorphic to $\Z$. By part~(1) of Lemma~\ref{ss in quasi-lines}, every hyperplane of $X$ is compact and, since now $X$ is hyperplane-essential, we must have $X\simeq\R$. The action $G\acts X$ factors through a faithful action of either $\Z$ or $D_{\infty}$. In the former case, the only $g\in G$ with $\ell_X(g)=0$ are those in the (finite) kernel of the action $G\acts X$. In the latter case, we have infinitely many $g\in G$ with $\ell_X(g)=0$, for instance all reflections. 

Since $\ell_X=\ell_Y$, the actions $G\acts X$ and $G\acts Y$ either both factor through a faithful action of $\Z$ or both factor through a faithful action of $D_{\infty}$. In the former case, the two actions must coincide, as both the kernel and the $\Z$--action can be described in terms of length functions. In the latter, the two actions are $G$--equivariantly isomorphic, since actions $D_{\infty}\acts\R$ are determined, up to conjugacy, by the restriction to the maximal $\Z$ subgroup.
\end{proof}

As the next two examples demonstrate, there is no way of removing the essentiality and hyperplane-essentiality requirements from Theorem~\ref{hyp CR} and Corollary~\ref{hyp MLSR}.

\begin{figure}
\centering
\includegraphics[width=3in]{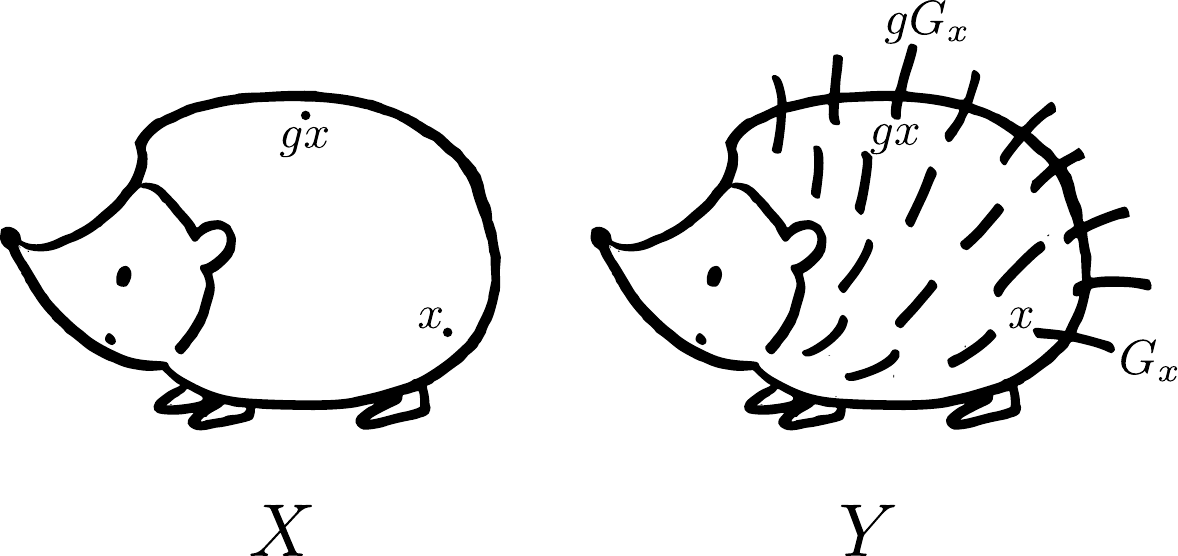}
\caption{Giving a hedgehog back its spines.}
\label{hedgehog}
\end{figure}

\begin{ex}\label{need essential}
Let $G\acts X$ be any proper cocompact action on a $\CAT$ cube complex. Fix a basepoint $x\in X$ and let $G_x\leq G$ denote its stabiliser. The disjoint union $Y=X\sqcup G/G_x$ is endowed with a natural $G$--action and we give it a structure of $\CAT$ cube complex by adding edges connecting $gG_x$ and $gx$ for every $g\in G$. The procedure is depicted in Figure~\ref{hedgehog} in a more general context. The $\CAT$ cube complex $Y$ is irreducible and the action $G\acts Y$ is proper and cocompact. Note that $Y$ is hyperplane-essential if and only if $X$ is, but $Y$ is never going to be essential, as all points of $G/G_x$ are vertices of degree $1$. It is easy to see that $\ell_X=\ell_Y$ and $\Cr_X=\Cr_Y$.
\end{ex}

We now describe a general procedure that takes any finite dimensional $\CAT$ cube complex $X$ as input and gives out another $\CAT$ cube complex $S(X)$ as output. We will refer to $S(X)$ as the \emph{squarisation} of $X$. One can already get a good idea of the definition by looking at Figure~\ref{tree of squares}.

Let $\mscr{H}'$ denote the disjoint union of two copies of the pocset $(\mscr{H}(X),\cu,\ast)$, labelling by $\mf{h}_1$ and $\mf{h}_2$ the two elements arising from $\mf{h}\in\mscr{H}(X)$. We turn $\mscr{H}'$ into a pocset by declaring that, given $\mf{h},\mf{k}\in\mscr{H}(X)$ with $\mf{h}\not\in\{\mf{k},\mf{k}^*\}$, we have $\mf{h}_i\cu\mf{k}_j$ if and only if $\mf{h}\cu\mf{k}$, no matter what the indices $i$ and $j$ are. On the other hand, the halfspaces $\mf{h}_1$ and $\mf{h}_2$ are transverse for all $\mf{h}\in\mscr{H}(X)$. 

Now, $S(X)$ is obtained by applying Sageev's construction \cite{Sageev,Sageev-notes} to the pocset $(\mscr{H}',\cu,\ast)$. Note that, if we had instead declared that $\mf{h}_1$ and $\mf{h}_2$ are nested, we would have obtained the first cubical subdivision $X'$.

\begin{figure}
\centering
\includegraphics[width=4in]{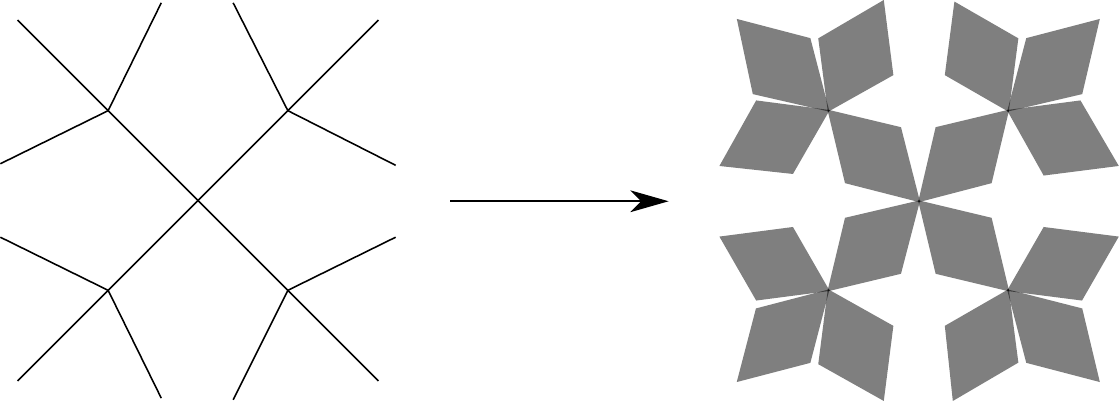}
\caption{Part of a $4$--regular tree $T$ and its squarisation $S(T)$.}
\label{tree of squares}
\end{figure}

\begin{ex}\label{need hyperplane-essential}
Let $G\acts X$ be any proper cocompact action on an essential, irreducible $\CAT$ cube complex. The cube complexes $X'$ and $S(X)$ are both essential, irreducible and naturally endowed with proper cocompact actions of $G$. It is easy to see that the actions $G\acts X'$ and $G\acts S(X)$ determine the same length function, namely the double of the length function associated to $G\acts X$. Moreover, the (co-meagre) subset of $\partial_{\infty}G$ arising from non-terminating ultrafilters is the same for $X$, $X'$ and $S(X)$ and there we have $\Cr_{S(X)}=\Cr_{X'}=2\cdot\Cr_X$. 

The failure of Theorem~\ref{hyp CR} and Corollary~\ref{hyp MLSR} is to be traced back to hy\-per\-plane-essentiality. Indeed, $S(X)$ is never hyperplane-essential. All its hyperplanes split as $S(\mf{w})\x[0,1]$, where $\mf{w}$ is the corresponding hyperplane of $X$. 
\end{ex}

\bibliography{mybib}
\bibliographystyle{alpha}

\end{document}